\documentclass[10pt,a4paper,french]{smfart}

\usepackage[latin1]{inputenc}
\usepackage[T1]{fontenc}
\usepackage[french,english]{babel}
\usepackage{amsmath} 
\usepackage{amssymb}
\usepackage[all,2cell,emtex]{xy}
\input xypic

\usepackage{xpatch}
\usepackage[hidelinks,pdfencoding=unicode]{hyperref}
\usepackage{lmodern}

\SilentMatrices

\SelectTips{cm}{10}

\makeatletter

\patchcmd\@begintheorem
  {\ignorespaces}
  {\hypertarget{\csname @currentHref\endcsname}{}\ignorespaces}
  {}{}

\makeatother

\theoremstyle{plain}
\newtheorem{thm"}{Th\'eor\`eme}[]
\newtheorem{thm}{Th\'eor\`eme}[section]
\newtheorem{prop}[thm]{Proposition}
\newtheorem{lemme}[thm]{Lemme}
\newtheorem{cor}[thm]{Corollaire}
\newtheorem{conj}[thm]{Conjecture}

\theoremstyle{remark}
\newtheorem{rem}[thm]{Remarque}

\newtheorem{sch}[thm]{Scholie}

\theoremstyle{definition}

\newtheorem{paragr}[thm]{}

\newcommand\pdfoo{\texorpdfstring{$\infty$}{\unichar{"221E}}}

\newcommand{\noopsort}[1]{\relax}

\newcommand{\W}{\mathcal W}
\newcommand{\Wfolk}{\W_{\mathrm{folk}}}
\newcommand{\Wpol}{\W_{\mathrm{pol}}}
\newcommand{\Wthom}{\W_{\mathrm{Thom}}}
\newcommand{\Wsmp}{\W_{\pref{\Catsmp}}}

\newcommand{\Fl}{\mathrm{Fl}}

\let\nbd\nobreakdash
\let\e\varepsilon
\let\Eta\eta
\let\Eps\varepsilon
\renewcommand{\leq}{\leqslant}
\renewcommand{\geq}{\geqslant}

\newcommand{\SLF}{système local faible}
\newcommand{\SLFS}{systèmes locaux faibles}
\newcommand{\SL}{système local}
\newcommand{\SLS}{systèmes locaux}

\newcommand{\clh}{limite inductive homotopique}
\newcommand{\clhs}{limites inductives homotopiques}

\newcommand{\Ab}{\mathrm{Ab}}
\newcommand{\AB}{L}
\newcommand{\Ens}{\mathrm{Ens}}
\newcommand{\Ho}{\mathrm{Ho}}
\newcommand{\C}{\mathrm{C}}
\newcommand{\LL}{\mathrm{L}}
\newcommand{\N}{\mathrm{N}}
\newcommand{\Tor}{\mathrm{Tor}}
\newcommand{\nd}{\mathrm{nd}}
\renewcommand{\H}{\mathrm{H}}
\newcommand{\Hp}[1]{\H^{\mathrm{pol}\vrule depth 2.3pt width 0pt}_{#1}}
\newcommand{\HP}{\Hp{\scriptscriptstyle \bullet}}
\newcommand{\HT}{\H_{\scriptscriptstyle \bullet}}
\newcommand{\F}{F}
\newcommand{\G}{G}
\newcommand{\E}{E}
\newcommand{\Hom}{\mathrm{Hom}}
\newcommand{\Homint}{\operatorname{\underline{\mathrm{Hom}}}}
\newcommand{\Comp}{\mathrm{Comp}}
\newcommand{\Der}{\mathrm{Der}}
\newcommand{\Cat}{\mathrm{Cat}}
\newcommand{\ooCat}{\infty\hbox{-}\mathrm{Cat}}
\newcommand{\ooCCat}{\infty\hbox{-}C\hbox{-}\mathrm{Cat}}
\newcommand{\Catsmp}{{\mathbf \Delta}}
\newcommand{\smp}[1]{\Delta_{#1}}
\newcommand{\Or}[1]{\mathcal{O}_{#1}}
\newcommand{\Int}{\textstyle{\int}}
\newcommand{\tauint}[1]{\tau^{}_{#1}}

\newcommand{\rev}[1]{\widetilde{#1}}
\newcommand{\pref}[1]{\widehat{#1}}
\newcommand{\Yon}[1]{\widetilde{#1}}

\newcommand{\transf}[2]{{#1}^{}_{#2}}
\newcommand{\comp}[1]{*^{}_{#1}}

\newcommand{\cm}[2]{\mathchoice {#1\raise -1.8pt\vbox{\hbox{$\kern -.8pt/#2$}}} {#1\raise -1.8pt\vbox{\hbox{$\kern -.8pt/#2$}}} {#1\raise -1.8pt\vbox{\hbox{$\scriptstyle\kern -.8pt /#2$}}} {#1\raise -1.8pt\vbox{\hbox{$\scriptscriptstyle\kern -.8pt /#2$}}} }

\newcommand{\mc}[2]{\mathchoice {\raise -1.8pt\vbox{\hbox{$#1\backslash$}}#2} {\raise -1.8pt\vbox{\hbox{$#1\backslash$}}#2} {\raise -1.8pt\vbox{\hbox{$\scriptstyle#1\backslash$}}#2} {\raise -1.8pt\vbox{\hbox{$\scriptscriptstyle#1\backslash$}}#2} }

\def\hocolim{\mathop{\oalign{\rm holim\cr
\hidewidth$\mathrel{\hbox to 26.5pt{\rightarrowfill}}$\hidewidth\cr}}}%

\def\shocolim{\mathop{\oalign{\scriptsize\rm holim\cr
\hidewidth$\mathrel{\hbox to 21.5pt{\rightarrowfill}}$\hidewidth\cr}}}%

\author{Léonard GUETTA}
\address{Universiteit Utrecht, Nederland}
\email{l.s.guetta@uu.nl}
\urladdr{\href{https://leoguetta.github.io/}{https://leoguetta.github.io/}}

\author{Georges MALTSINIOTIS}
\address{CNRS IMJ-PRG, Universit\'e Paris Cit\'e, France}
\email{georges.maltsiniotis@imj-prg.fr}
\urladdr{\href{https://webusers.imj-prg.fr/~georges.maltsiniotis/}{https://webusers.imj-prg.fr/\raise -3.3pt\vbox{\hbox{$\widetilde{ \ }\,$}}georges.maltsiniotis/}}
\title{Homologie polygraphique des systèmes locaux}

\begin{document}

\frontmatter

\begin{abstract} 
Dans cet article, on introduit une notion d'homologie polygraphique d'une $\infty$\nbd-catégorie stricte à coefficients dans un \SL, généralisant l'homologie polygraphique à coefficients dans $\mathbb Z$, introduite par François Métayer. On montre que l'homologie d'un ensemble simplicial à coefficients dans un \SL{} coïncide avec l'homologie polygraphique de son image par l'adjoint à gauche du nerf de Street à coefficients dans le \SL{} correspondant. On définit dans ce cadre un morphisme de comparaison entre l'homologie polygraphique d'une $\infty$\nbd-catégorie stricte et l'homologie de son nerf de Street, et on montre que ce morphisme est un isomorphisme pour les (1\nbd-)catégories. Il n'en est pas de même pour une $\infty$\nbd-catégorie stricte arbitraire. Néanmoins, on conjecture que pour une construction analogue dans le cadre des $\infty$\nbd-catégories faibles\penalty-500{} \og à la Grothendieck\fg{} on obtiendrait toujours un isomorphisme.
\end{abstract}

\begin{altabstract} 
In this article, we introduce a notion of polygraphic homology of a strict $\infty$\nbd-category with coefficients in a local system, generalizing the polygraphic homology with coefficients in $\mathbb Z$, introduced by François Métayer. We show that the homology of a simplicial set with coefficients in a local system coincides with the polygraphic homology of its image by the left adjoint of the Street nerve with coefficients in the corresponding local system. We define in this framework a comparison morphism between the polygraphic homology of a strict $\infty$\nbd-category and the homology of its Street nerve, and we show that this morphism is an isomorphism for (1\nbd-)categories. This is not true for an arbitrary strict $\infty$\nbd-category. Nevertheless, we conjecture that for an analogous construction in the framework of weak $\infty$\nbd-categories\penalty-500{} \og à la Grothendieck\fg{} we would always obtain an isomorphism.
\end{altabstract}


\dedicatory{À François Métayer pour son départ à la retraite}

\maketitle
\tableofcontents

\mainmatter

\section*{Introduction}

Dans \emph{Pursuing Stacks}, Grothendieck suggère que l'homologie d'un $\infty$\nbd-groupoïde faible (\og$\infty$\nbd-stack\fg{} dans sa terminologie) se calcule par le complexe dont la composante de degré $n$ est le groupe abélien libre engendré par ses $n$\nbd-cellules et dont la différentielle est définie en prenant la différence entre le but et la source d'une $n$\nbd-cellule~\cite[note~97,~page~29]{PS1}. 
Cette idée est un peu naïve (elle ne donne même pas l'homologie correcte pour le $\infty$\nbd-groupoïde ponctuel) mais comme souvent avec Grothendieck, l'intuition sous-jacente semble correcte, sauf qu'il faudrait prendre un foncteur d'abélianisation plus subtil, et remplacer le $\infty$\nbd-groupoïde considéré par une résolution cofibrante dans un sens adéquat.
\smallbreak

Indépendamment, François Métayer a introduit l'homologie polygraphique d'une $\infty$\nbd-catégorie stricte (qu'on appellera ici simplement $\infty$\nbd-catégorie) en termes de résolutions polygraphiques, dans un travail qu'il précise avoir effectué en collaboration avec Albert Burroni~\cite{FrMet}. La notion de \emph{polygraphe}~\cite{Burr1},~\cite{Burr2} (ou \emph{computad}~\cite{StrPol},~\cite{Or}) est le concept pertinent de $\infty$\nbd-catégorie libre et une \emph{résolution polygraphique} est l'analogue $\infty$\nbd-catégorique d'une résolution libre d'un groupe abélien. 
\smallbreak

On rappelle que les objets $\infty$\nbd-catégorie (stricte) de la catégorie $\Ab$ des groupes abéliens forment une catégorie équivalente à celle des complexes de chaînes de groupes abéliens en degrés positifs~\cite{Bourn}. Le foncteur d'oubli $\ooCat(\Ab)\to\ooCat$, de la catégorie des objets $\infty$\nbd-catégorie stricte de $\Ab$ vers celle des $\infty$\nbd-catégories strictes, admet un adjoint à gauche $\AB$, le \emph{foncteur d'abélianisation}, qui s'identifie donc à un foncteur de $\ooCat$ vers la catégorie $\Comp(\Ab)$ des complexes de groupes abéliens. Si~$X$ est une $\infty$\nbd-catégorie et $n\geq0$, le groupe $\AB(X)_n$ est le \emph{quotient} du groupe abélien libre engendré par les $n$-flèches de $X$ par les relations
$x*_iy=x+y$, pour $0\leq i<n$ et $x,y$ $n$\nbd-flèches $i$\nbd-composables de $X$, où $x*_iy$ désigne le composé de $y$ suivi de $x$, au-dessus de la $i$\nbd-source de $x$, supposée égale au $i$\nbd-but de $y$. Pour $n>0$, la différentielle de l'image dans $\AB(X)_n$ d'une $n$\nbd-cellule est définie par la différence de l'image dans $\AB(X)_{n-1}$ de son but moins celle de sa source. L'\emph{homologie polygraphique} de $X$ est définie en choisissant une résolution polygraphique $P\to X$ de $X$, en posant pour $n\geq0$, $\Hp{n}(X)=\H_{n}(\AB(P))$, et en montrant que cette définition est indépendante du choix de la résolution~\cite{FrMet}.
\smallbreak

L'homologie polygraphique s'interprète plus conceptuellement en termes de la structure de catégorie de modèles \og folk\fg{} sur $\ooCat$, introduite dans~\cite{LMW}. La notion d'équivalence faible pour cette structure généralise celle d'équivalence de catégories, et les objets cofibrants sont les polygraphes~\cite{MetCof}. L'homologie polygraphique s'identifie au foncteur dérivé à gauche du foncteur d'abélianisation. Yves Lafont et François Métayer ont montré que l'homologie polygraphique d'un monoïde (vu comme une $\infty$\nbd-catégorie avec un seul objet et des $i$\nbd-cellules triviales pour $i>1$) coïncide avec son homologie usuelle~\cite{LafMet}. Le premier des auteurs du présent article a montré qu'il en est de même pour toute (1\nbd-)catégorie~\cite{Leonard},~\cite{LeonardTh}. En revanche, pour une $\infty$\nbd-catégorie stricte générale $X$, son homologie polygraphique ne coïncide pas toujours avec son homologie, définie comme étant celle de son nerf de Street~\cite{Or}. Par exemple, si $X$ est la $2$\nbd-catégorie ayant un seul objet, l'unité de cet objet comme seule $1$\nbd-flèche et $\mathbb N$ comme monoïde d'endomorphismes de cette unité, l'homologie polygraphique de $X$ est $\mathbb Z$ en degrés $0$ et $2$, et $0$ ailleurs, tandis que son homologie est $\mathbb Z$ en \emph{toute} dimension paire et $0$ en dimension impaire~\cite[4.5.2]{LeonardTh},~\cite[Theorem~4.9 et Example~4.10]{AraThB}. 
\smallbreak

Le but de cet article est d'introduire et étudier l'homologie polygraphique d'une $\infty$\nbd-catégorie stricte $X$ à coefficients dans un \emph{système local} sur $X$ (voir plus loin), généralisant ainsi l'homologie polygraphique à coefficients constants $\mathbb Z$ de Métayer. Dans un appendice, on adapte cette définition aux $\infty$\nbd-catégories faibles \og à la Grothendieck\fg{} et on conjecture que dans ce cadre l'homologie polygraphique coïncide avec la \og vraie\fg{} homologie.
\smallbreak

Si $X$ est un espace topologique, un \emph{\SL} sur $X$ est un foncteur contravariant $M$ du groupoïde fondamental de $X$ vers la catégorie des groupes abéliens. Autrement dit, à tout point $x$ de $X$, on associe un groupe abélien $M_x$ et à tout chemin $\gamma$ d'origine $x_0$ et d'extrémité $x_1$ un isomorphisme de groupes abéliens $\gamma^*:M_{x_1}\to M_{x_0}$ qui ne dépend que de la classe d'homotopie (à extrémités fixes) de $\gamma$, et ceci de façon compatible à la composition des chemins. L'homologie d'un espace $X$ à coefficients dans un système local $M$ a été définie par Steenrod~\cite[Section~10]{Steen1},~\cite{Steen2}, comme étant l'homologie d'un complexe de groupes abéliens, généralisant celui qui calcule l'homologie de $X$ à coefficients dans un groupe abélien constant. 
\smallbreak

Plus précisément, notons $\Delta_n$ le simplexe topologique standard de dimension $n$, enveloppe convexe de la base canonique $(e_i)_{0\leq i\leq n}$ de $\mathbb{R}^{n+1}$. On définit un complexe de groupes abéliens $\C(X,M)$ comme suit. La composante de dimension $n$ est donnée par la formule
$$\C_n(X,M)=\textstyle\bigoplus\limits_{x\,:\,\Delta_n\to X}M_{x_n:=x(e_n)}\,.$$
La différentielle est définie par
$$
d(x,m)=\textstyle\sum\limits_{0\leq i<n}(-1)^i(d_ix,m)+(-1)^n(d_nx,x_{n-1\,n}^*(m))\,,\quad x:\Delta_n\to X\,,\ m\in M_{x_n}\,,
$$
où $x_{n-1\,n}$ désigne le chemin dans $X$ image par $x$ du segment $e_{n-1}\,e_n$ dans $\Delta_n$, et~$d_ix$ la restriction de $x$ à la $i$\nbd-ème face de $\Delta_n$, enveloppe convexe de la famille $(e_k)_{k\neq i}$ (dont le \og dernier sommet\fg{} est $e_n$ si $i<n$ et $e_{n-1}$ si $i=1$, de sorte que la formule a un sens).
\smallbreak

Ces définitions s'adaptent facilement au cadre des ensembles simpliciaux (voir section~1), et l'homologie d'un espace topologique à coefficients dans un système local~$M$ s'interprète comme l'homologie de son complexe singulier à coefficients dans le système local induit par $M$. De plus, une variante d'un  théorème de Whitehead montre qu'une application continue ou un morphisme d'ensembles simpliciaux est une équivalence faible topologique ou simpliciale si et seulement si il induit une équivalence des groupoïdes fondamentaux et un isomorphisme sur l'homologie à coefficients dans tout système local sur le but du morphisme (voir théorème~\ref{thtopol}).
\smallbreak

Le cas $\infty$\nbd-catégorique est plus subtil. Si $X$ est une $\infty$\nbd-catégorie stricte, la \emph{catégorie fondamentale} de $X$ est la catégorie obtenue de $X$ en oubliant les $i$\nbd-flèches pour $i>1$, et en identifiant deux $1$\nbd-flèches si elles sont reliées par un zigzag de $2$\nbd-flèches. Le \emph{groupoïde fondamental} $\Pi_1(X)$ de $X$ est le groupoïde enveloppant de sa catégorie fondamentale (obtenu en inversant formellement toutes les flèches de cette dernière). Un \emph{système local} sur $X$ est un foncteur contravariant du groupoïde $\Pi_1(X)$ vers la catégorie $\Ab$ des groupes abéliens. Il revient au même de se donner un $\infty$\nbd-foncteur $M:X^\circ\to\Ab$ du dual total de $X$ (obtenu de $X$ en inversant le sens de ses $i$\nbd-flèches pour tout $i>0$) vers la catégorie des groupes abéliens (vue comme une $\infty$\nbd-catégorie dont les $i$\nbd-flèches sont des identités pour $i>1$) tel que l'image par $M$ de toute $1$\nbd-flèche de $X$ soit un isomorphisme. Pour tout objet $x$ de $X$, on notera $M_x$ le groupe abélien correspondant et pour toute $1$\nbd-flèche $x:x_0\to x_1$ de $X$, $x^*=M_x:M_{x_1}\to M_{x_0}$ l'isomorphisme de groupes abéliens image de $x$ par $M$.
\smallbreak

Pour définir l'homologie polygraphique d'une $\infty$\nbd-catégorie $X$ à coefficients dans un système local $M$, on associe d'abord au couple $(X,M)$ un complexe de groupes abéliens $\C(X,M)$ défini comme suit. Pour $n\geq0$, on pose
$$
\C_n(X,M)=\textstyle\bigoplus\limits_{x\in X_n}M_{t_0x}\Bigm/\sim_n\,,
$$
où $X_n$ désigne l'ensemble des $n$\nbd-flèches de $X$, $t_0x$ le $0$\nbd-but de la $n$\nbd-flèche $x$, et $\sim_n$ est la plus petite relation d'équivalence compatible à l'addition engendrée par les relations:
$$
(x_1*_0x_0,m)\sim_n(x_1,m)+(x_0,t_1(x_1)^*(m))\,,\quad m\in M_{t_0(x_1*_0x_0)=t_0x_1}\,,\leqno(*_0)
$$
$$
(x_1*_ix_0,m)\sim_n(x_1,m)+(x_0,m)\,,\quad 0<i<n\,,\ m\in M_{t_0(x_1*_ix_0)=t_0x_1=t_0x_0}\,,\leqno(*_i)
$$
quand ces composés $x_1*_ix_0$ ont un sens (où dans la relation $(*_0)$, $t_1(x_1)$ désigne le $1$\nbd-but de $x_1$ de sorte que $t_1(x_1)^*(m)$ appartient à $M_{t_0x_0})$. La différentielle 
$$
d_n:\C_n(X,M)\to\C_{n-1}(X,M)\,,\qquad n>0\,,
$$
est définie par la formule
$$
d_n(x,m)=\left\{\begin{aligned}
&(t_0x,m)-(s_0x,x^*m)\,,\quad\kern 11pt n=1\,,\cr
&(t_{n-1}x,m)-(s_{n-1}x,m)\,,\quad n>1\,.
\end{aligned}\right.
$$
On observe que si $M$ est le foncteur constant de valeur $\mathbb Z$, alors le complexe $\C(X,M)$ n'est autre que l'image de la $\infty$\nbd-catégorie $X$ par le foncteur d'abélianisation.
\smallbreak

Comme pour l'homologie polygraphique d'une $\infty$\nbd-catégorie $X$ à coefficients constants $\mathbb Z$, l'homologie polygraphique de $X$ à coefficients dans un système local $M$ est définie en choisissant une résolution polygraphique $p:P\to X$, autrement dit une fibration triviale de source cofibrante pour la structure folk sur $\ooCat$, et en posant, pour $n\geq0$, $\Hp{n}(X,M)=\H_n(\C(P,p^*(M)))$, où $p^*(M)$ désigne le système local sur $P$ obtenu de $M$ en précomposant avec $p$. De façon plus sophistiquée, le groupe abélien $\Hp{n}(X,M)$ est le $n$\nbd-ème groupe d'homologie de l'objet $\HP(X,M)$ de la catégorie dérivée des groupes abéliens correspondant au complexe $\C(P,p^*(M))$. Le résultat clef (théorème~\ref{th:Quilleng}) pour montrer qu'à isomorphisme canonique près $\HP(X,M)$ ne dépend pas de la résolution choisie est que le foncteur 
$$
F_{X,M}:\cm{\ooCat}{X}\to\Comp(\Ab)\,,\qquad (X',p:X'\to X)\mapsto\C(X',p^*(M))\,,
$$
est un foncteur de Quillen à gauche (admettant un adjoint à droite et respectant les cofibrations et les cofibrations triviales) pour la structure de catégorie de modèles sur $\cm{\ooCat}{X}$ induite de la structure folk sur $\ooCat$, et la structure injective sur $\Comp(\Ab)$. Ce théorème est non trivial et pour le démontrer on utilise une description précise des cofibrations triviales génératrices de la structure folk sur $\ooCat$.

En s'inspirant de la variante du théorème de Whitehead mentionnée précédemment, on introduit une classe d'équivalences faibles dans $\ooCat$, les équivalences polygraphiques. On dit qu'un $\infty$\nbd-foncteur $f:X'\to X$ est une \emph{équivalence polygraphique} si il induit une équivalence des groupoïdes fondamentaux et si pour tout système local $M$ sur $X$ et tout $n\geq0$, le morphisme de groupes abéliens $\Hp{n}(X',f^*(M))\to\Hp{n}(X,M)$, induit par $f$, est un isomorphisme. On vérifie que cette classe satisfait au 2 sur 3 (proposition~\ref{prop:2sur3}), et on démontre le théorème suivant (voir théorème~\ref{th:fond}):

\begin{thm"}
Toute équivalence folk est une équivalence polygraphique.
\end{thm"}
 
La théorie de l'homotopie des $\infty$\nbd-catégories strictes~\cite{DGCondE},~\cite{DGAI},~\cite{DGAII},~\cite{AraThB},~\cite{DGComp}, généralisant celle des catégories~\cite{QuK},~\cite{Th2},~\cite{PS1},~\cite{MalPRST},~\cite{Ci}, est basée sur le nerf de Street, qui prolonge le nerf usuel (1\nbd-)catégorique de Grothendieck~\cite{Nerf}. Dans~\cite{Or}, Ross Street a introduit un objet cosimplicial de $\ooCat$, l'objet cosimplicial des orientaux, qui définit par le procédé de Kan un couple de foncteurs adjoints entre la catégorie des ensembles simpliciaux et celle des $\infty$\nbd-catégories strictes
$$
c:\pref{\Catsmp}\longrightarrow\ooCat\,,\qquad N:\ooCat\longrightarrow\pref{\Catsmp}\,,
$$
formé du foncteur $c$ de réalisation $\infty$\nbd-catégorique et du foncteur nerf $N$, connu sous le nom de nerf de Street. Par définition, le type d'homotopie d'une $\infty$\nbd-catégorie stricte est le type d'homotopie de son nerf de Street. Plus précisément, si on définit les \emph{équivalences de Thomason} comme étant les $\infty$\nbd-foncteurs dont le nerf de Street est une équivalence faible simpliciale, Andrea Gagna a prouvé que le foncteur $N$ induit une équivalence de catégories entre la localisation de la catégorie $\ooCat$ par les équivalences de Thomason et la catégorie homotopique des ensembles simpliciaux, autrement dit, celle des types d'homotopie~\cite{Andrea}. On conjecture même qu'il existe une structure de catégorie de modèles sur $\ooCat$ avec comme équivalences faibles les équivalences de Thomason et comme cofibrations \emph{certaines} cofibrations folk~\cite{DG}.
\smallbreak

Le foncteur de réalisation $\infty$\nbd-catégorique $c$ n'est pas compatible aux équivalences de Thomason, au sens qu'il ne transforme pas les équivalences simpliciales en équivalences de Thomason, et en particulier ne définit pas un quasi-inverse homotopique du nerf de Street. En revanche, ce foncteur est compatible à l'homologie polygraphique à coefficients dans des \SLS, et aux équivalences polygraphiques. Si $X$ est un ensemble simplicial, il est facile de vérifier que son groupoïde fondamental est canoniquement isomorphe au groupoïde fondamental de la $\infty$\nbd-catégorie $c(X)$, et par suite, les systèmes locaux sur $X$ et sur $c(X)$ se correspondent bijectivement. On démontre le théorème suivant (théorème~\ref{thm:compc} et corollaire~\ref{cor:compc}):

\begin{thm"}
Pour tout ensemble simplicial $X$, tout \SL{} $M$ sur $X$, et tout $n\geq0$, on a un isomorphisme canonique $\Hp{n}(c(X),M)\simeq\H_n(X,M)$. En particulier, un morphisme d'ensembles simpliciaux est une équivalence d'homotopie faible si et seulement si son image par $c$ est une équivalence polygraphique.
\end{thm"}

Par définition, l'homologie singulière, ou plus simplement homologie, d'une $\infty$\nbd-catégorie est l'homologie de son type d'homotopie, c'est-à-dire de son nerf de Street. Plus précisément, si $X$ est une $\infty$\nbd-catégorie, le groupoïde fondamental de~$X$ est canoniquement isomorphe au groupoïde fondamental de l'ensemble simplicial~$N(X)$, les \SLS{} sur $X$ et sur $N(X)$ se correspondent donc bijectivement, et on définit l'homologie de $X$ à coefficients dans un système local $M$ sur $X$ par la formule $\H_n(X,M)=\H_n(N(X),M)$, $n\geq0$. Le morphisme d'adjonction $cN(X)\to X$ induit un morphisme de groupes abéliens $\Hp{n}(cN(X),M)\to\Hp{n}(X,M)$, qui en vertu du théorème~2, définit donc un morphisme de comparaison $\H_n(X,M)\to\Hp{n}(X,M)$. On démontre le théorème suivant (théorème~\ref{thm:Leonard}), qui généralise un théorème du premier auteur dans le cas d'un \SL{} constant de valeur $\mathbb Z$~\cite{Leonard},~\cite{LeonardTh}:

\begin{thm"}
Pour toute $(1\hbox{-})$catégorie $X$ et tout $n\geq0$, le morphisme de comparaison $\H_n(X,M)\to\Hp{n}(X,M)$ est un isomorphisme.
\end{thm"}

Pour résumer, on est en présence de trois classes d'équivalences faibles dans $\ooCat$, la classe $\Wfolk$ des équivalences faibles de la structure folk, la classe $\Wpol$ des équivalences polygraphiques et la classe $\Wthom$ des équivalences de Thomason. On a des inclusions
$$
\Wfolk\subset\Wpol\,,\qquad\Wfolk\subset\Wthom\,
$$
(la première en vertu du théorème~1, la seconde en vertu de la proposition~3.6.2 de~\cite{LeonardTh}). En revanche, il n'y a aucune inclusion entre les classes $\Wpol$ et $\Wthom$ (voir scholie~\ref{scholie}). Par ailleurs, si on note $\Wsmp$ la classe des équivalences faibles simpliciales, on a des égalités
$$
\Wsmp=c^{-1}(\Wpol)\,,\qquad\Wthom=N^{-1}(\Wsmp)
$$
(la première en vertu du théorème~2 et la seconde par définition). Enfin, il résulte du théorème~3 (et du théorème~1.6) que
$$
\Wpol\cap\Fl(\Cat)=\Wthom\cap\Fl(\Cat)\,,
$$
où $\Fl(\Cat)$ désigne la classe des flèches de $\Cat$. Les classes $\Wpol$ et $\Wthom$ partagent un grand nombre de propriétés formelles (voir remarque~\ref{remlocfond}). En particulier, la classe~$\Wpol$ satisfait à un théorème~A de Quillen relatif au dessus d'une ($1$\nbd-)catégorie (proposition~\ref{prop:thA}).
\smallbreak

On conjecture que dans le cadre des $\infty$\nbd-catégories faibles \og à la Grothendieck\fg{} les analogues des classes $\Wpol$ et $\Wthom$ coïncident et que pour les $\infty$\nbd-groupoïdes faibles les trois classes $\Wpol$, $\Wthom$ et $\Wfolk$ coïncident.
\smallbreak

Dans le corps de l'article, on développe la théorie de l'homologie et de l'homologie polygraphique dans un cadre un peu plus général, celui des \SLFS{} (où le rôle du groupoïde fondamental est remplacé par celui de la \emph{catégorie fondamentale}), généralisant ainsi ce qui est connu, dans le cadre des ($1$\nbd-)catégories, comme homologie des foncteurs. Une partie des résultats est valable dans ce contexte plus général, comme par exemple, le théorème~1 (au sens où l'homologie polygraphique à coefficients dans un \SLF{} est compatible aux équivalences folk) ainsi que le théorème~2, mais pas le théorème~3.

\subsection*{Plan de l'article} Dans la première section on rappelle la définition de l'homologie d'un ensemble simplicial à coefficients dans un \SL{}, en introduisant une variante, les \SLFS, et on esquisse la preuve d'un théorème de Whitehead caractérisant les équivalences faibles simpliciales. Le but de la deuxième section est de fixer la terminologie et les notations relatives aux $\infty$\nbd-catégories strictes, et faire quelques rappels sur la structure folk $\infty$\nbd-catégorique. Dans la troisième, on définit l'homologie polygraphique d'une $\infty$-catégorie stricte à coefficients dans un \SLF, on montre son invariance par équivalences faibles folk, et on étudie ses propriétés de fonctorialité. Dans la quatrième, on rappelle la définition du nerf de Street, et on prouve que l'homologie d'un ensemble simplicial à coefficients dans un \SLF{} coïncide avec l'homologie polygraphique de son image par l'adjoint à gauche du nerf de Street. On introduit les équivalences faibles polygraphiques et on montre qu'elles permettent de caractériser les équivalences faibles simpliciales. La cinquième section est consacrée à l'étude de plusieurs définitions équivalentes de l'homologie d'une catégorie à coefficients dans un système local faible, connue sous le nom d'homologie des foncteurs. Dans la dernière section, on montre que pour un \SL{}, cette homologie coïncide avec l'homologie polygraphique. On démontre également que les équivalences faibles polygraphiques $\infty$\nbd-catégoriques satisfont à un théorème~A de Quillen relatif au-dessus d'une ($1$\nbd-)catégorie. Le but de l'appendice~A est de prouver un lemme technique d'adjonction utile dans la section~3. Dans l'appendice~B, on interprète l'homologie polygraphique d'un système local libre en termes de la construction de Grothendieck $\infty$\nbd-catégorique. L'appendice~C est consacré à la définition de l'homologie polygraphique d'une $\infty$\nbd-catégorie faible.

\section{Rappels sur l'homologie d'un système local simplicial}

Dans cette section, on rappelle quelques résultats sur l'homologie d'un ensemble simplicial à coefficients dans un \SL, introduite par Steenrod~\cite{Steen2} (pour une exposition plus actuelle, voir par exemple dans~\cite{DK}). Aucune originalité n'est revendiquée.

\begin{paragr}
On note $\Catsmp$ la catégorie des simplexes standard dont les objets sont les ensembles ordonnés
$$\smp n=\{0<1<\cdots<n\}\,,\qquad n\geq0\,,$$
et les morphismes les applications croissantes entre iceux, $\pref\Catsmp$ la catégorie des ensembles simpliciaux, catégorie des préfaisceaux sur $\Catsmp$, et pour un ensemble simplicial $X$,\penalty -500{} $X_n$ l'ensemble de ses $n$\nbd-simplexes, $X_n=X(\smp n)$, et
$$d_i:X_n\to X_{n-1}\,,\quad s_i:X_n\to X_{n+1}\,,\qquad0\leq i\leq n\,,$$ 
les opérateurs de face et de dégénérescence.
Pour tout $n$\nbd-simplexe $x\in X_n$, et toute suite d'entiers $0\leq i_0\leq i_1\leq\cdots\leq i_p\leq n$, on désigne par $x_{i_0i_1\dots i_p}$ le $p$\nbd-simplexe de $X$ image inverse de $x$ par l'application $k\mapsto i_k$. En particulier,
$$
d_ix=x_{1\dots\widehat i\dots n}\quad\hbox{et}\quad s_ix=x_{1\dots i-1i\,i\,i+1\dots n}\,,\qquad0\leq i\leq n\,.
$$
\end{paragr}

\begin{paragr}
Soit $X$ un ensemble simplicial. On note $c^{}_1(X)$ (resp. $\Pi_1(X)$) sa \emph{catégorie fondamentale} (resp. son \emph{groupoïde fondamental}). On rappelle que $c^{}_1(X)$ est la catégorie définie par générateurs et relations par le graphe 
$$
\xymatrix{
X_1\ar@<.6ex>[r]^{d_0}\ar@<-.6ex>[r]_{d_1} 
&X_0
}
$$
($d_1$ l'application source et $d_0$ l'application but) et les relations
$$
d_1(x)=d_0(x)\circ d_2(x)\,,\quad x\in X_2\,,\qquad 1_x=s_0(x)\,,\quad x\in X_0\,,
$$
et $\Pi_1(X)$ le \emph{groupoïde enveloppant} de $c^{}_1(X)$, autrement dit, le groupoïde obtenu en inversant formellement toutes les flèches de la catégorie $c^{}_1(X)$. On a une équivalence canonique de groupoïdes $\Pi_1(X)\simeq\Pi_1(|X|)$, où $|X|$ désigne la \emph{réalisation topologique} de l'ensemble simplicial $X$ et $\Pi_1(|X|)$ son groupoïde fondamental. Tout morphisme d'ensembles simpliciaux $f:X'\to X$, induit un morphisme $c^{}_1(f):c^{}_1(X')\to c^{}_1(X)$ (resp.~$\Pi_1(f):\Pi_1(X')\to\Pi_1(X)$), définissant ainsi un foncteur de la catégorie des ensembles simpliciaux vers celle des petites catégories (resp.~des groupoïdes).
\smallbreak

Un \emph{\SLF} sur $X$ est un préfaisceau abélien sur $c^{}_1(X)$, autrement dit, un foncteur $M:c^{}_1(X)^\circ\to\Ab$, où $\Ab$ désigne la catégorie des groupes abéliens. Pour tout $x\in X_0$, on a donc un groupe abélien $M_x$, et pour tout $x\in X_1$, un morphisme de groupes abéliens $x^*:=x^*_M:=M_x:M_{x_1}\to M_{x_0}$ satisfaisant aux conditions 
$$
x^*_{02}=x^*_{01}\circ x^*_{12}\,,\quad x\in X_2\,, \qquad x^*_{00}=1_{M_x}\,,\quad x\in X_0\,.
$$
On dit que $M$ est un \emph{\SL} si pour tout $x\in X_1$, le morphisme $x^*$ est un isomorphisme, autrement dit, si le foncteur $M$ se factorise par $\Pi_1(X)$.
\end{paragr}

\begin{paragr}\label{def:homsimpl}
Soient $X$ un ensemble simplicial et $M$ un \SLF{} sur $X$. On définit un groupe abélien simplicial $\C(X,M)$ en posant
$$
\C_n(X,M)=\textstyle\bigoplus\limits_{x\in X_n}M_{x_n}\,,\quad n\geq0\,,
$$
et en définissant les opérateurs simpliciaux par
$$\begin{aligned}
&d_i(x,m)=\left\{\begin{aligned}
&(d_ix,m)\,,\kern 10pt\qquad0\leq i< n\,,\cr
&(d_nx,x_{n-1,n}^*(m))\,,\quad i=n\,,
\end{aligned}\right.\cr
\noalign{\vskip 3pt}
&s_i(x,m)=(s_ix,m)\,.
\end{aligned}$$
On en déduit un complexe de chaînes de groupes abéliens noté également $\C(X,M)$ dont la différentielle est donnée par
$$\begin{aligned}
&d(x,m)=\textstyle\sum\limits_{0\leq i<n}(-1)^i(d_ix,m)+(-1)^n(d_nx,x^*_{n-1,n}(m))\,,\ x\in X_n,\, m\in M_{x_n},\, n>0\,.
\end{aligned}$$
L'homologie de l'ensemble simplicial $X$ à valeurs dans le \SLF{} $M$ est définie par $\H_n(X,M)=\H_n(\C(X,M))$. Cette homologie est également celle du complexe normalisé correspondant, quotient de $\C(X,M)$ par le sous-complexe engendré pas les simplexes dégénérés de $X$.
On remarque que si le foncteur~$M$ est constant de valeur un groupe abélien fixe $M$, alors $\H_n(X,M)$ est l'homologie usuelle de l'ensemble simplicial $X$ à valeurs dans le groupe abélien $M$. 
\smallbreak

Soient $f:X'\to X$ un morphisme d'ensembles simpliciaux, et $M:c^{}_1(X)^\circ\to\Ab$ un \SLF{} sur $X$. Le composé 
$$
\xymatrix{
c^{}_1(X')^\circ\ar[r]^{c^{}_1(f)^\circ}
&c^{}_1(X)^\circ\ar[r]^-{M}
&\Ab
}
$$
définit un \SLF{} sur $X'$, noté $f^*(M)$, et le morphisme $f$ induit un morphisme de complexes $\C(f,M):\C(X',f^*(M))\to\C(X,M)$, d'où un morphisme de groupes abéliens $\H_n(f,M):\H_n(X',f^*(M))\to\H_n(X,M)$, pour $n\geq0$.
\end{paragr}

\begin{paragr}
Soit $(X,a)$ un ensemble simplicial pointé connexe. Le \emph{revêtement universel} de $(X,a)$ est l'ensemble simplicial pointé $(\rev{X},\rev{a})$ défini comme suit:
$$\begin{aligned}
&\rev{X}_n=\{(x,g)\,|\,x\in X_n,\,g\in\Pi_1(X;x_n,a)\}\cr
\noalign{\vskip 3pt}
&d_i(x,g)=\left\{\begin{aligned}
&(d_ix,g)\,,\kern 10pt\qquad0\leq i< n\,,\cr
&(d_nx,g\circ x_{n-1,n})\,,\quad i=n\,,
\end{aligned}\right.\cr
\noalign{\vskip 3pt}
&s_i(x,g)=(s_ix,g)\,,\cr
\noalign{\vskip 3pt}
&\rev{a}=(a,1_a)\,,
\end{aligned}$$
où $\Pi_1(X;x_n,a)=\Hom_{\Pi_1(X)}(x_n,a)$. On a un morphisme d'ensembles simpliciaux pointés
$$
\rev{X}\to X\,,\qquad(x,g)\mapsto x\,,
$$
dont la réalisation topologique $|\rev{X}|\to|X|$ s'identifie au revêtement universel de la réalisation topologique de $X$. Un morphisme d'ensembles simpliciaux connexes pointés $f:X'\to X$ induit un morphisme d'ensembles simpliciaux pointés
$$
\rev{f}:\rev{X}'\to\rev{X}\,,\qquad(x,g)\mapsto(f(x),\Pi_1(f)(g))\,,
$$
rendant commutatif le carré
$$
\xymatrix{
&\rev{X}'\ar[r]^{\rev{f}}\ar[d]
&\rev{X}\ar[d]
\\
&X'\ar[r]_f
&X
&\kern -20pt.\kern 20pt
}$$
\end{paragr}

\begin{paragr}
On rappelle qu'une \emph{équivalence faible simpliciale} est un morphisme d'ensembles simpliciaux $X'\to X$ induisant une équivalence d'homotopie $|X'|\to|X|$ entre les réalisations topologiques. 
\end{paragr}

On a la variante suivante du théorème de Whitehead (voir par exemple~\cite[Theorem~6.71]{DK}, ou pour la version cohomologique~\cite[Chapter~II, §3, Proposition~4]{Qu}), dont on esquisse une preuve pour la commodité du lecteur:

\begin{thm}\label{thtopol}
Soit $f:X'\to X$ un morphisme d'ensembles simpliciaux. Les conditions suivantes sont équivalentes:

\emph{(a)} $f$ est une équivalence faible simpliciale;

\emph{(b)} $\Pi_1(f)$ est une équivalence de groupoïdes, et pour tout \SL{} $M$ sur $X$ et tout $n\geq0$, le morphisme de groupes abéliens $\H_n(f,M):\H_n(X',f^*(M))\to\H_n(X,M)$, induit par $f$, est un isomorphisme;

\emph{(c)} $\Pi_1(f)$ est une équivalence de groupoïdes, et pour toute composante connexe $X'_0$ de $X'$ et tout $n\geq0$, le morphisme 
de groupes abéliens 
$\H_n(\rev{X}'_0,\mathbb{Z})\to\H_n(\rev{X}_0,\mathbb{Z})$,
où $X_0$ est la composante connexe de $X$ correspondant à $X'_0$ par l'équivalence $\Pi_1(f)$, et $\rev{X}'_0$ \emph{(resp.}~$\rev{X}_0$\emph{)} le revêtement universel de $X'_0$ \emph{(resp.} de $X_0$\emph{)} pointé par un de ses points \emph{(resp.} par l'image par $f$ de ce point\emph{),} est un isomorphisme.
\end{thm} 

\begin{proof}[Esquisse de preuve]
(\emph{a}) $\Rightarrow$ (\emph{b}). Supposons que $f$ est une équivalence faible simpliciale. Par définition, $\Pi_1(f)$ est alors une équivalence de groupoïdes. Il suffit donc de montrer que pour tout \SL{} $M:\Pi_1(X)^\circ\to \Ab$, le morphisme de complexes $\C(f,M):\C(X',f^*(M))\to\C(X,M)$ est un quasi-isomorphisme. Fixons un tel \SL{} $M$, et considérons la catégorie \smash{$\cm{\pref{\Catsmp}}{X}$} des ensembles simpliciaux au-dessus de $X$ munie de la structure de catégorie de modèles induite de celle de \smash{$\pref{\Catsmp}$}, la catégorie $\Comp(\Ab)$ des complexes de groupes abéliens munie de la structure de catégorie de modèles injective (dont les cofibrations sont les monomorphismes et les équivalences faibles les quasi-isomorphismes), et $F:\cm{\pref{\Catsmp}}{X}\to\Comp(\Ab)$ le foncteur associant à un objet $(T,p:T\to X)$ le complexe $\C(T,p^*(M))$. Il suffit de montrer que ce foncteur respecte les équivalences faibles, et pour cela, comme tous les objets de \smash{$\cm{\pref{\Catsmp}}{X}$} sont cofibrants, qu'il est un foncteur de Quillen à gauche. On vérifie facilement que le foncteur $F$ respecte les monomorphismes, et qu'il commute aux limites inductives et est donc un adjoint à gauche, puisqu'il est un foncteur entre catégories localement présentables. Pour conclure, il reste donc à montrer que ce foncteur envoie les cofibrations triviales génératrices de \smash{$\cm{\pref{\Catsmp}}{X}$} sur des quasi-isomorphismes. Or, les cofibrations triviales génératrices de \smash{$\cm{\pref{\Catsmp}}{X}$} sont les inclusions de cornets au-dessus de $X$, 
$$
\Lambda^k_n\overset{i}{\hookrightarrow}\smp{n}\overset{p}{\longrightarrow}X\,,
$$
et comme le groupoïde $\Pi_1({\smp{n}})$ est trivial le \SL{} $p^*(M)$ sur $\smp n$ est isomorphe au système constant de valeur $M_{p(n)}$. On en déduit que l'image du morphisme $i$ au-dessus de $X$ par le foncteur $F$ s'identifie au morphisme $\C(\Lambda^k_n,M_{p(n)})\to\C(\smp{n},M_{p(n)})$ des complexes calculant l'homologie à valeurs dans le groupe constant $M_{p(n)}$.  Ce morphisme de complexes est un quasi-isomorphisme car, d'une part, l'image d'une inclusion de cornets dans la catégorie des ensembles simpliciaux à homotopie près est un isomorphisme, et d'autre part, le foncteur homologie à valeurs dans un groupe abélien constant transforme les homotopies simpliciales en homotopies de complexes. 
\smallbreak

(\emph{b}) $\Rightarrow$ (\emph{c}). Pour montrer cette implication on se ramène aussitôt au cas où $X$ et $X'$ sont $0$\nbd-connexes. Choisissons un point $a'$ de $X'$, soit $a$ son image par $f$, et considérons le revêtement universel $\rev X'$ (resp. $\rev X$) de l'ensemble simplicial pointé $(X',a')$ (resp.~$(X,a)$) et le morphisme $\rev f:\rev X'\to\rev X$ induit par $f$. On définit un \SL{} $M$ sur $X$ comme suit. Pour tout $0$\nbd-simplexe $x$ de $X$, $M_x$ est le $\mathbb Z$\nbd-module libre engendré par l'ensemble $\Pi_1(X;x,a)$. Pour tout $1$\nbd-simplexe $x$ de $X$, on a un morphisme de $\mathbb Z$\nbd-modules $x^*:M_{x_1}\to M_{x_0}$, induit par l'application
$$
\Pi_1(X;x_1,a)\to\Pi_1(X;x_0,a)\,,\qquad g\mapsto g\circ x\,.
$$
La définition du groupoïde $\Pi_1(X)$ par générateurs et relations montre que pour tout $2$\nbd-simplexe $x$ de $X$, on a $x^*_{02}=x^*_{01}\circ x^*_{12}$, et qu'on définit ainsi un foncteur $M:\Pi_1(X)^\circ\to\Ab$. Une vérification simple mais fastidieuse montre alors que le morphisme de complexes $\C(f,M):\C(X',f^*(M))\to\C(X,M)$ s'identifie au morphisme $\C(\rev f,\mathbb Z):\C(\rev X',\mathbb Z)\to\C(\rev X,\mathbb Z)$, ce qui prouve l'implication.
\smallbreak

(\emph{c}) $\Rightarrow$ (\emph{a}). On peut à nouveau supposer $X'$ et $X$ $0$\nbd-connexes, choisir un point $a'$ de $X'$ et considérer le morphisme $\rev f:\rev X'\to\rev X$ induit par $f$, entre les revêtements universels correspondants. La réalisation topologique $|\rev X'|$ (resp.~$|\rev X|$) de $\rev X'$  (resp.~$\rev X$) s'identifie au revêtement universel de $X'$ (resp.~$X$), et $|\rev f|$ au morphisme entre les revêtements universels induit par $|f|$. Or, en vertu d'un théorème \hbox{d'Eilenberg} (voir par exemple~\cite[Appendix Two, §1]{GZ}), pour tout ensemble simplicial $T$ et tout $n\geq0$, on a un isomorphisme $\H_n(|T|,\mathbb Z)\simeq\H_n(T,\mathbb Z)$, fonctoriel en $T$. L'hypothèse de la condition (\emph{c}) implique donc que pour tout $n\geq0$, le morphisme $\H_n(|\rev f|,\mathbb Z):\H_n(|\rev X'|,\mathbb Z)\to\H_n(|\rev X|,\mathbb Z)$ est un isomorphisme. Comme les espaces $|\rev X'|$ et $|\rev X|$ sont simplement connexes, le théorème de Whitehead (voir par exemple~\cite[Corollary~4.33]{Hat}) implique que $|\rev f|$ est une équivalence d'homotopie. Comme par hypothèse $\Pi_1(f)$, donc aussi $\Pi_1(|f|)$, est une équivalence de groupoïdes, la longue suite exacte des groupes d'homotopie implique que $|f|$, donc aussi $f$, est une équivalence faible.
\end{proof}

\begin{rem}
Le fait que dans la condition (\emph{b}) du théorème ci-dessus on a supposé que \emph{$M$ est un \SL}, et pas seulement un \SLF, est essentiel. En effet, il faut se garder de croire que l'homologie à coefficients dans un \SLF{} soit invariante par les équivalences faibles simpliciales. Par exemple, si $f:X'\to X$ est le nerf de l'inclusion de la catégorie  ponctuelle $\{1\}$ dans la catégorie $\{0\to1\}$, et $M$ le \SLF{} $M_1\to M_0$, avec $M_0=\mathbb Z$ et $M_1=0$, alors on vérifie facilement que $\H_0(X',f^*(M))=0$, tandis que $\H_0(X,M)=\mathbb Z$.
\end{rem}

\section{Rappels sur les \pdfoo-catégories strictes et les polygraphes}

Toutes les $\infty$\nbd-catégories considérées dans cet article seront des $\infty$\nbd-catégories \emph{strictes} et les $\infty$\nbd-foncteurs des $\infty$\nbd-foncteurs \emph{stricts}, et on dira simplement $\infty$\nbd-catégorie pour $\infty$-catégorie stricte et $\infty$\nbd-foncteur pour $\infty$\nbd-foncteur strict. Aucune hypothèse d'inversibilité (stricte ou faible) des $i$\nbd-flèches n'est faite en général, autrement dit, il s'agit de $(\infty,\infty)$\nbd-catégories. 

\begin{paragr}\label{ooCat}
On note $\ooCat$ la catégorie des (petites) $\infty$-catégories et $\infty$\nbd-foncteurs. Pour~$X$ une $\infty$-catégorie et $n\geq0$, on note $X_n$ l'ensemble des ses $n$\nbd-cellules, qu'on appellera également \emph{objets}, si $n=0$, et \emph{$n$\nbd-flèches}, si $n>0$. Pour $x\in X_n$ et $0\leq i\leq n$, $s_ix$ et $t_ix$ désignent respectivement la $i$\nbd-cellule $i$\nbd-source et $i$\nbd-but (itérés) de $x$ (en particulier $s_nx=x=t_nx$), et $1_x$ la $(n+1)$\nbd-flèche unité de $x$. On dira parfois qu'une $n$\nbd-cellule est \emph{triviale} si elle est l'unité d'une $(n-1)$\nbd-cellule. Pour $0\leq i<\min\{n_0,n_1\}$, $x_0\in X_{n_0}$, $x_1\in X_{n_1}$, on dira que $x_0$ et $x_1$ sont \emph{$i$\nbd-composables} si $t_ix_0=s_ix_1$, et on notera $x_1*_ix_0$ leur $i$\nbd-composé. Par convention, si $i_1<i_2$, la $i_1$\nbd-composition sera prioritaire par rapport à la $i_2$\nbd-composition. Par exemple, on écrira $x_2*_1x_1*_0x_0$ pour $x_2*_1(x_1*_0x_0)$, dès que cette dernière expression a un sens. Pour $n\geq0$, on dit qu'une $n$\nbd-cellule est \emph{indécomposable} si, ou bien $n=0$, c'est-à-dire $x$ est un objet, ou bien $n>0$, $x$ n'est pas une unité, et pour tout $0\leq i<n$ et toute décomposition $x=x_1*_ix_0$ de $x$, au moins une des deux cellules $x_0,\,x_1$ est une unité (itérée) d'une $i$\nbd-cellule. On note $X^\circ$ la $\infty$\nbd-catégorie \emph{dual total} de $X$, obtenue de $X$ en inversant le sens de ses $i$\nbd-flèches pour tout $i>0$.
\smallbreak

Une $n$\nbd-catégorie (stricte) sera toujours considérée comme une $\infty$\nbd-catégorie dont les $i$\nbd-cellules sont triviales pour $i>n$. En particulier, la catégorie $\Cat$ des petites catégories, sera considérée comme une sous-catégorie pleine de $\ooCat$. L'inclusion de la sous-catégorie pleine de $\ooCat$ formée des $n$\nbd-catégories admet un adjoint à gauche et un adjoint à droite appelés respectivement foncteur de \emph{troncation intelligente} et foncteur de \emph{troncation bête}. Si $X$ est une $\infty$\nbd-catégorie, son $n$\nbd-tronqué bête est la sous-$\infty$-catégorie de $X$ ayant mêmes $i$\nbd-cellules que $X$ pour $0\leq i\leq n$, et seulement des $i$\nbd-cellules triviales pour $i>n$. Son tronqué intelligent est le quotient de $X$ ayant mêmes $i$\nbd-cellules que $X$ pour $0\leq i<n$, dont les $n$\nbd-cellules sont obtenues de celles de~$X$ en identifiant deux $n$\nbd-cellules si elles sont reliées par un zigzag de $(n+1)$\nbd-cellules de~$X$, et n'ayant que des $i$\nbd-cellules triviales pour $i>n$. La \emph{catégorie fondamentale} de $X$ est son $1$\nbd-tronqué intelligent, et son \emph{groupoïde fondamental} le groupoïde enveloppant de cette dernière. Par convention, le (-1)\nbd-tronqué bête d'une $\infty$\nbd-catégorie est la $\infty$\nbd-catégorie vide $\varnothing$.
\smallbreak

Pour $n\geq0$, on note $D_n$ le \emph{$n$\nbd-disque}, unique à isomorphisme unique près, $n$\nbd-catégorie ayant exactement une $n$\nbd-cellule non triviale
$\zeta_n\kern1pt$, appelée sa \emph{cellule fondamentale}, et pour $0\leq i<n$, exactement deux $i$\nbd-cellules non triviales, égales à $s_i\zeta_n$ et $t_i\zeta_n\kern1pt$.
\[
    \shorthandoff{;:}
    D_{0} = \bullet
    \,,\quad
    D_{1} = \xymatrix{\bullet \ar[r] & \bullet},
    \quad
    D_{2} = \xymatrix@C=3pc@R=3pc{\bullet \ar@/^2.5ex/[r]_{}="0"
      \ar@/_2.5ex/[r]_{}="1"
      \ar@2"0";"1"
    &  \bullet},\quad
    D_{3} = \xymatrix@C=3pc@R=3pc{\bullet \ar@/^3ex/[r]_(.47){}="0"^(.53){}="10"
      \ar@/_3ex/[r]_(.47){}="1"^(.53){}="11"
      \ar@<2ex>@2"0";"1"_{}="2" \ar@<-2ex>@2"10";"11"^{}="3"
      \ar@3"3";"2"_{}
    &  \bullet}\,,\quad\dots
  \]
Le couple $(D_n,\zeta_n)$ coreprésente le foncteur
$$
\ooCat\to\Ens\,,\qquad X\mapsto X_n\,,
$$
associant à une $\infty$\nbd-catégorie l'ensemble de ses $n$\nbd-cellules, autrement dit, on a une bijection naturelle en $X$
$$
\Hom_{\ooCat}(D_n,X)\overset{\sim}{\to}X_n\,,\qquad f\mapsto f(\zeta_n)\,.
$$
En particulier, la $(n-1)$\nbd-cellule $s_{n-1}\zeta_n$ (resp.~$t_{n-1}\zeta_n$) définit le morphisme \emph{cosource} (resp.~\emph{cobut}) $D_{n-1}\to D_n$.
\smallbreak

On note $S^{n-1}$ la \emph{$(n-1)$\nbd-sphère}, $(n-1)$\nbd-tronqué bête de~$D_n$, 
\[
    \shorthandoff{;:}
    S^{-1} = \varnothing
    \,,\quad
    S^{0} = \xymatrix{\bullet & \bullet},
    \quad
    S^{1} = \xymatrix@C=3pc@R=3pc{\bullet \ar@/^2.5ex/[r]_{}="0"
      \ar@/_2.5ex/[r]_{}="1"
          &  \bullet},\quad
    S^{2} = \xymatrix@C=3pc@R=3pc{\bullet \ar@/^3ex/[r]_(.47){}="0"^(.53){}="10"
      \ar@/_3ex/[r]_(.47){}="1"^(.53){}="11"
      \ar@<2ex>@2"0";"1"_{}="2" \ar@<-2ex>@2"10";"11"^{}="3"
          &  \bullet}\,,\quad\dots
  \]
et $i_n:S^{n-1}\hookrightarrow D_n$ l'inclusion. 
\end{paragr}

\begin{paragr}
La notion de $\infty$\nbd-catégorie libre au sens des polygraphes, ou plus simplement polygraphe, introduite indépendamment par Ross Street~\cite{StrPol},~\cite{Or} (sous le nom de \og computad\fg{}) et Albert Burroni~\cite{Burr1},~\cite{Burr2}, est l'analogue exact en théorie des $\infty$\nbd-catégories de la notion de CW\nbd-complexe en topologie. 
\smallbreak

Un \emph{polygraphe} est une $\infty$\nbd-catégorie $X$, admettant une filtration notée
$$
\varnothing=X_{\leq -1}\hookrightarrow X_{\leq 0}\hookrightarrow X_{\leq 1}\hookrightarrow\cdots \hookrightarrow X_{\leq n-1}\hookrightarrow X_{\leq n}\hookrightarrow\cdots\hookrightarrow \varinjlim X_{\leq n}=X
$$
telle que pour tout $n\geq0$, il existe un carré cocartésien de $\ooCat$ de la forme
$$
\xymatrixcolsep{3pc}
\xymatrix{
&&\protect\smash{\mathop\amalg\limits_{\,I_n}}S^{n-1}\ar[r]\ar[d]_(.58){\mathop\amalg\limits_{\,\,I_n}\kern -2pti_n}
&X_{\leq n-1}\vrule height 10pt width 0pt\ar[d]
&&
\\
&&\protect\smash{\mathop\amalg\limits_{\,I_n}}D_n\ar[r]
&X_{\leq n}\vrule height 7pt width 0pt
&\kern -30pt.\kern 30pt
&
}$$
\vskip 6pt
\noindent
La $\infty$\nbd-catégorie $X_{\leq n}$ est alors une $n$\nbd-catégorie et s'identifie au $n$\nbd-tronqué bête de~$X$. 
De plus, la flèche horizontale du bas est déterminée par les images de la cellule principale de $D_n$ par les composantes $D_n\to X_{\leq n}$ de cette flèche, établissant ainsi une bijection entre l'ensemble $I_n$ des indices et l'ensemble $B_n$ de ces images. On démontre que~$B_n$ est l'ensemble des $n$\nbd-cellules indécomposables de $X$~\cite[section~5]{Makkai}. Par conséquent, aussi bien la filtration que les carrés cocartésiens ci-dessus sont déterminés par la seule structure de $\infty$\nbd-catégorie de $X$. On dit alors que l'ensemble de cellules $B=\amalg_{n\geq}B_n$ est la \emph{base} du polygraphe $X$ ou que $B$ \emph{engendre librement} (\emph{au sens des polygraphes}) la $\infty$\nbd-catégorie~$X$.
\end{paragr}

\begin{paragr}\label{strfolk}
La catégorie $\ooCat$ des $\infty$\nbd-catégories admet une structure de catégorie de modèles à engendrement cofibrant, connue sous le nom de \og structure folk\fg~\cite{LMW}. Les cofibrations de cette structure sont engendrées par l'ensemble $I$ formé des inclusions canoniques $i_n:S^{n-1}\hookrightarrow D_n$, $n\geq 0$, des sphères dans les disques, et les cofibrations triviales par l'ensemble $J$ des inclusions $j_n:D_n\hookrightarrow J_{n+1}$, $n\geq 0$, définies comme suit. En vertu de l'argument du petit objet, le morphisme $S^n\to D_n$ envoyant les deux $n$\nbd-cellules non triviales de $S^n$ sur la cellule principale de $D_n$ et induisant un isomorphisme des $(n-1)$\nbd-tronqués bêtes, se décompose en une cofibration suivie d'une fibration triviale (flèche ayant la propriété de relèvement à droite relativement aux flèches $i_n$, $n\geq0$) 
$$
S^n\overset{k_n}{\longrightarrow} J_{n+1} \overset{r_n}{\longrightarrow} D_n\ .
$$
Le morphisme $j_n$ est le composé de l'inclusion $D_n\to S^n$, induite par la cosource $D_n\to D_{n+1}$, suivi de $k_n$.
Les équivalences faibles de cette structure, appelées \emph{équivalences folk}, sont une généralisation $\infty$\nbd-catégorique naturelle des équivalences de catégories (et de $2$\nbd-catégories). Si $f$ est une équivalence folk, pour tout $n\geq0$, le $n$\nbd-tronqué intelligent de $f$ l'est aussi, mais la réciproque est fausse. Les objets cofibrants de cette structure sont les polygraphes~\cite{MetCof}, et tout objet est fibrant.
\end{paragr}

\section{Homologie polygraphique d'un système local \pdfoo-catégorique}

\begin{paragr}
Soit $X$ une $\infty$\nbd-catégorie. Un \emph{\SLF} sur $X$ est un $\infty$\nbd-foncteur $M:X^\circ\to\Ab$, où $\Ab$ désigne la catégorie des groupes abéliens vue comme une $\infty$\nbd-catégorie dont les $n$\nbd-cellules sont des identités pour $n>1$. Par adjonction, un tel $\infty$\nbd-foncteur correspond à un foncteur contravariant $c_1(X)^\circ\to\Ab$ de la catégorie fondamentale ($1$\nbd-tronqué intelligent) de $X$ vers $\Ab$. Il revient donc au même de se donner pour tout objet $x$ de $X$ un groupe abélien $M_x$ et pour toute $1$\nbd-flèche $x:x_0\to x_1$ de $X$ un morphisme de groupes abéliens $x^*:=x^*_M:=M_x:M_{x_1}\to M_{x_0}$ satisfaisant aux conditions suivantes. Pour toute $2$\nbd-flèche $x:x_0\Rightarrow x_1$ de $X$, on a $x^*_0=x^*_1$, pour tout couple de $1$\nbd-flèches composables $x_0,x_1$ de $X$, on a $(x_1*_0x_0)^*=x_0^*x_1^*$, et pour tout objet $x$ de $X$, on a $1^*_x=1_{M_x}$. On dit que le \SLF{} $M$ est un \emph{\SL} si pour toute $1$\nbd-flèche de $X$, le morphisme de groupes abéliens $x^*$ est un isomorphisme, autrement dit, si le $\infty$\nbd-foncteur $M$ se factorise par le groupoïde fondamental $\Pi_1(X)$ de $X$, groupoïde enveloppant de sa catégorie fondamentale $c_1(X)$.
\end{paragr}

\begin{paragr}\label{relfond}
Soient $X$ une $\infty$\nbd-catégorie et $M$ un \SLF{} sur $X$. On définit un complexe de chaînes de groupes abéliens $\C(X,M)$ comme suit: 
$$
\C_n(X,M)=\textstyle\bigoplus\limits_{x\in X_n}M_{t_0x}\Bigm/\sim_n\,,\quad n\geq0\,,
$$
où $\sim_n$ désigne la plus petite relation d'équivalence compatible à l'addition engendrée par les relations:
$$
(x_1*_0x_0,m)\sim_n(x_1,m)+(x_0,t_1(x_1)^*(m))\,,\quad m\in M_{t_0(x_1*_0x_0)=t_0x_1}\,,\leqno(*_0)
$$
$$
(x_1*_ix_0,m)\sim_n(x_1,m)+(x_0,m)\,,\quad 0<i<n\,,\ m\in M_{t_0(x_1*_ix_0)=t_0x_1=t_0x_0}\,,\leqno(*_i)
$$
quand ces composés $x_1*_ix_0$ ont un sens (en particulier $\sim_0$ est l'égalité et $\sim_1$ est engendrée par la seule relation $(*_0)$). La différentielle 
$$
d_n:\C_n(X,M)\to\C_{n-1}(X,M)\,,\qquad n>0\,,
$$
est définie par la formule
$$
d_n(x,m)=\left\{\begin{aligned}
&(t_0x,m)-(s_0x,x^*m)\,,\quad\kern 11pt n=1\,,\cr
&(t_{n-1}x,m)-(s_{n-1}x,m)\,,\quad n>1\,.
\end{aligned}\right.
$$
\end{paragr}

\begin{prop}
Les formules ci-dessus définissent bien un complexe de chaînes.
\end{prop}

\begin{proof}
Vérifions d'abord la compatibilité de la différentielle $d_n$, $n>0$, aux relations $(*_i)$, $0\leq i<n$. On distingue plusieurs cas:

--- $n=1$, $i=0$, $x_0,\,x_1$ $1$\nbd-cellules $0$\nbd-composables. On a: 
$$
\begin{aligned}
d_1(x_1*_0x_0,m)&=(t_0(x_1*_0x_0),m)-(s_0(x_1*_0x_0), (x_1*_0x_0)^*(m))\cr
&=(t_0x_1,m)-(s_0x_0,x_0^*x_1^*(m))
\end{aligned}
$$
et
$$
\begin{aligned}
d_1(x_1,m)&+d_1(x_0,t_1(x_1)^*(m))=\cr
&=(t_0x_1,m)-(s_0x_1,x_1^*m)+(t_0x_0,x_1^*m)-(s_0x_0,x_0^*x_1^*(m))\cr
&=(t_0x_1,m)-(s_0x_0,x_0^*x_1^*(m))\,.
\end{aligned}
$$

--- $n>1$, $i=0$, $x_0,\,x_1$ $n$\nbd-cellules $0$\nbd-composables. On a:
$$
\begin{aligned}
d_n(x_1*_0x_0,m)&=(t_{n-1}(x_1*_0x_0),m)-(s_{n-1}(x_1*_0x_0), m)\cr
&=(t_{n-1}x_1*_0t_{n-1}x_0,m)-(s_{n-1}x_1*_0s_{n-1}x_0, m)\cr
&\sim_{n-1}(t_{n-1}x_1,m)+(t_{n-1}x_0,t_1t_{n-1}(x_1)^*(m))\cr
&\kern 18pt-(s_{n-1}x_1,m)-(s_{n-1}x_0,t_1s_{n-1}(x_1)^*(m))
\end{aligned}
$$
et
$$
\begin{aligned}
d_n(x_1,&m)+d_n(x_0,t_1(x_1)^*(m))=\cr
&=(t_{n-1}x_1,m)-(s_{n-1}x_1,m)+(t_{n-1}x_0,t_1(x_1)^*(m))-(s_{n-1}x_0,t_1(x_1)^*(m))\,.
\end{aligned}
$$
Si $n>2$, on conclut par les relations globulaires, et si $n=2$ en remarquant que $t_1s_1(x_1)^*=s_1(x_1)^*=t_1(x_1)^*$, puisque la $2$\nbd-flèche $x_1$ de $X$ s'envoie sur une identité dans $\Ab$ qui est une ($1$-)catégorie. 

--- $n>1$, $0<i<n-1$, $x_0,\,x_1$ $n$\nbd-cellules $i$\nbd-composables. On a:
$$
\begin{aligned}
d_n(x_1*_ix_0,m)&=(t_{n-1}(x_1*_ix_0),m)-(s_{n-1}(x_1*_ix_0), m)\cr
&=(t_{n-1}x_1*_it_{n-1}x_0,m)-(s_{n-1}x_1*_is_{n-1}x_0, m)\cr
&\sim_{n-1}(t_{n-1}x_1,m)+(t_{n-1}x_0,m)-(s_{n-1}x_1,m)-(s_{n-1}x_0,m)
\end{aligned}
$$
et
$$
\begin{aligned}
d_n(x_1,&m)+d_n(x_0,m)=\cr
&=(t_{n-1}x_1,m)-(s_{n-1}x_1,m)+(t_{n-1}x_0,m)-(s_{n-1}x_0,m)\,.
\end{aligned}
$$

--- $n>1$, $i=n-1$, $x_0,\,x_1$ $n$\nbd-cellules $(n-1)$\nbd-composables. On a:
$$
\begin{aligned}
d_n(x_1*_{n-1}x_0,m)&=(t_{n-1}(x_1*_{n-1}x_0),m)-(s_{n-1}(x_1*_{n-1}x_0), m)\cr
&=(t_{n-1}x_1,m)-(s_{n-1}x_0, m)
\end{aligned}
$$
et
$$
\begin{aligned}
d_n(x_1,&m)+d_n(x_0,m)=\cr
&=(t_{n-1}x_1,m)-(s_{n-1}x_1,m)+(t_{n-1}x_0,m)-(s_{n-1}x_0,m)\cr
&=(t_{n-1}x_1,m)-(s_{n-1}x_0, m)\,.
\end{aligned}
$$
Il reste à prouver que $d_{n-1}d_n=0$, pour $n\geq2$;

--- $n>2$. L'égalité est conséquence immédiate des relations globulaires.

--- $n=2$, $x$ $2$\nbd-cellule de $X$, $m\in M_{t_0x}$. On a:
$$
\begin{aligned}
d_1d_2(x,m)&=d_1((t_1x,m)-(s_1x,m))\cr
&=(t_0t_1x,m)-(s_0t_1x,t_1(x)^*(m))\cr
&-(t_0s_1x,m)+(s_0s_1x,s_1(x)^*(m))=0\,,
\end{aligned}
$$
grâce aux relations globulaires et au fait que $t_1(x)^*(m)=s_1(x)^*(m)$, puisqu'on a une $2$\nbd-flèche $x:s_1(x)\Rightarrow t_1(x)$.
\end{proof}

\begin{rem}\label{remunit}
Pour $n>0$, la relation $(*_{n-1})$ implique aussitôt que si $x$ est une $(n-1)$\nbd-cellule de $X$ et $m\in M_{t_0x}$, on a $(1_x,m)\sim_n0$. On en déduit que pour $0\leq i<n'<n$, la relation $(*_i)$ implique que si $x_0,x_1$ sont deux cellules $i$\nbd-composables de $X$ et $m\in M_{t_0(x_1*_i x_0)}$, et si $x_1$ est de dimension $n$ et $x_0$ de dimension $n'$, on a 
$$(x_1*_ix_0,m)\sim_n(x_1,m)\,,$$
et inversement, si $x_0$ est de dimension $n$ et $x_1$ de dimension $n'$, on a
$$(x_1*_ix_0,m)\sim_n\left\{\begin{aligned}
&(x_0,t_1(x_1)^*(m))\,,\quad i=0\,,\cr
&(x_0,m)\,,\quad i>0\,.
\end{aligned}\right.$$
\end{rem}

\begin{paragr}\label{FonctCf}
Soient $f:X'\to X$ un $\infty$\nbd-foncteur strict et $M:X^\circ\to\Ab$ un \SLF{} sur $X$. Le composé
$$
\xymatrix{
X'{}^\circ\ar[r]^{f^\circ}
&X^\circ\ar[r]^M
&\Ab
}$$
définit un \SLF{} sur $X'$ noté $f^*(M)$. Pour toute $1$\nbd-flèche $x'$ de $X'$ et tout $m\in f^*(M)_{t_0x'}=M_{f(t_0x')}=M_{t_0(fx')}$, on a alors 
$$x'^*(m)=f^*(M)_{x'}(m)=M_{f(x')}(m)=f(x')^*(m)$$
(où \og l'étoile en haut\fg{} du membre de gauche est relative au \SLF{} $f^*(M)$ sur $X'$, tandis que celle du membre de droite est relative au \SLF{} $M$ sur $X$).
\smallbreak

Pour $n\geq0$, le $\infty$\nbd-foncteur $f$ induit, en posant $M'=f^*(M)$, un morphisme de groupes abéliens
$$
\textstyle\bigoplus\limits_{x'\in X'_n}\kern -5pt M'_{t_0x'}\to\textstyle\bigoplus\limits_{x\in X_n}\kern -5pt M_{t_0x}\,,\quad(x',m)\mapsto(f(x'),m)\,,\ x'\in X'_n\,,\ m\in M'_{t_0x'}=M_{t_0(fx')} 
$$ 
qui est compatible aux relations $(*_i)$, $0\leq i\leq n$, et aux différentielles, et définit donc un morphisme de complexes $\C(f,M):\C(X',f^*(M))\to\C(X,M)$. De plus, pour tout $\infty$\nbd-foncteur $f':X''\to X'$, on vérifie aussitôt qu'on a l'égalité
$$
\C(ff',M)=\C(f,M)\circ\C(f',f^*(M))\,,
$$
et que si $X'=X$ et $f=1_X$, alors $\C(1_X,M)=1_{\C(X,M)}$.
\end{paragr}

\begin{paragr}\label{FonctCphi}
Soient $X$ une $\infty$\nbd-catégorie, et $M$ et $M'$ deux \SLFS{} sur $X$. Un \emph{morphisme de \SLFS{} sur $X$} de $M'$ vers $M$ est une transformation naturelle du foncteur $M'$ vers le foncteur $M$. Concrètement, la donnée d'une telle transformation $\varphi$ revient à la donnée, pour tout objet $x$ de $X$, d'un morphisme de groupes abéliens $\varphi_x:M'_x\to M_x$ tel que pour toute $1$\nbd-flèche $x:x_0\to x_1$, le carré
$$
\xymatrixcolsep{3pc}
\xymatrix{
M'_{x_1}\ar[r]^{\varphi_{x_1}}\ar[d]_{x^*=M'_x}
&M_{x_1}\ar[d]^{x^*=M_x}
\\
M'_{x_0}\ar[r]_{\varphi_{x_0}}
&M_{x_0}
}
$$
soit commutatif. Le morphisme $\varphi$ induit, pour $n\geq0$, un morphisme de groupes abéliens
$$
\textstyle\bigoplus\limits_{x\in X_n}\kern -5pt M'_{t_0x}\to\textstyle\bigoplus\limits_{x\in X_n}\kern -5pt M_{t_0x}\,,\quad(x,m)\mapsto(x,\varphi_{t_0x}(m))\,,\ x\in X_n\,,\ m\in M'_{t_0x}\,, 
$$ 
dont la compatibilité aux relations $(*_i)$ est tautologique pour $0<i\leq n$, et résulte facilement de la commutativité du carré ci-dessus pour $i=0$. De même, sa compatibilité aux différentielles $d_i$ est évidente pour $i>1$, et résulte aussitôt de la commutativité dudit carré pour $i=1$.
Il définit donc un morphisme de complexes $\C(X,\varphi):\C(X,M')\to\C(X,M)$. De plus, si $\varphi':M''\to M'$ est un morphisme de \SLFS{} sur $X$, on a
$$
\C(X,\varphi\varphi')=\C(X,\varphi)\circ\C(X,\varphi')\,,
$$
et si $M'=M$ et $\varphi=1_M$, alors $\C(X,1_M)=1_{\C(X,M)}$.
\end{paragr}

\begin{paragr}\label{CompFonctCfphi}
Les fonctorialités des deux paragraphes précédents satisfont à la compatibilité suivante dont la vérification est immédiate. Soient $X$ une $\infty$\nbd-catégorie, $\varphi:M'\to M$ un morphisme de \SLFS{} sur $X$, et $f:X'\to X$ un $\infty$\nbd-foncteur. Alors la transformation naturelle $f^*(\varphi):=\varphi\ast f^\circ$ définit un morphisme de \SLFS{} $f^*(\varphi):f^*(M')\to f^*(M)$ sur $X'$, et le carré
$$\xymatrixcolsep{4.5pc}
\xymatrix{
&\C(X'\kern-2pt,f^*(M'))\ar[r]^-{\C(X'\kern-2pt,f^*(\varphi))\,}\ar[d]_{\C(f,M')}
&\C(X'\kern-2pt,f^*(M))\ar[d]^{\C(f,M)}
\\
&\C(X,M')\ar[r]_-{\C(X,\varphi)}
&\C(X,M)
&
}$$
est commutatif.
\end{paragr}

\begin{paragr}\label{DoubleFonct}
On peut combiner les fonctorialités des paragraphes~\ref{FonctCf} et~\ref{FonctCphi} comme suit. La \emph{catégorie des \SLFS{} dans $\ooCat$} est la catégorie ayant comme objets les couples $(X,M)$, où $X$ est une $\infty$\nbd-catégorie et $M$ un \SLF{} sur~$X$, un morphisme $(X',M')\to(X,M)$ étant un couple $(f,\varphi)$, où $f:X'\to X$ est un $\infty$\nbd-foncteur et $\varphi:M'\to f^*(M)$ une transformation naturelle, la composition de deux morphismes composables étant définie par la formule $(f,\varphi)\circ(f'\kern-2pt,\varphi')=(ff'\kern-2pt,f'^*(\varphi)\varphi')$.
\smallbreak

Si $(f,\varphi):(X'\kern-2pt, M')\to(X,M)$ est un morphisme de \SLFS{} dans $\ooCat$, le composé
$$
\xymatrixcolsep{3.5pc}
\xymatrix{
\C(X'\kern -2pt,M')\ar[r]^-{\C(X'\kern -3pt,\kern1pt\varphi)}
&\C(X'\kern -2pt,f^*(M))\ar[r]^-{\C(f,M)}
&\C(X,M)
}
$$
définit un morphisme de complexes de groupes abéliens 
$$
\C(f,\varphi):\C(X'\kern -2pt,M')\to\C(X,M)\,.
$$ 
Explicitement, ce morphisme est induit par les morphismes de groupes abéliens
$$
\textstyle\bigoplus\limits_{x'\in X'_n}\kern -6pt M'_{t_0x'}\kern-1pt\to\kern-2pt\textstyle\bigoplus\limits_{x\in X_n}\kern -6pt M_{t_0x}\,,\quad(x',m')\mapsto(f(x'),\varphi_{t_0x'}(m'))\,,\ x'\in X'_n\,,\ m'\in M'_{t_0x'}\,,\ n\geq0\,. 
$$ 
On laisse au lecteur le soin de vérifier que, grâce à la compatibilité du paragraphe précédent, on obtient ainsi un foncteur de la catégorie des \SLFS{} dans $\ooCat$ vers celle des complexes de groupes abéliens. Si $X'=X$ et $f=1_X$, on a $\C(1_X,\varphi)=\C(X,\varphi)$ (cf. paragraphe~\ref{FonctCphi}), et si $M'=f^*(M)$ et $\varphi=1_{f^*(M)}$, on a $\C(f,1_{f^*(M)})=\C(f,M)$ (cf. paragraphe~\ref{FonctCf}).
\end{paragr}

\begin{paragr}
On rappelle qu'étant donnés deux $\infty$\nbd-foncteurs $f_0,f_1:X'\to X$ une \emph{transformation oplax}, ou plus simplement \emph{transformation}, de $f_0$ vers $f_1$ est un $\infty$-foncteur $\alpha:D_1\otimes X'\to X$ tel que $f_k=\alpha\partial_{X'}^k$, $k=0,1$, où $\otimes$ désigne le produit de Gray $\infty$\nbd-catégorique~\cite{AglSt},~\cite{Crans},~\cite[appendice~A]{DGjoint}, $D_1$ la catégorie $\{0\to1\}$, et $\partial_{X'}^k$ le $\infty$\nbd-foncteur
$$\xymatrix{
X'\simeq D_0\otimes X'\ar[r]^-{k\otimes X'}
&D_1\otimes X'
}$$
($D_0$ la catégorie ponctuelle et $k:D_0\to D_1$ le foncteur défini par l'objet $k$ de $D_1$). Il revient au même (voir par exemple~\cite[appendice~B]{DGjoint}) de se donner, pour tout $n\geq0$ et toute $n$\nbd-cellule $x$ de $X'$, une $(n+1)$\nbd-cellule $\transf{\alpha}{x}$ de $X$ satisfaisant aux conditions suivantes:
\smallbreak

0) \textsc{Compatibilité aux sources et aux buts.} Pour tout $n\geq0$, toute $n$\nbd-cellule $x$ de $X'$ et tout $i$, $0\leq i\leq n$, on a
$$\begin{aligned}
&s_i\transf{\alpha}{x}=\transf{\alpha}{t_{i-1}x}\comp{i-1}\cdots\comp{1}\transf{\alpha}{t_{0}x}\comp{0}f_0(s_ix)\,,\cr
\noalign{\vskip 3pt}
&t_i\transf{\alpha}{x}=f_1(t_ix)\comp{0}\transf{\alpha}{s_{0}x}\comp{1}\cdots\comp{i-1}\transf{\alpha}{s_{i-1}x}\ .
\end{aligned}$$
En particulier, pour $i=0$, $s_0\transf{\alpha}{x}=f_0(s_0x)$ et $t_0\transf{\alpha}{x}=f_1(t_0x)$, et si $n=0$, c'est-à-dire si $x$ est un objet de $X'$, 
alors $\transf{\alpha}{x}$ est une $1$\nbd-flèche de $X$ de source $f_0(x)$ et but $f_1(x)$. Pour $i=1$ et $n\geq1$, on a $s_1\transf{\alpha}{x}=\transf{\alpha}{t_0x}\comp{0}f_0(s_1x)$ et $t_1\transf{\alpha}{x}=f_1(t_1x)\comp{0}\transf{\alpha}{s_{0}x}$, et si $n=1$, c'est-à-dire si $x:x_0\to x_1$ est une $1$\nbd-flèche de $X'$, alors $\transf{\alpha}{x}$ est une $2$\nbd-flèche de $X$ de source $\transf{\alpha}{x_1}\comp{0}f_0(x)$ et but $f_1(x)\comp{0}\transf{\alpha}{x_0}$.
\smallbreak

1) \textsc{Compatibilité aux unités.} Pour tout $n\geq0$ et toute $n$\nbd-cellule $x$ de $X'$, on a $\transf{\alpha}{1_x}=1_{\transf{\alpha}{x}}$.
\smallbreak

2) \textsc{Compatibilité aux compositions.} Pour tout $n\geq0$, tout $i$, $0\leq i<n$, et tout couple de $n$\nbd-cellules $i$\nbd-composables $x_0,x_1$ de $X'$, on a
$$\transf{\alpha}{x_1\comp{\kern -1pti}x_0}=f_1(t_{i+1}x_1)\kern-2pt\comp{0}\kern-2pt\transf{\alpha}{s_0}\kern-4pt\comp{1}\kern-2pt\cdots\kern-1pt\comp{\kern -1pt i-1}\kern-2pt\transf{\alpha}{s_{i-1}}\kern-4pt\comp{\kern -1pt i}\kern-2pt\transf{\alpha}{x_0}\kern-2pt\comp{\kern -1pt i+1}\kern-2pt\transf{\alpha}{x_1}\kern-2pt\comp{\kern -1pt i}\kern-2pt\transf{\alpha}{t_{i-1}}\kern-4pt\comp{\kern -1pt i-1}\kern-2pt\cdots\kern-1pt\comp{1}\kern-2pt\transf{\alpha}{t_0}\kern-4pt\comp{0}\kern-2ptf_0(s_{i+1}x_0),$$
où pour tout $k$, $0\leq k<i$, on note
$$s_k=s_kx_0=s_kx_1\qquad\hbox{et}\qquad t_k=t_kx_0=t_kx_1\,.$$
En particulier, pour $i=0$, on a 
$$
\transf{\alpha}{x_1\comp{0}x_0}=(f_1(t_1x_1)\comp{0}\transf{\alpha}{x_0})\comp{1}(\transf{\alpha}{x_1}\comp{0}f_0(s_1x_0))\,.
$$

Si $g$ est un $\infty$-foncteur de source $X$, le composé du $\infty$\nbd-foncteur $\alpha:X'\otimes D_1\to X$ avec $g$ définit une transformation, notée $g*\alpha$, de $gf_0$ vers $gf_1$. De même, si $h$ est un $\infty$\nbd-foncteur de but $X'$, le composé $\alpha(h\otimes1_{D_1})$ définit une transformation, notée $\alpha*h$ de $f_0h$ vers $f_1h$.
\end{paragr}

\begin{lemme}\label{reltransf}
Soient $P$ un polygraphe, $q:X'\to X$ une fibration triviale pour la structure folk sur $\ooCat$ \emph{(cf.~paragraphe~\ref{strfolk})}, $f'_0,f'_1:P\to X'$ deux $\infty$\nbd-foncteurs,\penalty -500{} et $\alpha$ une transformation oplax de $f_0:=qf'_0$ vers $f_1:=qf'_1$. Alors il existe une transformation oplax $\alpha'$ de $f'_0$ vers $f'_1$ telle que $\alpha=q*\alpha'$.
\end{lemme}

\begin{proof}
Considérons le carré commutatif
$$
\xymatrixcolsep{3pc}
\xymatrix{
&P\amalg P\ar[d]\ar[r]^-{(f'_0,f'_1)}
&X'\ar[d]^q
\\
&P\otimes D_1\ar[r]_-{\alpha}
&X
&\kern -30pt,\kern30pt
}$$
où la flèche verticale de gauche est l'inclusion canonique, qui s'identifie au produit de Gray de l'inclusion $S^0=D_0\amalg D_0\hookrightarrow D_1$ par l'objet cofibrant $P$, et est donc, en vertu du théorème~5.6 de~\cite{DimMax}, une cofibration de la structure folk. Comme $q$ est une fibration triviale pour cette structure, ce carré admet un relèvement $\alpha':P\otimes D_1\to X'$, qui est une transformation oplax de $f'_0$ vers $f'_1$ telle que $\alpha=q*\alpha'$. (Pour plus de détails voir~\cite[Lemma~6.5]{Leonard}.)
\end{proof}

\begin{paragr}\label{paragr:prep}
Soient $f_0,f_1:X'\to X$ deux $\infty$\nbd-foncteurs, $\alpha:f_0\Rightarrow f_1$ une transformation oplax 
de $f_0$ vers $f_1$, et $M$ un \SLF{} sur $X$. Le composé $\varphi:=M\ast\alpha$ définit un morphisme de \SLFS{} sur $X'$ (cf.~paragraphe~\ref{FonctCphi})
$$
\varphi:f_1^*(M)\to f_0^*(M)\,.
$$ 
En effet, en vertu des formules de compatibilité de $\alpha$ aux sources et buts, pour tout objet $x'$ de $X'$, $\transf{\alpha}{x'}$ est une $1$\nbd-flèche de $X$ de source $f_0(x')$ et but $f_1(x')$, d'où un morphisme de groupes abéliens $\varphi_{x'}:=\alpha^*_{x'}:=M_{{\alpha}_{x'}}:M_{f_1(x')}\to M_{f_0(x')}$. De même, pour toute $1$\nbd-flèche $x':x'_0\to x'_1$ de $X'$, $\transf{\alpha}{x'}$ est une $2$\nbd-flèche de $X$ de source $\transf{\alpha}{x'_1}\comp{0}f_0(x')$ et but $f_1(x')\comp{0}\transf{\alpha}{x'_0}$. L'existence de cette $2$\nbd-flèche implique que $M_{{\alpha}_{x'_1}\comp{0}f_0(x')}=M_{f_1(x')\comp{0}{\alpha}_{x'_0}}$, d'où en vertu de la fonctorialité contravariante de $M$, la commutativité du carré
$$
\xymatrixcolsep{4pc}
\xymatrix{
&f_1^*(M)_{x'_1}\ar[r]^{\varphi_{x'_1}=M_{{\alpha}_{x'_1}}}\ar[d]_{x'^*=f_1^*(M)_{x'}=M_{f_1(x')}}
&f_0^*(M)_{x'_1}\ar[d]^{x'^*=f_0^*(M)_{x'}=M_{f_0(x')}}
\\
&f_1^*(M)_{x'_0}\ar[r]_{\varphi_{x'_0}=M_{{\alpha}_{x'_0}}}
&f_0^*(M)_{x'_0}
&\kern 13pt.
}
$$
On en déduit un morphisme de \SLFS{} dans $\ooCat$ 
$$
(f_0,\varphi):(X',f_1^*(M))\to(X,M)\,,
$$
d'où (cf.~paragraphe~\ref{DoubleFonct}) un morphisme de complexes 
$$
\C(f_0,\varphi):\C(X'\kern -2pt,f_1^*(M))\to\C(X,M)
$$
composé des morphismes
$$
\xymatrixcolsep{3.5pc}
\xymatrix{
\C(X'\kern -2pt,f_1^*(M))\ar[r]^-{\C(X'\kern -3pt,\kern1pt\varphi)}
&\C(X'\kern -2pt,f_0^*(M))\ar[r]^-{\C(f_0,M)}
&\C(X,M)\,.
}
$$
Par ailleurs, on a aussi un morphisme de complexes de même source et même but
$$
\C(f_1,M):\C(X'\kern -2pt,f_1^*(M))\to\C(X,M)\,.
$$
\end{paragr}

\begin{lemme}\label{lemme:3.12}
En gardant les hypothèses et notations ci-dessus, on a une homotopie de morphismes de complexes $h$ de $\C(f_0,\varphi)$ vers $\C(f_1,M)$, définie par
$$
\C_n(X'\kern -2pt,f_1^*(M))\to\C_{n+1}(X,M)\,,\qquad(x',m)\mapsto(\transf{\alpha}{x'},m)\,,\qquad n\geq0\,.
$$
\end{lemme}

\begin{proof}
Posons $M'=f_1^*(M)$. On remarque d'abord que pour tout $n\geq0$, tout $x'\in X'_n$, et tout $m\in M'_{t_0x'}$, comme en vertu des formules de compatibilité d'une transformation aux buts $t_0\transf{\alpha}{x'}=f_1(t_0x')$ et comme par définition $M'_{t_0x'}=M_{f_1(t_0x')}$, on a $m\in M_{t_0\transf{\alpha}{x'}}$, et l'application $(x',m)\mapsto(\transf{\alpha}{x'},m)$ définit donc un morphisme de groupes abéliens
$$
\textstyle\bigoplus\limits_{x'\in X'_n}\kern -2pt M'_{t_0x'}\longrightarrow\kern -5pt\textstyle\bigoplus\limits_{x\in X_{n+1}}\kern -5pt M_{t_0x}\ .
$$
Montrons que ce morphisme est compatible aux relations $(*_i)$ du paragraphe~\ref{relfond}. Soient donc $0\leq i<n$ et $x'_0,x'_1$ deux $n$\nbd-cellules $i$\nbd-composables de $X'$. On distingue deux cas:
\smallbreak

--- $i=0$. En vertu de la formule de compatibilité d'une transformation à la $0$\nbd-composition, la relation $(*_1)$, et la remarque~\ref{remunit}, on a
$$\begin{aligned}
h_n(x'_1\comp{0}x'_0,m)&=(\transf{\alpha}{x'_1\comp{0}x'_0},m)\cr
&=((f_1(t_1x'_1)\comp{0}\transf{\alpha}{x'_0})\comp{1}(\transf{\alpha}{x'_1}\comp{0}f_0(s_1x'_0)),m)\cr
&\sim_{n+1}(f_1(t_1x'_1)\comp{0}\transf{\alpha}{x'_0},m)+(\transf{\alpha}{x'_1}\comp{0}f_0(s_1x'_0),m)\cr
&\sim_{n+1}(\transf{\alpha}{x'_0},f_1(t_1x'_1)^*(m))+(\transf{\alpha}{x'_1},m)\cr
&=(\transf{\alpha}{x'_1},m)+(\transf{\alpha}{x'_0},t_1(x'_1)^*(m))\cr
&=h_n((x'_1,m)+(x'_0, t_1(x'_1)^*(m)))
\end{aligned}$$
(en se souvenant, pour l'avant dernière égalité, que $f_1(t_1x'_1)^*(m)=t_1(x'_1)^*(m)$, où \og l'étoile en haut\fg{} du membre de gauche est relative au \SLF{} $M$ sur~$X$, tandis que celle du membre de droite est relative au \SLF{} $M'$ sur $X'$).
\smallbreak

--- $i>0$. En vertu de la formule de compatibilité d'une transformation à la $i$\nbd-composition, la relation $(*_{i+1})$, et la remarque~\ref{remunit}, on a
$$\begin{aligned}
h_n(x'_1\comp{i}x'_0,m)&=(\transf{\alpha}{x'_1\comp{i}x'_0},m)\cr
&\sim_{n+1}(f_1(t_{i+1}x'_1)\comp{0}\transf{\alpha}{s_0x'_0}\comp{1}\cdots\comp{i-1}\transf{\alpha}{s_{i-1}x'_0}\comp{i}\transf{\alpha}{x'_0},\,m)\cr
&\kern 16pt+\,(\transf{\alpha}{x'_1}\comp{i}\transf{\alpha}{t_{i-1}x'_1}\comp{i-1}\cdots\comp{1}\transf{\alpha}{t_0x'_1}\comp{0}f_0(s_{i+1}x'_0),\,m)\cr
&\sim_{n+1}(\transf{\alpha}{x'_0},m)+(\transf{\alpha}{x'_1},m)=h_n((x'_1,m)+(x'_0,m))\,.
\end{aligned}$$
Il reste à montrer que, pour $n\geq0$, on a 
$$d_{n+1}h_n+h_{n-1}d_n=\C_n(f_1,M)-\C_n(f_0,\varphi)$$
(avec la convention $d_0=0$ et $h_{-1}=0$). On distingue trois cas:
\smallbreak

--- $n=0$, $x'$ objet de $X'$, $m\in M'_{x'}=M_{f_1x'}$. On a
$$\begin{aligned}
d_1h_0(x',m)&=d_1(\transf{\alpha}{x'},m)=(t_0\transf{\alpha}{x'},m)-(s_0\transf{\alpha}{x'},\alpha_{x'}^*(m))\cr
&=(f_1(x'),m)-(f_0(x'),\varphi_{x'}(m))=(\C_0(f_1,M)-\C_0(f_0,\varphi))(x',m)\,.
\end{aligned}$$

--- $n=1$, $x'$ $1$\nbd-flèche de $X'$, $m\in M'_{t_0x'}$. En vertu de la formule de compatibilité d'une transformation aux sources et buts et de la relation $*_0$, on a
$$\begin{aligned}
(d_{2}h_1+h_{0}d_1)(x',m)&=d_2(\transf{\alpha}{x'},m)+h_0((t_0x',m)-(s_0x',x'{}^*(m)))\cr
&=(t_1\transf{\alpha}{x'},m)-(s_1\transf{\alpha}{x'},m)+(\transf{\alpha}{t_0x'},m)-(\transf{\alpha}{s_0x'},x'{}^*(m))\cr
&=(f_1(x')*_0\transf{\alpha}{s_0x'},m)-(\transf{\alpha}{t_0x'}*_0f_0(x'),m)\cr
&\qquad+(\transf{\alpha}{t_0x'},m)-(\transf{\alpha}{s_0x'},x'{}^*(m))\cr
&\sim_1(f_1(x'),m)+(\transf{\alpha}{s_0x'},f_1(x')^*(m))\cr
&\qquad-(\transf{\alpha}{t_0x'},m)-(f_0(x'),\alpha_{t_0x'}^*(m))\cr
&\qquad+(\transf{\alpha}{t_0x'},m)-(\transf{\alpha}{s_0x'},x'{}^*(m))\cr
&=(f_1(x'),m)-(f_0(x'),\varphi^{}_{t_0x'}(m))\cr
&=(\C_1(f_1,M)-\C_1(f_0,\varphi))(x',m)\,.
\end{aligned}$$

--- $n>1$, $x'$ $n$\nbd-flèche de $X'$, $m\in M'_{t_0x'}$. En vertu de la formule de compatibilité d'une transformation aux sources et buts, de la relation $*_{n-1}$, et de la remarque~\ref{remunit}, on a
$$\begin{aligned}
(d_{n+1}h_n+h_{n-1}d_n)(x',m)&=d_{n+1}(\transf{\alpha}{x'},m)+h_{n-1}((t_{n-1}x',m)-(s_{n-1}x',m))\cr
&=(t_n\transf{\alpha}{x'},m)-(s_n\transf{\alpha}{x'},m)+(\transf{\alpha}{t_{n-1}x'},m)-(\transf{\alpha}{s_{n-1}x'},m)\cr
&=(f_1(x')\comp{0}\transf{\alpha}{s_{0}x'}\comp{1}\cdots\comp{n-1}\transf{\alpha}{s_{n-1}x'},m)\cr
&\qquad-(\transf{\alpha}{t_{n-1}x'}\comp{n-1}\cdots\comp{1}\transf{\alpha}{t_{0}x'}\comp{0}f_0(x'),m)\cr
&\qquad+(\transf{\alpha}{t_{n-1}x'},m)-(\transf{\alpha}{s_{n-1}x'},m)\cr
&\sim_n(f_1(x'),m)+(\transf{\alpha}{s_{n-1}x'},m)\cr
&\qquad-(\transf{\alpha}{t_{n-1}x'},m)-(f_0(x'),\alpha_{t_0x'}^*(m))\cr
&\qquad+(\transf{\alpha}{t_{n-1}x'},m)-(\transf{\alpha}{s_{n-1}x'},m)\cr
&=(f_1(x'),m)-(f_0(x'),\varphi^{}_{t_0x'}(m))\cr
&=(\C_n(f_1,M)-\C_n(f_0,\varphi))(x',m)\,,
\end{aligned}$$
ce qui achève la démonstration du lemme.
\end{proof}

\begin{prop}\label{propqis}
Soient $f_0,f_1:X'\to X$ deux $\infty$\nbd-foncteurs, $\alpha:f_0\Rightarrow f_1$ une transformation oplax 
de $f_0$ vers $f_1$, et $M$ un \SLF{} sur $X$. On suppose que pour tout objet $x'$ de $X'$, le morphisme de groupes abéliens $\alpha^*_{x'}:=M_{{\alpha}_{x'}}:M_{f_1(x')}\to M_{f_0(x')}$ est un isomorphisme. Alors le morphisme de complexes $\C(f_0,M):\C(X',f_0^*(M))\to\C(X,M)$ est un quasi-isomorphisme si et seulement si le morphisme $\C(f_1,M):\C(X',f_1^*(M))\to\C(X,M)$ l'est.
\end{prop}

\begin{proof}
En vertu du lemme \ref{lemme:3.12} et dans les notations du paragraphe~\ref{paragr:prep}, il existe une homotopie entre les morphismes de complexes $\C(f_1,M)$ et $\C(f_0,\varphi)=\C(f_0,M)\circ\C(X',\varphi)$. On en déduit que $\C(f_1,M)$ est un quasi-isomorphisme si et seulement si $\C(f_0,\varphi)$ l'est. Or, l'hypothèse de la proposition implique aussitôt que $\C(X',\varphi)$ est un isomorphisme, ce qui achève la démonstration.
\end{proof}

\begin{cor}\label{retrqis}
Soient $j:X'\to X$ un $\infty$\nbd-foncteur et $M$ un \SLF{} sur $X$. On suppose que $j$ admet une rétraction $r$ et qu'il existe une transformation oplax $\alpha$ de $jr$ vers $1_X$ telle que pour tout objet $x$ de $X$, $\alpha^*_x:=M_{\alpha_x}:M_{jr(x)}\to M_x$ soit un isomorphisme. Alors le morphisme de complexes $\C(j,M):\C(X',j^*(M))\to\C(X,M)$ est un quasi-isomorphisme.
\end{cor}

\begin{proof}
En vertu de la proposition \ref{propqis}, l'hypothèse du corollaire implique que le morphisme de complexes $\C(jr,M)=\C(j,M)\circ\C(r,j^*(M))$ est un quasi-isomorphisme. Or, comme $r$ est une rétraction de $j$, on a l'égalité $1_{\C(X'\kern-2pt,\kern1ptj^*(M))}=\C(r,j^*(M))\circ\C(j,r^*j^*(M))$. On en déduit que $\C(r,j^*(M))$ est un quasi-isomorphisme, donc $\C(j,M)$ aussi.
\end{proof}

\begin{rem}\label{remast}
L'hypothèse sur la transformation $\alpha$ dans la proposition \ref{propqis} et dans le corollaire \ref{retrqis} est satisfaite par exemple si $M$ est un \SL, et pas seulement un \SLF, ou si pour tout objet $x'$ de $X'$ l'image de $\transf{\alpha}{x'}$ dans la catégorie fondamentale de $X$ ($1$\nbd-tronqué intelligent de $X$) est un isomorphisme.
\end{rem}

\begin{thm}\label{th:Quilleng}
Soient $X$ une $\infty$\nbd-catégorie et $M$ un \SLF. Le foncteur $\F=\F_{X,M}:\cm{\ooCat}{X}\to\Comp(\Ab)$ de la catégorie des $\infty$\nbd-catégories au-dessus de $X$ vers celle des complexes de groupes abéliens, associant à $(X',p:X'\to X)$ le complexe $\C(X',p^*(M))$ est un foncteur de Quillen à gauche pour la structure de catégories de modèles sur $\cm{\ooCat}{X}$ induite de la structure folk sur $\ooCat$ et la structure injective sur $\Comp(\Ab)$ \emph{(ayant les quasi-isomorphismes comme équivalences faibles et les monomorphismes comme cofibrations)}.  
\end{thm}

Ce théorème résulte des trois lemmes suivants.

\begin{lemme}\label{lemme:3.17}
Soient $X$ une $\infty$\nbd-catégorie et $M$ un \SLF. Le foncteur
$$
\F=\F_{X,M}:\cm{\ooCat}{X}\to\Comp(\Ab)\,,\qquad (X',p:X'\to X)\mapsto\C(X',p^*(M))\,,
$$
admet un adjoint à droite.
\end{lemme}

\begin{proof}
Voir appendice~\ref{adj}.
\end{proof}

\begin{lemme}\label{lemme:cof}
En gardant les hypothèses et les notations du lemme précédent, le foncteur $\F$ respecte les cofibrations.
\end{lemme}

\begin{proof}
Comme, en vertu du lemme \ref{lemme:3.17}, le foncteur $\F$ admet un adjoint à droite, il commute aux limites inductives et par suite il suffit de montrer qu'il transforme les cofibrations génératrices en cofibrations. On rappelle que les cofibrations génératrices de la structure folk sur $\ooCat$ sont les inclusions $S^{n-1}\hookrightarrow D_n$, $n\geq 0$, où $D_n$ est la $n$\nbd-catégorie ayant une seule cellule non triviale $x$ en dimension $n$ et exactement deux cellules non triviales en dimension $0\leq i<n$, à savoir $s_ix$ et~$t_ix$, et~$S^{n-1}$ la sous-$n$\nbd-catégorie de $D_n$, ayant les mêmes cellules que $D_n$ en dimension~$<n$, et pas de cellule non triviale en dimension $n$~\cite{LMW}. On en déduit que les cofibrations génératrices de la structure induite sur $\cm{\ooCat}{X}$ sont les inclusions $S^{n-1}\hookrightarrow D_n$ \emph{au-dessus de $X$}, indexées par $n\geq0$ et $p:D_n\to X$, $\infty$\nbd-foncteur arbitraire. Comme il n'y a aucune composition non triviale dans $D_n$ et $S^{n-1}$, il résulte de la remarque~\ref{remunit} que l'image par le foncteur $\F$ d'une telle inclusion s'identifie à l'inclusion des complexes $N'\hookrightarrow N$, où, en posant $M'=p^*(M)$, le complexe $N$ est défini par
$$\begin{aligned}
&N_0=M'_{t_0x}\oplus M'_{s_0x}\,,\ N_i=M'_{t_0x}\oplus M'_{t_0x}\,,\ 0<i<n\,,\ N_n=M'_{t_0x}\,,\ N_i=0\,,\ i>n\,,\cr
&d_1=
\begin{pmatrix}
1 &1
\\
-t_1(x)^* &-t_1(x)^*
\end{pmatrix},\quad
d_i=
\begin{pmatrix}
1 &1
\\
-1 &-1
\end{pmatrix},\ 1<i<n\,,\quad
d_n=
\begin{pmatrix}
1 
\\
-1
\end{pmatrix},
\end{aligned}$$
(avec les adaptations évidentes de ces formules pour $n=0,1$) et $N'$ est le sous-complexe de $N$ tel que $N'_n=0$ et $N'_i=N_i$, pour $0\leq i<n$, ce qui achève la démonstration.
\end{proof}

\begin{lemme}
En gardant les notations et les hypothèses des lemmes précédents, le foncteur $\F$ respecte les cofibrations triviales.
\end{lemme}

\begin{proof}
Comme en vertu des lemmes \ref{lemme:3.17} et \ref{lemme:cof} le foncteur $\F$ commute aux limites inductives et respecte les cofibrations, il suffit de montrer qu'il transforme les cofibrations triviales génératrices en équivalences faibles. 
\smallbreak

On rappelle que les cofibrations triviales génératrices de la structure folk sur $\ooCat$ sont des inclusions de la forme $j_n:D_n\to J_{n+1}$, $n\geq0$, admettant une rétraction $r_n$ qui est aussi une fibration triviale~\cite{LMW}. De plus, on peut choisir le polygraphe $J_{n+1}$ de sorte que son $(n+2)$\nbd-tronqué bête soit le $(n+2)$\nbd-polygraphe obtenu de $D_{n+1}$ en adjoignant une $(n+1)$\nbd-flèche $v$ dans le sens opposé de celui de sa cellule principale $u:x\to y$ et deux $(n+2)$\nbd-flèches $\varepsilon:v*_nu\to1_x$ et $\eta:1_y\to u*_nv$~\cite[Remark~20.4.4]{Poly}. Le morphisme $j_n$ est alors induit par la cosource $\sigma:D_n\to D_{n+1}$. On en déduit que pour $n=0$, la catégorie fondamentale $c_1(J_1)$ de $J_1$ est un groupoïde trivial, et que pour $n>1$, les objets de $J_{n+1}$ n'ont pas de $1$\nbd-endomorphismes non triviaux et que $j_nr_n$ induit l'identité sur les objets.
\smallbreak

La famille des inclusions $j_n:D_n\to J_{n+1}$ \emph{au-dessus de $X$}, indexée par $n\geq0$ et $p:J_{n+1}\to X$, $\infty$\nbd-foncteur arbitraire, engendre les cofibrations triviales de $\cm{\ooCat}{X}$. Il s'agit de montrer que $\C(j_n,p^*(M))$ est un quasi-isomorphisme. Or, comme $J_{n+1}$ est un polygraphe et la rétraction $r_n$ de $j_n$ une fibration triviale, en vertu du lemme~\ref{reltransf}, la transformation oplax identique de $r_n$ se relève (en remarquant que $r_nj_nr_n=r_n$) en une transformation oplax $\alpha$ de $j_nr_n$ vers $1_{J_{n+1}}$. Pour conclure, il suffit de vérifier que $\alpha$ satisfait aux hypothèses du corollaire~\ref{retrqis} (relativement à $p^*(M)$). Si $n>0$, comme $j_nr_n$ induit l'identité sur les objets et comme les objets de $J_{n+1}$ n'admettent pas de $1$\nbd-endomorphismes non triviaux, pour tout objet~$x$ de $J_{n+1}$ on a $\transf{\alpha}{x}=1_x$, et à plus forte raison l'hypothèse du corollaire est satisfaite. Cette hypothèse est également satisfaite pour $n=0$, puisque la catégorie fondamentale de~$J_1$ est un groupoïde (cf.~remarque~\ref{remast}).
\end{proof}

\begin{cor}\label{cor:3.20}
Soient $X$ une $\infty$\nbd-catégorie et $M$ un \SLF. Le foncteur $\F=\F_{X,M}:\cm{\ooCat}{X}\to\Comp(\Ab)$ de la catégorie des $\infty$\nbd-catégories au-dessus de $X$, munie des équivalences faibles induites de celles de la structure folk, vers celle des complexes de groupes abéliens, munie des quasi-isomorphismes, associant à $(X',p:X'\to X)$ le complexe $\C(X',p^*(M))$ admet un foncteur dérivé total à gauche
$$
\LL F=\LL \F_{X,M}:\Ho(\cm{\ooCat}{X})\longrightarrow\Ho(\Comp(\Ab))=\Der(\Ab)
$$
de la catégorie homotopique de $\cm{\ooCat}{X}$ vers la catégorie dérivée des groupes abéliens.
\end{cor}

\begin{proof}
Le corollaire est conséquence directe du théorème~\ref{th:Quilleng} et du théorème d'existence des foncteurs dérivés (voir par exemple~\cite[Theorem~8.5.8]{Hi}). 
\end{proof}

\begin{paragr}\label{def:Hp}
Soient $X$ une $\infty$\nbd-catégorie et $M:X^\circ\to\Ab$ un \SLF{} sur~$X$. On définit l'\emph{homologie polygraphique $\Hp{n}(X,M)$ de $X$ à coefficients dans le \SLF} $M$ comme suit. On choisit une \emph{résolution polygraphique} $p:P\to X$ de $X$, autrement dit, une \emph{fibration triviale} de source un objet \emph{cofibrant} pour la structure de catégorie de modèles folk sur la catégorie des $\infty$-catégories strictes, et pour tout $n\geq0$, on pose $\Hp{n}(X,M)=\H_n(\C(P,p^*(M)))$. On rappelle que les objets cofibrants pour cette structure sont les polygraphes, d'où le nom d'homologie polygraphique. On remarque que si $M$ est le foncteur constant de valeur le groupe $\mathbb Z$, on retrouve la définition de l'homologie polygraphique d'une $\infty$\nbd-catégorie introduite par François Métayer dans~\cite{FrMet}. De même, on définit l'\emph{homologie polygraphique totale de $X$ à coefficients dans le \SLF{} $M$} comme étant l'objet $\HP(X,M)$ de la catégorie dérivée des groupes abéliens représenté par le complexe $\C(P,p^*(M))$.
\smallbreak

En vertu du corollaire \ref{cor:3.20} et de la construction du foncteur dérivé total à gauche d'un foncteur de Quillen à gauche~\cite[§8.5]{Hi},~\cite[§1.3]{Ho}, l'homologie polygraphique $\Hp{n}(X,M)$ et l'homologie polygraphique totale définies ci-dessus sont indépendantes, à isomorphisme canonique près, du choix de la résolution polygraphique $p:P\to M$. Plus précisément, en gardant les notations dudit corollaire, on a un isomorphisme canonique dans la catégorie dérivée des groupes abéliens $\HP(X,M)\simeq\LL \F_{X,M}(X,1_X)$, 
et en particulier pour $n\geq0$, un isomorphisme canonique de groupes abéliens $\Hp{n}(X,M)\simeq\H_n(\HP(X,M))$, où $\H_n:\Der(\Ab)\to\Ab$ désigne le foncteur homologie de degré $n$. On remarque que $\H_n(\HP(X,M))=0$, pour $n<0$.
\end{paragr}
\goodbreak

On va maintenant s'intéresser à la fonctorialité de l'homologie polygraphique $\HP(X,M)$ en $X$ et en $M$.

\begin{paragr}\label{fonctHpX}
Soient $f:X'\to X$ un $\infty$\nbd-foncteur et $U=U_f$ le foncteur \og d'oubli par $f$\fg
$$
U:\cm{\ooCat}{X'}\to\cm{\ooCat}{X}\,,\qquad(T,p:T\to X')\mapsto(T,fp:T\to X)
$$
entre les catégories $\cm{\ooCat}{X'}$ et $\cm{\ooCat}{X}$ munies chacune de la structure de catégorie de modèles induite par la structure folk sur $\ooCat$. En vertu de la définition des structures induites, ce foncteur respecte les équivalences faibles et les cofibrations et en particulier (comme il admet un adjoint à droite) c'est un foncteur de Quillen à gauche. Il admet donc un foncteur dérivé total à gauche $\LL U:\Ho(\cm{\ooCat}{X'})\to\Ho(\cm{\ooCat}{X})$ qui coïncide avec le foncteur $\bar U$ obtenu de $U$ par la propriété universelle de la localisation.
\smallbreak

Soit maintenant un \SLF{} $M$ sur $X$, et consdérons le triangle commutatif de foncteurs de Quillen à gauche
$$
\xymatrixcolsep{.2pc}
\xymatrix{
\cm{\ooCat}{X'}\ar[rr]^U\ar[dr]_{\F_{X'\kern -2pt,f^*(M)}}
&&\cm{\ooCat}{X}\ar[ld]^{\F_{X,M}}
\\
&\Comp(\Ab)
&.
}
$$
En vertu du théorème de Quillen de composition des foncteurs dérivés (voir par exemple~\cite[Theorem~1.3.7]{Ho}), on a un isomorphisme canonique de foncteurs
$$
\LL\F_{X'\kern -2pt,f^*(M)}\simeq\LL\F_{X,M}\circ\LL U=\LL\F_{X,M}\circ\bar U\,,
$$
et en particulier des isomorphismes canoniques dans la catégorie dérivée des groupes abéliens
$$
\HP(X'\kern -2pt,f^*(M))\simeq\LL\F_{X'\kern -2pt,f^*(M)}(X'\kern -2pt,1_{X'})\simeq\LL\F_{X,M}\circ\LL U(X'\kern -2pt,1_{X'})=\LL\F_{X,M}(X'\kern -2pt,f)\,.
$$
En composant avec le morphisme 
$$
\LL\F_{X,M}(X',f)\to\LL\F_{X,M}(X,1_X)\simeq\HP(X,M)
$$
image par le foncteur $\LL\F_{X,M}$ du morphisme $f:(X',f)\to(X,1_X)$ de $\cm{\ooCat}{X}$, on obtient un morphisme 
$$
\HP(f,M):\HP(X'\kern -2pt,f^*(M))\to\HP(X,M)\,\phantom{.} 
$$ 
dans $\Der(\Ab)$, et en particulier pour tout $n\geq0$, un morphisme de groupes abéliens
$$
\Hp{n}(f,M):\Hp{n}(X'\kern -2pt,f^*(M))\to\Hp{n}(X,M)\,. 
$$
On remarque que \smash{$\HP(1_X,M)=1_{\HP(X,M)}$}, et on laisse le soin au lecteur de vérifier que si $f':X''\to X'$ est un $\infty$\nbd-foncteur composable avec $f$, alors on a $$\HP(ff',M)=\HP(f,M)\circ\HP(f',f^*(M))\,.$$
\end{paragr}

\begin{rem}\label{retenir}
Ce qu'il faut retenir est que, dans les notations du paragraphe précédent, on a un isomorphisme canonique
$\HP(X'\kern -2pt,f^*(M))\simeq\LL\F_{X,M}(X',f)$ identifiant le morphisme $\HP(f,M):\HP(X'\kern -2pt,f^*(M))\to\HP(X,M)$ au morphisme
$$
\LL\F_{X,M}((X',f)\overset{f}{\to}(X,1_X))\,.
$$
\end{rem}

\begin{thm}\label{th:fond}
Soient $f:X'\to X$ un $\infty$\nbd-foncteur et $M$ un \SLF{} sur $X$. Si $f$ est une équivalence folk, alors $\HP(f,M):\HP(X'\kern -2pt,f^*(M))\to\HP(X,M)$ est un isomorphisme dans la catégorie dérivée des groupes abéliens, et en particulier, pour tout $n\geq0$, $\Hp{n}(f,M):\Hp{n}(X'\kern -2pt,f^*(M))\to\Hp{n}(X,M)$ est un isomorphisme de groupes abéliens.
\end{thm}

\begin{proof}
En vertu de la définition du morphisme $\HP(f,M)$, il suffit de montrer que l'image par $\LL\F_{X,M}$ du morphisme $f:(X',f)\to(X,1_X)$ de $\cm{\ooCat}{X}$ est un isomorphisme, ce qui résulte du fait que, par hypothèse, $f:(X',f)\to(X,1_X)$ est une équivalence faible de $\cm{\ooCat}{X}$.
\end{proof}

\begin{paragr}\label{fonctHpM}
Soient $X$ une $\infty$\nbd-catégorie, $M$ et $M'$ deux \SLFS{} sur $X$, et $\varphi:M'\to M$ un morphisme de \SLFS{} sur $X$ (cf.~paragraphe~\ref{FonctCphi}).
Pour tout objet $(X',p:X'\to X)$ de $\cm{\ooCat}{X}$, on définit une transformation naturelle $p^*(\varphi):p^*(M')\to p^*(M)$, en posant $p^*(\varphi)=\varphi\kern 1pt*p^\circ$, et par suite, un morphisme de complexes $\C(X',p^*(\varphi)):\C(X',p^*(M'))\to\C(X',p^*(M))$. Il résulte aussitôt du paragraphe~\ref{CompFonctCfphi} qu'on obtient ainsi une transformation naturelle $\F_{X,\varphi}:\F_{X,M'}\to\F_{X,M}$, et par $2$\nbd-fonctorialité des foncteurs dérivés~\cite[§1.3]{Ho}, une transformation naturelle 
$$
\LL\F_{X,\varphi}:\LL\F_{X,M'}\to\LL\F_{X,M}\,,
$$
d'où, en évaluant à l'objet $(X,1_X)$ de $\cm{\ooCat}{X}$, un morphisme 
$$
\HP(X,\varphi):\HP(X,M')\to\HP(X,M)
$$
dans la catégorie dérivée des groupes abéliens, et en particulier, pour tout $n\geq0$, un morphisme de groupes abéliens
 $$
\Hp{n}(X,\varphi):\Hp{n}(X,M')\to\Hp{n}(X,M)\,.
$$
La $2$\nbd-fonctorialité des foncteurs dérivés implique aussitôt que \smash{$\HP\kern-.6pt(X,\kern -1.2pt1_M\kern-.6pt)=1_{\HP(X,M)}$} et que si $\varphi':M''\to M'$ est un morphisme de \SLFS{} sur $X\vrule height 11.5pt depth 2.2pt width0pt$, composable avec $\varphi$, alors on a $\HP(X,\varphi\varphi')=\HP(X,\varphi)\circ\HP(X,\varphi')$. De même, si $f:X'\to X$ est un $\infty$\nbd-foncteur, on a (cf.~paragraphe~\ref{CompFonctCfphi} et remarque~\ref{retenir}) un diagramme commutatif
$$\xymatrixcolsep{4.5pc}
\xymatrix{
&\HP(X'\kern-2pt,f^*(M'))\ar[r]^-{\HP(X'\kern-2pt,f^*(\varphi))\,}\ar[d]_{\HP(f,M')}
&\HP(X'\kern-2pt,f^*(M))\ar[d]^{\HP(f,M)}
\\
&\HP(X,M')\ar[r]_-{\HP(X,\varphi)}
&\HP(X,M)
&\kern-50pt,\kern50pt
}$$
où $f^*(\varphi):f^*(M')\to f^*(M)$ désigne le morphisme de \SLFS{} sur $X'$ défini par la transformation naturelle $\varphi*f^\circ$.
\end{paragr}

\begin{paragr}
Soit $(f,\varphi):(X',M')\to(X,M)$ un morphisme de \SLFS{} dans $\ooCat$ (cf.~paragraphe~\ref{DoubleFonct}). On définit un morphisme de la catégorie dérivée des groupes abéliens
$$
\HP(f,\varphi):\HP(X'\kern-2pt,M')\to\HP(X,M)
$$
en posant
$$
\HP(f,\varphi)=\HP(f,M)\circ\HP(X'\kern-2pt,\varphi)\,.
$$
La commutativité du carré ci-dessus implique aussitôt que l'association
$$
(X,M)\mapsto\HP(X,M)\,,\qquad(f,\varphi)\mapsto\HP(f,\varphi)
$$
définit un foncteur de la catégorie des \SLFS{} dans $\ooCat$ vers la catégorie dérivée des groupes abéliens.
\end{paragr}

\begin{prop}\label{prop:3.27}
Soient $f_0,f_1:X'\to X$ deux $\infty$\nbd-foncteurs, $\alpha:f_0\Rightarrow f_1$ une transformation oplax 
de $f_0$ vers $f_1$, et $M$ un \SLF{} sur $X$. On suppose que pour tout objet $x'$ de $X'$, le morphisme de groupes abéliens $\alpha^*_{x'}:=M_{{\alpha}_{x'}}:M_{f_1(x')}\to M_{f_0(x')}$ est un isomorphisme. Alors $\HP(f_0,M)$ est un isomorphisme dans la catégorie dérivée des groupes abéliens si et seulement si $\HP(f_1,M)$ l'est.
\end{prop}

\begin{proof}
Choisissons des résolutions polygraphiques $p:P\to X$ et $p':P'\to X'$ de $X$ et $X'$ respectivement. Pour $\varepsilon=0,1$, le morphisme $\HP(f_\varepsilon,M)$ s'identifie, par définition (cf.~paragraphe~\ref{fonctHpX}) et construction des foncteurs dérivés, à l'image dans la catégorie dérivée des groupes abéliens du morphisme de complexes $\C(f'_\varepsilon,p^*(M))$, où $f'_\varepsilon:P'\to P$ désigne un relèvement de $f_\varepsilon$. En vertu du lemme~\ref{reltransf}, appliqué à la transformation $\alpha*p'$, il existe une transformation oplax $\alpha'$ de $f'_0$ vers~$f'_1$ telle que $p*\alpha'=\alpha*p'$. Or, l'hypothèse faite sur $\alpha$ implique que la transformation $\alpha'$ satisfait aux conditions de la proposition~\ref{propqis}. On en déduit que $\C(f'_0,M)$ est un quasi-isomorphisme si et seulement si $\C(f'_1,M)$ l'est, ce qui achève la démonstration.
\end{proof}

\begin{cor}\label{retris}
Soient $j:X'\to X$ un $\infty$\nbd-foncteur admettant une rétraction~$r$, et $M$ un \SLF{} sur $X$. On suppose qu'il existe une transformation oplax~$\alpha$ de $jr$ vers $1_X$ telle que pour tout objet $x$ de $X$, $\alpha^*_x:=M_{\alpha_x}:M_{jr(x)}\to M_x$ soit un isomorphisme. Alors les morphismes $\HP(j,M)$ et $\HP(r,j^*(M))$ de la catégorie dérivée des groupes abéliens sont des isomorphismes.
\end{cor}

\begin{proof}
En vertu de la proposition \ref{prop:3.27}, l'hypothèse du corollaire implique que le morphisme $\HP(jr,M)=\HP(j,M)\circ\HP(r,j^*(M))$ est un isomorphisme. Or, comme $r$ est une rétraction de $j$, on a l'égalité 
$$1_{\HP(X',j^*(M))}=\HP(r,j^*(M))\circ\HP(j,r^*j^*(M))\,.$$ 
On en déduit que $\HP(r,j^*(M))$ est un isomorphisme, donc $\HP(j,M)$ aussi.
\end{proof}

\begin{rem}\label{remastbis}
L'hypothèse sur la transformation $\alpha$ dans la proposition \ref{prop:3.27} et dans le corollaire~\ref{retris} est satisfaite par exemple si $M$ est un \SL, et pas seulement un \SLF.
\end{rem}

Dans le cas où $X$ est un polygraphe, on dispose d'une description plus précise du complexe $\C(X,M)$ calculant alors son homologie polygraphique à coefficients dans un \SLF{} $M$.

\begin{prop}\label{CasPol}
Soient $X$ une $\infty$\nbd-catégorie librement engendrée au sens des polygraphes par l'ensemble de cellules $B$, et $M$ un \SLF{} sur $X$. Alors pour tout $n\geq0$, le composé
$$
\textstyle\bigoplus\limits_{b\in B_n}M_{t_0b}\hookrightarrow\textstyle\bigoplus\limits_{x\in X_n}M_{t_0x}\to\C_n(X,M)\,,
$$
de l'inclusion canonique suivie par la surjection canonique, est un isomorphisme.
\end{prop}

\begin{proof}
L'assertion étant tautologique pour $n=0$, supposons que $n\geq1$.
Il suffit de montrer que pour tout groupe abélien $N$ l'application 
$$
\Hom_\Ab(\C_n(X,M),N)\longrightarrow\Hom_\Ab(\textstyle\bigoplus\limits_{b\in B_n}M_{t_0b},N)\simeq\textstyle\prod\limits_{b\in B_n}\Hom_\Ab(M_{t_0b},N)\,,
$$
induite par ce composé, est une bijection. Comme le $n$\nbd-tronqué bête d'un polygraphe est un polygraphe et que la formation du complexe $\C(X,M)$ commute aux troncations bêtes, on peut supposer que le polygraphe $X$ est une $n$\nbd-catégorie et que le complexe $\C(X,M)$ est nul en degré $>n$. On a alors une bijection canonique
$$
\Hom_\Ab(\C_n(X,M),N)\simeq\Hom_{\Comp(\Ab)}(\C(X,M),N[n])\,,
$$
où $N[n]$ désigne le complexe concentré en degré (homologique) $n$ de valeur $N$. Or, par adjonction, on a une bijection canonique
$$
\Hom_{\Comp(\Ab)}(\C(X,M),N[n])\simeq\Hom_{\cm{\ooCat}{X}}(X,\G N[n])\,,
$$
où $\G=\G_{X,M}$ désigne l'adjoint à droite décrit dans l'appendice~\ref{adj}. Dans ce cas particulier, $\G N[n]$ est une $n$\nbd-catégorie ayant même $(n-1)$\nbd-tronqué bête que $X$ et dont les $n$\nbd-cellules s'identifient aux couples $(x,u_n)$, avec $x$ $n$\nbd-cellule de $X$ et $u_n:M_{t_0x}\to N$. En vertu de la propriété universelle des polygraphes, se donner un $\infty$\nbd-foncteur de $X$ vers $\G N[n]$, induisant l'identité sur son $(n-1)$\nbd-tronqué bête, revient à se donner pour tout $b\in B_n$ une telle cellule $(x,u_n)$, de façon compatible aux sources et buts. Pour que ce $\infty$\nbd-foncteur soit au-dessus de $X$, il est \emph{nécessaire} que ce couple soit de la forme $(b,u_n)$, et alors la compatibilité aux sources et buts est automatique. Mais cette condition est également \emph{suffisante} puisqu'en composant ce $\infty$\nbd-foncteur avec la projection $\G N[n]\to X$ on obtient un $\infty$\nbd-foncteur qui induit l'identité sur le $(n-1)$\nbd-tronqué bête et sur les générateurs, ce qui achève la démonstration.
\end{proof}

\begin{rem}
La description de la différentielle en termes des générateurs du polygraphe n'est pas aussi simple que pour l'homologie polygraphique à coefficients constants. 
Dans la section suivante, on l'explicitera dans un cas particulier.
\end{rem}

\section{L'homologie polygraphique d'un système local simplicial}

\begin{paragr}
La théorie de l'homotopie des $\infty$\nbd-catégories strictes est basée sur le nerf de Street, généralisation du nerf usuel $1$\nbd-catégorique introduit par Grothendieck~\cite{Nerf}. Ross~Street a introduit dans~\cite{Or} un objet cosimplicial dans $\ooCat$, la catégorie cosimpliciale $\Or{}$ des orientaux. En basse dimension, on a:
$$\begin{aligned}
\smp{0}\mapsto\Or{0}&=\{0\}=\hbox{catégorie ponctuelle}\,,
\\
\smp{1}\mapsto\Or{1}&=\{0\longrightarrow1\}\,,
\\
\smp{2}\mapsto\Or{2}&=\left\{\raise 21pt\vbox{
\xymatrixrowsep{-.07pc}
\xymatrixcolsep{1.pc}
\xymatrix{
&1\ar[dddddr]
\\
\\
&
\\
&
\\
&\ar@{=>}[uu]^{}
\\
0\ar[uuuuur]^{}\ar[rr]_{}
&\vrule height 6pt width 0pt\ar@{=>}[uuu]^{}
&2
}
}\right\}\,,
\\
\smp{3}\mapsto\Or{3}&=\left\{\raise 30pt\vbox{
\UseTwocells
\xymatrixcolsep{.2pc}
\xymatrixrowsep{.0pc}
\xymatrix{
&&&1\ar[dddrrr]
&&&&&&&
&&&1\ar[dddrrr]\ar[dddddd]
\\&&&&&&&&&&
\\&&
&&\lltwocell<>
&&&&&&
&&&&&&
\\
0\ar[dddrrr]\ar[uuurrr]\ar[rrrrrr]
&&&
&&&2\ar[dddlll]
&\ar@3{->}[rr]
&&&
0\ar[dddrrr]\ar[uuurrr]
&&&&&&2\ar[dddlll]
\\&&&&&&&&&&
&&&\uulltwocell<>
&\uutwocell<>
\\
&&&\uultwocell<>
&&&
\\
&&&3
&&&&&&&
&&&3
}
}\right\}\,.
\end{aligned}$$
Ce qu'il faut retenir, dans le cas général, est que $\Or n$ est un $n$\nbd-polygraphe dont les $i$\nbd-cellules génératrices sont en bijection avec les $i$\nbd-simplexes non dégénérés de $\smp n$, la source (resp. le but) d'une telle cellule étant un composé des $(i-1)$-cellules correspondant aux faces impaires (resp.~paires) du $i$\nbd-simplexe. En particulier, $\Or n$ a $n+1$ objets notés $0,1,\dots,n$ et une unique $n$\nbd-cellule non triviale $\xi_n$, qu'on appellera sa cellule principale. Pour $0\leq j_0<j_1<\cdots<j_i\leq n$, on notera $(j_0j_1\dots j_i)$ à la fois le $i$-simplexe de $\smp n$ défini par l'application $k\mapsto j_k$ est la $i$\nbd-cellule correspondante de $\Or n$. 
Avec cette notation, on a donc $\xi_n=(01\dots n)$. La $0$\nbd-source de $\xi_n$ est $0$ et son $0$\nbd-but est $n$.
On aura besoin du lemme suivant.
\end{paragr}

\begin{lemme}\label{LemmeOr}
Pour $n\geq1$ pair, respectivement impair, la source et le but de la cellule principale de $\Or n$ sont des composés
$$
s_{n-1}\xi_n=A_{n-1}*_{n-2}\cdots*_{n-2}A_3*_{n-2}A_1\,,\quad t_{n-1}\xi_n=A_{0}*_{n-2}A_2*_{n-2}\cdots*_{n-2}A_n\,,
$$
respectivement
$$
s_{n-1}\xi_n=A_{n}*_{n-2}\cdots*_{n-2}A_3*_{n-2}A_1\,,\qquad t_{n-1}\xi_n=A_{0}*_{n-2}A_2*_{n-2}\cdots*_{n-2}A_{n-1}\,,
$$
où pour tout $i$, $A_i$ est une $(n-1)$\nbd-cellule de $\Or n$ de la forme
$$
A_i=L_{n-2}*_{n-3}\cdots*_2L_2*_1L_1*_0(0\dots\widehat{i}\dots n)*_0R_1*_1R_2*_2\dots*_{n-3}R_{n-2}\,,
$$
$L_k,R_k$ étant des $k$\nbd-cellules de $\Or n$, pour $1\leq k\leq n-2$. 
De plus, si $0<i<n$, $L_1$~et~$R_1$ sont des unités, si $i=0$, $L_1$ est une unité et $R_1=(01)$, et si $i=n$, $R_1$ est une unité et $L_1=(n-1,n)$. 
\end{lemme}

\begin{proof}
Voir~\cite{Street} ou \cite[section 4.3]{Felix}. La dernière assertion résulte du fait qu'il n'existe pas de cellule non triviale dans $\Or n$ de $0$\nbd-but $0$ ou de $0$\nbd-source $n$, et du fait que la seule $1$\nbd-cellule de source $0$ (resp. $n-1$) et de but $1$ (resp. $n$) est $(01)$ (resp.~$(n-1,n)$).
\end{proof}

\begin{rem}
En particulier, dans les notations du lemme, pour $n=1$, on a
$$
s_0(\xi_1)=A_1=(0\widehat1)=0\,,\quad t_0(\xi_1)=A_0=(\widehat01)=1\,,
$$
pour $n=2$,
$$
s_1(\xi_2)=A_1=(0\widehat12)=(02)\,,\quad t_1(\xi_2)=A_0*_0A_2=(\widehat012)*_0(01\widehat2)=(12)*_0(01)\,,
$$
et pour $n=3$,
$$\begin{aligned}
&s_2(\xi_3)=A_3*_1A_1=((23)*_0(012))*_1(023)\,,\cr
&t_2(\xi_3)=A_0*_1A_2=((123)*_0(01))*_1(013)\,.
\end{aligned}$$
Pour $n$ arbitraire, ce lemme ne donne pas une description complète de la source et du but de $\xi_n$, mais fournit des informations suffisantes pour ce qui suit.
\end{rem}

\begin{paragr}\label{NerfStreet}
La $\infty$\nbd-catégorie cosimpliciale $\Or{}:\Catsmp\to\ooCat$ définit par le procédé de Kan un couple de foncteurs adjoints
$$
c:\pref{\Catsmp}\longrightarrow\ooCat\,,\qquad N:\ooCat\longrightarrow\pref{\Catsmp}\,,
$$
où $c$ est le prolongement du foncteur $\Or{}$ par limites inductives, et le \emph{nerf de Street} $N$ est défini par
$$
C\mapsto(\smp n\mapsto\Hom_{\ooCat}(\Or{n},C))\,.
$$
On remarque aussitôt que la restriction du foncteur $N$ à $\Cat$ est le foncteur nerf usuel et que le foncteur $c^{}_1$ associant à un ensemble simplicial sa catégorie fondamentale est le composé de $c$ avec le foncteur de troncation intelligente $\tauint{1}$, adjoint à gauche de l'inclusion de $\Cat$ dans $\ooCat$. Autrement dit, pour tout ensemble simplicial $X$, sa catégorie fondamentale est canoniquement isomorphe à la catégorie fondamentale de la $\infty$\nbd-catégorie $c(X)$, et par suite les groupoïdes fondamentaux de $X$ et de $c(X)$ sont aussi canoniquement isomorphes.
\end{paragr}

\begin{paragr}
Soient $X$ un ensemble simplicial et $M:c^{}_1(X)^\circ\to\Ab$ un \SLF{} sur $X$. En composant le foncteur $M$ avec le morphisme d'adjonction $c(X)\to\tauint{1}c(X)=c^{}_1(X)$, on obtient un \SLF{} sur $c(X)$ qu'on notera également $M:c(X)^\circ\to\Ab$. Réciproquement, si $M:c(X)^\circ\to\Ab$ est un \SLF{} sur la $\infty$\nbd-catégorie $c(X)$, comme $\Ab$ est une $(1\hbox{-})$catégorie le foncteur $M$ se factorise par $\tauint{1}c(X)=c^{}_1(X)$, et définit donc un \SLF{} sur l'ensemble simplicial $X$. Ainsi les \SLFS{} sur $X$ et sur $c(X)$ se correspondent bijectivement.
\end{paragr}

Le but de la suite de cette section est de montrer que si $X$ est un ensemble~\hbox{simplicial} et $M$ un \SLF{} sur $X$, pour tout $n\geq0$, le groupe d'\hbox{homologie}~$\H_n(X,M)$ est canoniquement isomorphe au groupe d'homologie polygraphique $\Hp{n}(c(X),M)$. Plus précisément, on va montrer que le complexe normalisé du groupe abélien simplicial $\C(X,M)$ est canoniquement isomorphe au complexe $\C(c(X),M)$. Comme la $\infty$\nbd-catégorie $c(X)$ est un polygraphe, on obtient donc un isomorphisme canonique $\HT(X,M)\simeq\HP(c(X),M)$ dans la catégorie dérivée des groupes abéliens. 

\begin{paragr}
Soient $X$ un ensemble simplicial et $M$ un \SLF{} sur $X$. On note $\N(X,M)$ le complexe normalisé associé au groupe abélien simplicial $\C(X,M)$. 
On rappelle que pour $n\geq0$,
$$
\N_n(X,M)=\textstyle\bigoplus\limits_{x\in X^\nd_n}M_{x_n}\,,\quad n\geq0\,,
$$
où $X^\nd_n$ désigne l'ensemble des $n$\nbd-simplexes non dégénérés de $X$, la différentielle étant définie par
$$\begin{aligned}
&d(x,m)=\textstyle\sum\limits_{0\leq i<n}(-1)^i(d_ix,m)+(-1)^n(d_nx,x^*_{n-1,n}(m))\,,\ x\in X^\nd_n,\, m\in M_{x_n},\, n>0\,,
\end{aligned}
$$
avec la convention que les termes où $d_ix$ est dégénéré sont nuls.
\smallbreak

Un $n$\nbd-simplexe $x$ de $X$ définit par Yoneda un morphisme d'ensembles simpliciaux $\Yon{x}:\smp n\to X$, d'où par application du foncteur $c$, un $\infty$\nbd-foncteur $c(\Yon{x}):\Or{n}\to c(X)$, et par suite, une $n$\nbd-cellule $c(x):=c(\Yon{x})(\xi_n)$ de $c(X)$, où $\xi_n$ est la cellule principale de~$\Or n$. Pour toute application croissante $\varphi:\smp{n_0}\to\smp{n_1}$ et tout $n_1$\nbd-simplexe $x_1$~de~$X$, si~$x_0=X_\varphi(x_1)$, de sorte que $\Yon{x}_0=\Yon{x_1}\varphi$, on a
$$
c(x_0)=c(\Yon{x}_0)(\xi_{n_0})=c(\Yon{x}_1\varphi)(\xi_{n_0})=c(\Yon{x}_1)(\Or{\varphi}(\xi_{n_0}))\,.
$$
En particulier, pour $\varphi$ un opérateur de face, on obtient que si $x$ est un $n$\nbd-simplexe et $0\leq i\leq n$, on a 
$$c(d_ix)=c(\Yon{x})(0\dots\widehat i\dots n)\,.$$
De même, pour $\varphi$ un opérateur de dégénérescence, on obtient que si $x$ est un $n$\nbd-simplexe dégénéré la $n$\nbd-cellule $c(x)(\xi_n)$ est une unité. Enfin, pour $\varphi$ l'application $\smp0\to\smp n$, $0\mapsto n$, on obtient $t_0c(x)=c(x_n)$. Ainsi, on définit un morphisme de groupes abéliens
$$\begin{aligned}
f_n:\N_n(X,M)&=\kern -5pt\textstyle\bigoplus\limits_{x\in X^\nd_n}\kern-3ptM_{x_n}\longrightarrow\kern -3pt\textstyle\bigoplus\limits_{y\in c(X)_n}\kern -5ptM_{t_0y}\Bigm/\sim_n\kern 3pt=\C_n(c(X),M)\cr
&\kern 18pt(x,m)\kern4pt\longmapsto\kern5pt (c(x),m)\ .
\end{aligned}$$
\end{paragr}

\begin{prop}\label{prop:4.7}
Le morphisme de groupes abéliens gradués $f=(f_n)_{n\geq0}$ est un isomorphisme de complexes. 
\end{prop}

\begin{proof}
Vérifions la compatibilité de $f$ aux différentielles. Soient donc $n\geq1$, $x$ un  $n$\nbd-simplexe de $X$, et $m\in M_{x_n}$. On distingue plusieurs cas.
\smallbreak

--- $n=1$. On a:
$$\begin{aligned}
df_1(x,m)&=d(c(x),m)=(t_0(c(x)),m)-(s_0(c(x)),c(x)^*m) \cr
&=(t_0(c(\Yon{x})(\xi_1)),m)-(s_0(c(\Yon{x})(\xi_1)),x^*m)\cr
&=(c(\Yon{x})(t_0\xi_1),m)-(c(\Yon{x})(s_0\xi_1),x^*m)\cr
&=(c(\Yon{x})(\widehat01),m)-(c(\Yon{x})(0\widehat1),x^*m)\cr
&=(c(d_0x),m)-(c(d_1x),x^*m)\cr
&=f_0(d_0x,m)-f_0(d_1x,x^*m)=f_0d(x)\,.
\end{aligned}$$ 

--- $n=2$. On a:
$$\begin{aligned}
df_2(x,m)&=d(c(x),m)=(t_1(c(x)),m)-(s_1(c(x)),m) \cr
&=(t_1(c(\Yon{x})(\xi_2)),m)-(s_1(c(\Yon{x})(\xi_2)),m)\cr
&=(c(\Yon{x})(t_1\xi_2),m)-(c(\Yon{x})(s_1\xi_2),m)\cr
&=(c(\Yon{x})((\widehat012)*_0(01\widehat2)),m)-(c(\Yon{x})(0\widehat12),m)\cr
&=(c(\Yon{x})(\widehat012)*_0c(\Yon{x})(01\widehat2),m)-(c(\Yon{x})(0\widehat12),m)\cr
&\sim_1(c(\Yon{x})(\widehat012),m)+(c(\Yon{x})(01\widehat2),c(\Yon{x})(\widehat012)^*(m))-(c(\Yon{x})(0\widehat12),m)\cr
&=(c(\Yon{x})(\widehat012),m)-(c(\Yon{x})(0\widehat12),m)+(c(\Yon{x})(01\widehat2),x_{12}^*(m))\cr
&=(c(d_0x),m)-(c(d_1x),m)+(c(d_2x),x_{12}^*(m))\cr
&=f_1(d_0x,m)-f_1(d_1x,m)+f_1(d_2x,x_{12}^*(m))=f_1d(x)\,.
\end{aligned}$$
\smallbreak

--- $n>2$, $n$ impair. En vertu du lemme~\ref{LemmeOr}, et dans ses notations, on a:
$$\begin{aligned}
df_n(x,m)&=d(c(x),m)=(t_{n-1}(c(x)),m)-(s_{n-1}(c(x)),m) \cr
&=(t_{n-1}(c(\Yon{x})(\xi_n)),m)-(s_{n-1}(c(\Yon{x})(\xi_n)),m)\cr
&=(c(\Yon{x})(t_{n-1}\xi_n),m)-(c(\Yon{x})(s_{n-1}\xi_n),m)\cr
&=(c(\Yon{x})(A_{0}*_{n-2}A_{2}*_{n-2}\cdots*_{n-2}A_{n-1}),m)\cr
&\qquad-(c(\Yon{x})(A_{n}*_{n-2}\cdots*_{n-2}A_{3}*_{n-2}A_1),m)\cr
&=(c(\Yon{x})(A_{0})*_{n-2}(c(\Yon{x})(A_{2})*_{n-2}\cdots*_{n-2}(c(\Yon{x})(A_{n-1}),m)\cr
&\qquad-(c(\Yon{x})(A_{n})*_{n-2}\cdots*_{n-2}(c(\Yon{x})(A_{3})*_{n-2}(c(\Yon{x})(A_1),m)\cr
&\sim_{n-1}(c(\Yon{x})(A_{0}),m)+(c(\Yon{x})(A_{2}),m)+\cdots+(c(\Yon{x})(A_{n-1}),m)\cr
&\qquad-(c(\Yon{x})(A_{n}),m)-\cdots-(c(\Yon{x})(A_{3}),m)-(c(\Yon{x})(A_1),m)\,,\cr
\end{aligned}$$
où
$$
A_i=L_{n-2}*_{n-3}\cdots*_1L_1*_0(0\dots\widehat{i}\dots n)*_0R_1*_1\dots*_{n-3}R_{n-2}\,,
$$
avec $L_k,R_k$ $k$-cellules de $\Or{n}$, et $L_1$ une unité si $i\neq n$ et $R_i$ une unité si $i\neq0$. En vertu de la remarque~\ref{remunit}, on a donc,
si $0<i<n$,
$$
(c(\Yon{x})(A_{i}),m)\sim_{n-1}(c(\Yon{x})(0\dots\widehat{i}\dots n),m)=(c(d_ix),m)\,,
$$
si $i=0$, 
$$\begin{aligned}
(c(\Yon{x})(A_{0}),m)&\sim_{n-1}(c(\Yon{x})((\widehat{0}1\dots n)*_0R_1),m)=(c(\Yon{x})(\widehat{0}1\dots n)*_0c(\Yon{x})(R_1),m)\cr
&\sim_{n-1}(c(\Yon{x})(\widehat{0}1\dots n),m)=(c(d_0x),m)\,,
\end{aligned}$$
et si $i=n$, 
$$\begin{aligned}
(c(\Yon{x})(A_{n}),m)&\sim_{n-1}(c(\Yon{x})(L_1*_0(01\dots\widehat{n})),m)=(c(\Yon{x})(L_1)*_0c(\Yon{x})(01\dots\widehat{n}),m)\cr
&\sim_{n-1}(c(\Yon{x})(01\dots\widehat{n}),c(\Yon{x})(L_1)^*(m))=(c(d_nx),x_{n-1,n}^*(m))\,
\end{aligned}$$
(en se souvenant qu'en vertu du lemme~\ref{LemmeOr}, $L_1=(n-1,n)$). On en déduit que $df_n(x)\sim_{n-1}f_{n-1}d(x)$, ce qui achève la vérification de la compatibilité de $f$ aux différentielles (le cas $n$ pair étant tout à fait similaire).
\smallbreak

Vu que $c(X)$ est un polygraphe de base l'ensemble des $c(x)$, pour $x$ simplexe non dégénéré de $X$, la proposition résulte alors de la proposition~\ref{CasPol}.
\end{proof}

\begin{thm}\label{thm:compc}
Soient $X$ un ensemble simplicial et $M$ un système local faible sur~$X$. Alors, on a un isomorphisme canonique $\HT(X,M)\simeq\HP(c(X),M)$ dans la catégorie dérivée des groupes abéliens, et en particulier, pour tout $n\geq0$, un isomorphisme de groupes abéliens $\H_{n}(X,M)\simeq\Hp{n}(c(X),M)$.
\end{thm}

\begin{proof}
Comme la $\infty$\nbd-catégorie $c(X)$ est un polygraphe, le théorème est conséquence directe de la proposition \ref{prop:4.7}.
\end{proof}

\begin{paragr}\label{def:eqpol}
On dit qu'un $\infty$\nbd-foncteur $f:X'\to X$ est une \emph{équivalence polygraphique} si $f$ induit une équivalence des groupoïdes fondamentaux, et si pour tout \SL{}~$M$ sur $X$ et tout $n\geq0$, le morphisme de groupes abéliens 
$$\Hp{n}(f,M):\Hp n(X', f^*(M))\to\Hp n(X,M)$$
est un isomorphisme (ou de façon équivalente si
$$\HP(f,M):\HP(X', f^*(M))\to\HP(X,M)$$
est un isomorphisme dans la catégorie dérivée des groupes abéliens).
\end{paragr}

\begin{cor}\label{cor:compc}
Pour qu'un morphisme d'ensembles simpliciaux soit une équivalence faible simpliciale, il faut et il suffit que son image par le foncteur $c$, adjoint à gauche du nerf de Street, soit une équivalence polygraphique.
\end{cor}
 
\begin{proof}
Le corollaire est conséquence immédiate des théorèmes~\ref{thtopol} et~\ref{thm:compc}, en se souvenant que pour tout ensemble simplicial $X$, les groupoïdes fondamentaux de la $\infty$\nbd-catégorie $c(X)$ et de l'ensemble simplicial $X$ sont canoniquement isomorphes (cf.~paragraphe~\ref{NerfStreet}).
\end{proof}

\begin{prop}\label{prop:2sur3}
Si dans un triangle commutatif de $\ooCat$ deux parmi les trois flèches sont des équivalences polygraphiques, il en est de même de la troisième.
\end{prop}

\begin{proof}
Soient $f:X\to Y$ et $g:Y\to Z$ deux morphismes composables de $\ooCat$. La seule chose non triviale à montrer pour établir la proposition est que si~$g$ et $gf$ sont des équivalences polygraphiques, alors pour tout \SL{} $M$ sur~$Y$ et tout $n\geq0$, $\Hp n(f,M):\Hp n(X,f^*(M))\to\Hp n(Y,M)$ est un isomorphisme de groupes abéliens. Comme par hypothèse $\Pi_1(g)$ est une équivalence de groupoïdes, on peut en choisir un quasi-inverse $I:\Pi_1(Z)\to\Pi_1(Y)$. En se souvenant que la donnée d'un \SL{} sur une $\infty$\nbd-catégorie revient à la donnée d'un foncteur contravariant de son groupoïde fondamental vers la catégorie des groupes abéliens, le composé $M\circ I^\circ$ définit un système local $M'$ sur $Z$, et l'isomorphisme de foncteurs $I\circ\Pi_1(g)\overset{\sim}{\to}1_{\Pi_1(Y)}$ un isomorphisme $\varphi:g^*(M')\overset{\sim}{\to}M$ de \SLS{} sur $Y$. Or, comme $g$ et $gf$ sont des équivalences polygraphiques, l'égalité $\Hp n(gf,M')=\Hp n(g,M')\circ\Hp n(f,g^*(M'))$ (cf.~paragraphe~\ref{fonctHpX}) implique que $\Hp n(f,g^*(M'))$ est un isomorphisme, et le diagramme commutatif
$$\xymatrixcolsep{4.5pc}
\xymatrix{
\Hp{n}(X,f^*g^*(M'))\ar[r]^-{\Hp{n}(X,f^*(\varphi))\,}\ar[d]_{\Hp{n}(f,g^*(M'))}
&\Hp{n}(X,f^*(M))\ar[d]^{\Hp{n}(f,M)}
\\
\Hp{n}(Y,g^*(M'))\ar[r]_-{\Hp{n}(Y,\varphi)}
&\Hp{n}(Y,M)
}$$
(cf.~paragraphe~\ref{fonctHpM}), dont les lignes sont des isomorphismes, qu'il en est de même de $\Hp{n}(f,M)$.
\end{proof}

\begin{prop}\label{satfaible}
Soit $j:X'\to X$ un $\infty$\nbd-foncteur. On suppose que $j$ admet une rétraction $r$ et que $jr$ est une équivalence polygraphique. Alors $j$ est une équivalence polygraphique.
\end{prop}

\begin{proof}
L'hypothèse que $r$ est une rétraction de $j$ implique que $\Pi_1(r)\Pi_1(j)=1_{\Pi_1(X')}$, et l'hypothèse que $jr$ est une équivalence polygraphique implique que $\Pi_1(j)\Pi_1(r)=\Pi_1(jr)$ est un isomorphisme. On en déduit que $\Pi_1(j)$ est un isomorphisme. De même, pour tout \SL{} $M$ sur $X$, le morphisme $\HT(jr,M)=\HT(j,M)\circ\HT(r,j^*(M))$ (cf.~paragraphe~\ref{fonctHpX}) est un isomorphisme de la catégorie dérivée des groupes abéliens et on a l'égalité 
$$
\HT(r,j^*(M))\circ\HT(j,r^*j^*(M))=\HT(rj,j^*(M))=1_{\HT(X'\kern-2pt,j^*(M))}\,.
$$
On en déduit que $\HT(r,j^*(M))$ est un isomorphisme, donc $\HT(j,M)$ aussi, ce qui achève la démonstration.
\end{proof}

\begin{rem}
Dans la terminologie de Grothendieck~\cite{PS1},~\cite{MalPRST},~\cite{Ci}, le contenu des propositions~\ref{prop:2sur3} et~\ref{satfaible} est que les équivalences polygraphiques forment une classe faiblement saturée de flèches de $\ooCat$.
\end{rem}

\begin{prop}\label{retropl}
Soit $j:X'\to X$ un $\infty$\nbd-foncteur. On suppose que $j$ admet une rétraction $r$ et qu'il existe une transformation oplax $\alpha$ de $jr$ vers $1_X$. Alors $j$ est une équivalence polygraphique.
\end{prop}

\begin{proof}
L'hypothèse que $r$ est une rétraction de $j$ implique que $\Pi_1(r)\Pi_1(j)=1_{\Pi_1(X')}$, et la transformation oplax $\alpha$ induit une transformation naturelle $\Pi_1(j)\Pi_1(r)=\Pi_1(jr)\to1_{\Pi_1(X)}$, qui est un isomorphisme, puisque $\Pi_1(X)$ est un groupoïde. On en déduit que $\Pi_1(j)\Pi_1(r)$ est un isomorphisme, donc $\Pi_1(j)$ aussi. La proposition résulte alors du corollaire~\ref{retris} et de la remarque~\ref{remastbis}. 
\end{proof}

\begin{cor}\label{projpol}
Pour toute $\infty$\nbd-catégorie $X$, la projection $D_1\otimes X\overset{r}{\to} D_0\otimes X\simeq X$ est une équivalence polygraphique.
\end{cor}

\begin{proof}
Soit $j:X\simeq D_0\otimes X\to D_1\otimes X$ le produit de Gray de l'inclusion $D_0=\{0\}\hookrightarrow\{0\to1\}=D_1$ par $X$. Le $\infty$\nbd-foncteur $r$ est une rétraction de $j$. La transformation naturelle de l'endofoncteur constant de valeur $0$ de $D_1$ vers $1_{D_1}$ définit un morphisme $D_1\otimes D_1\to D_1$. On vérifie aussitôt que le produit de Gray de ce morphisme par $X$ définit une transformation oplax de $jr$ vers $1_{D_1\otimes X}$, ce qui implique en vertu de la proposition \ref{retropl} que $j$ est une équivalence polygraphique, et par suite $r$ aussi (proposition~\ref{prop:2sur3}).
\end{proof}

\section{Rappels sur l'homologie d'un système local $1$-catégorique}

Soient $X$ une petite catégorie et $M$ un \SLF{} (autrement dit, un préfaisceau abélien $M$ sur $X$). L'homologie $\HT(X,M)$ de $X$ à coefficients dans $M$ (appelée aussi homologie du foncteur $M$) généralise l'homologie d'un groupe $G$ à coefficients dans un $\mathbb Z$\nbd-module à droite (c'est-à-dire un préfaisceau abélien sur $G$, vu comme une catégorie à un seul objet). Classiquement, si $G$ est un groupe, ou plus généralement un monoïde, et $M$ un $G$\nbd-module à droite, les groupes d'homologie $\H_n(G,M)$ sont définis comme étant les groupes abéliens $\Tor^{\mathbb{Z}[G]}_n(M,\mathbb Z)$, où $\mathbb{Z}[G]$ désigne l'algèbre du groupe (ou du monoïde) $G$,\penalty-550{} $M$ le $\mathbb{Z}[G]$\nbd-module à droite défini par l'action de $G$, et $\mathbb Z$ est vu comme $\mathbb{Z}[G]$\nbd-module à gauche, $G$ agissant trivialement~\cite{Cart-Eil}. Or, les foncteurs $M\mapsto\Tor^{\mathbb{Z}[G]}_n(M,\mathbb Z)$ sont les foncteurs dérivés à gauche du foncteur $M\mapsto M\otimes_{\mathbb{Z}[G]}\mathbb Z$, et par ailleurs on a un isomorphisme canonique
$$
M\otimes_{\mathbb{Z}[G]}\mathbb Z\simeq\varinjlim_{G^\circ} M
$$
du groupe $M\otimes_{\mathbb{Z}[G]}\mathbb Z$ avec la limite inductive du foncteur $M:G^\circ\to\Ab$, groupe des \emph{coïnvariants} de $M$ sous l'action de $G$. Par conséquent, le groupe d'homologie $\H_n(G,M)$ s'identifie à l'évaluation en $M$ du $n$\nbd-ème foncteur dérivé à gauche du\penalty-600{} foncteur limite inductive, autrement dit au $n$\nbd-ème groupe d'homologie de la \clh{} de $M$, vue comme objet de la catégorie dérivée des groupes abéliens. 
\smallbreak

Ainsi, si $X$ est une petite catégorie et $M$ un \SLF{} sur~$X$, autrement dit un préfaisceau $X^{\circ} \to \Ab$, on définit l'\emph{homologie de $X$ à coefficients dans $M$} par la formule suivante
  \[
    \HT(X,M):=\hocolim M=\hocolim_{x \in X^{\circ}}M_x\,,
  \]
  où la \clh{} est prise dans la catégorie dérivée des groupes abéliens. 
\smallbreak

Si $X$ est une $\infty$\nbd-catégorie, on aimerait pouvoir définir son homologie par la même formule. Néanmoins, en l'état actuel des connaissances, on ne sait pas interpréter cette formule dans le cadre $\infty$\nbd-catégorique. Comme par définition le type d'homotopie d'une $\infty$\nbd-catégorie est le type d'homotopie de son nerf, on est conduit à poser la définition suivante. Si $X$ est une $\infty$\nbd-catégorie et~$M$ un \SLF{} sur X, l'\emph{homologie de $X$ à coefficient dans $M$} est définie par la formule
$$
\HT(X,M):=\HT(N(X),M)\,,
$$
où $N$ désigne le nerf de Street. Cette formule a un sens car on vérifie facilement que la catégorie fondamentale de l'ensemble simplicial $N(X)$ est canoniquement isomorphe à celle de la $\infty$\nbd-catégorie $X$, de sorte que les \SLFS{} sur $X$ s'identifient aux \SLFS{} sur $N(X)$.
\smallbreak

La proposition suivante montre que ces deux définitions coïncident quand $X$ est une ($1$\nbd-)catégorie et $M$ un \emph{\SL} (et pas seulement un \SLF)~sur~$X$.

\begin{prop}\label{comp2def}
Soient $X$ une catégorie et $M$ un système local sur $X$. On a un isomorphisme canonique
$$
\hocolim_{x \in X^{\circ}}M_x\simeq\HT(N(X),M)\,.
$$
\end{prop}

Même si cette proposition n'est pas vraiment originale, on en esquisse une preuve pour la convenance du lecteur. Pour cela, on aura besoin des lemmes suivants.

\begin{lemme}\label{lemme:abstr}
Soient $C$ une catégorie de modèles, $X$ un objet de $C$, et considérons la catégorie $\cm{C}{X}$, munie de la structure de catégorie de modèles induite de celle de~$C$. Alors le foncteur d'oubli $\cm{C}{X}\to C$ respecte et reflète les \clhs.
\end{lemme}

\begin{proof}
Par définition de la structure de catégorie de modèles induite, le foncteur d'oubli $U:\cm{C}{X}\to C$ respecte les équivalences faibles et les cofibrations, et comme il admet un adjoint à droite, il est, en particulier, un foncteur de Quillen à gauche. Il respecte donc les \clhs. Pour montrer qu'il les reflète également, il suffit alors de remarquer que le foncteur $\Ho(\cm{C}{X})\to\Ho(C)$, induit par~$U$, est conservatif. Cette dernière assertion résulte du fait que le foncteur d'oubli respecte et reflète les équivalences faibles, et du fait que toute flèche de la catégorie homotopique d'une catégorie de modèles $C$ peut s'écrire sous la forme $P(h)^{-1}P(g)P(f)^{-1}$, où $P:C\to\Ho(C)$ désigne le foncteur de localisation et $f,h$ sont des équivalence faibles et $g$ une flèche arbitraire de $C$. 
\end{proof}

\begin{lemme}\label{lemma:colslice}
Soit $X$ une petite catégorie, et considérons la catégorie $\cm{\Cat}{X}$ munie de la structure de catégorie de modèles induite de celle Thomason sur $\Cat$~\cite{Th2}. Alors l'objet $(X,1_X)$ de $\cm{\Cat}{X}$ est la \clh{} dans $\cm{\Cat}{X}$ du foncteur 
\[
X^{\circ} \to \cm{\Cat}{X}\,,\qquad x \mapsto \mc{x}{X},
\]
 où chaque $\mc{x}{X}$ est équipé de la projection canonique vers $X$.
\end{lemme}

\begin{proof}
D'après le lemme \ref{lemme:abstr}, il suffit de voir que la \clh{} dans $\Cat$ du foncteur $X^{\circ} \to \Cat$, $x \mapsto\mc{x}{X}$ est isomorphe à $X$, ce qui suit facilement du fait que ce foncteur est faiblement équivalent au foncteur constant de valeur la catégorie ponctuelle, et du calcul de la \clh{} de ce dernier à l'aide de la construction de Grothendieck (voir~\cite{Th1}). 
\end{proof}

\begin{lemme}\label{lemme5.4}
Soit $X$ une petite catégorie, et considérons les catégories $\cm{\Cat}{X}$ et $\cm{\pref{\Catsmp}}{N(X)}$ munies des structures de catégorie de modèles induites respectivement de celle de Thomason~\cite{Th2} et de celle de Kan-Quillen~\cite{Qu}. Alors le foncteur 
$$
\cm{N}{X}:\cm{\Cat}{X}\to\cm{\pref{\Catsmp}}{N(X)}\,,\quad (X'\kern -2pt,p:X'\to X)\mapsto(N(X'),N(p):N(X')\to N(X))\,,
$$
commute aux \clhs.
\end{lemme}

\begin{proof}
On a un carré commutatif
$$\xymatrix{
\cm{\Cat}{X}\ar[r]^-{\cm{N}{X}}\ar[d]
&\cm{\pref{\Catsmp}}{N(X)}\ar[d]
\\
\Cat\ar[r]_N
&\pref{\Catsmp}
}$$
dont les flèches verticales sont les foncteurs d'oubli. Comme en vertu du lemme~\ref{lemme:abstr}, ces deux foncteurs respectent et reflètent les \clhs, le lemme résulte du fait que le foncteur nerf respecte les équivalences faibles et préserve les \clhs, puisque, selon le goût du lecteur, il induit une équivalence de dérivateurs ou une équivalence de $(\infty,1)$\nbd-catégories.
\end{proof}

\begin{proof}[Démonstration de la proposition~\ref{comp2def}]
L'objet $\HT(N(X),M)$ de la catégorie dérivée des groupes abéliens est défini par le complexe image de l'objet $(X,1_X)$ de $\cm{\Cat}{X}$ par le composé des foncteurs
$$
\xymatrix{
&\cm{\Cat}{X}\ar[r]^-{\cm{N}{X}}
&\cm{\pref{\Catsmp}}{N(X)}\ar[r]^-F
&\Comp(\Ab)
&\kern -20pt,\kern20pt
}
$$
où $F$ désigne le foncteur associant à un objet $(T,p:T\to N(X))$ de $\cm{\pref{\Catsmp}}{N(X)}$ le complexe $\C(T,p^*(M))$, introduit au paragraphe~\ref{def:homsimpl}. Or, en vertu du lemme \ref{lemme5.4}, le foncteur $\cm{N}{X}$ commute aux \clhs, et il en est de même du foncteur~$F$, puisqu'on a vu dans la preuve du théorème~\ref{thtopol} qu'il est un foncteur de Quillen à gauche. Il résulte donc du lemme~\ref{lemma:colslice} qu'on a des isomorphismes canoniques dans la catégorie dérivée des groupes abéliens
$$\begin{aligned}
\HT(N(X),M)&=F\circ\cm{N}{X}\,(X,1_X)\simeq F\circ\cm{N}{X}\bigl(\hocolim_{x\in X^\circ}\mc{x}{X}\bigr)\cr
&\simeq\hocolim_{x\in X^\circ}F\circ\cm{N}{X}(\mc{x}{X})=\hocolim_{x\in X^\circ}\C(N(\mc{x}{X}),M\vert_{\mc{x}{X}})\,,
\end{aligned}$$
où $M\vert_{\mc{x}{X}}=p_x^*(M)$ pour $p_x:\mc{x}{X}\to X$ le foncteur canonique d'oubli. 
\smallbreak

Pour conclure, on va définir une équivalence faible naturelle du foncteur
$$
X^\circ\to\Comp(\Ab)\,,\qquad x\mapsto\C(N(\mc{x}{X}),M\vert_{\mc{x}{X}})\,,
$$
vers le foncteur
$$
X^\circ\to\Comp(\Ab)\,,\qquad x\mapsto M_x\,,
$$
le groupe abélien $M_x$ étant vu comme un complexe concentré en degré $0$. On définit un morphisme de groupes abéliens
$$
\varphi_x:\C_0(N(\mc{x}{X}),M\vert_{\mc{x}{X}})=\kern-5pt\textstyle\bigoplus\limits_{(y,\,p\,:\,x \to y)}\kern-5ptM_y \longrightarrow M_x\,,\qquad(y,p,m)\mapsto p^*(m)\,.
$$
On rappelle que la différentielle $d:\C_{1}(N(\mc{x}{X}),M\vert_{\mc{x}{X}}) \to \C_0(N(\mc{x}{X}),M\vert_{\mc{x}{X}})$ est définie par la
formule
\[
\textstyle\bigoplus\limits_{f\,:\,(y,p)\to (y',p')}\kern -10ptM_{y'} \ \longrightarrow\kern -6pt \textstyle\bigoplus\limits_{(y,\,p\,:\,x \to y)}\kern-2ptM_y\,,\qquad(f,m)\mapsto \left((y',p',m) - (y,p,f^*(m))\right).
\]
On vérifie alors que pour un élément $(f,m)$ de $\C_1(N(\mc{x}{X}),M\vert_{\mc{x}{X}})$, on a
\[
\begin{aligned}
    \varphi_x(d(f,m)) &= \varphi_x\left((y',p',m) - (y,p,f^*(m))\right) \\
    &= p'{}^*(m) - p^*(f^*(m))= p'{}^*(m) - (fp)^*(m)=0\,.
\end{aligned}
\]
Cela signifie exactement que $\varphi_x$ est compatible à la différentielle et on a donc défini un morphisme de complexes
  \[
    \varphi_x : \C(N(\mc{x}{X}),M\vert_{\mc{x}{X}}) \to M_x,
  \]
où $M_x$ est concentré en degré $0$. 
\smallbreak

La naturalité de $\varphi_x$ en $x$ étant évidente, il reste à prouver que pour tout objet $x$ de~$X$ le morphisme de complexes $\varphi_x$ est un quasi-isomorphisme. Pour cela, considérons le morphisme $e\to\mc{x}{X}$ de la catégorie ponctuelle vers la catégorie $\mc{x}{X}$ défini par l'objet initial de cette dernière. Le morphisme $N(e)\to N(\mc{x}{X})$ est alors une équivalence faible simpliciale et le théorème~\ref{thtopol} implique que la flèche induite
\[
M_x \simeq \C(N(e),M_x) \longrightarrow \C(N(\mc{x}{X}),M\vert_{\mc{x}{X}})
\]
est un quasi-isomorphisme. Comme ce morphisme est une section de $\varphi_x$, on conclut par 2 sur 3.
\end{proof}

\section{L'homologie polygraphique d'un système local $1$-catégorique}\label{Leonard2}

Soient $X$ une $\infty$\nbd-catégorie et $M$ un \SLF{} sur $X$. On définit un morphisme $\HT(X,M)\to\HP(X,M)$ dans la catégorie dérivée des groupes abéliens, de source l'homologie et but l'homologie polygraphique de $X$ à coefficients dans $M$ comme suit. Le morphisme d'adjonction $cN(X)\to X$ induit un morphisme $\HP(cN(X),M)\to\HP(X,M)$ (en se souvenant que les \SLFS{} sur $X$, $N(X)$ et $cN(X)$ se correspondent bijectivement). Or, en vertu du théorème~\ref{thm:compc}, on a un isomorphisme canonique
$$
\HP(cN(X),M)\simeq\HT(N(X),M)=\HT(X,M)\,,
$$
d'où le morphisme désiré.
\smallbreak

Le but de cette section est de montrer que dans le cas particulier où $X$ est une petite catégorie et $M$ un \SL{} sur $X$ (autrement dit, un préfaisceau abélien $M$ sur $\Pi_1(X)$), le morphisme de comparaison ci-dessus est un isomorphisme, autrement dit que l'homologie polygraphique $\HP(X,M)$ coïncide avec l'homologie usuelle $\HT(X,M)$. Cela résultera du théorème suivant, de la proposition~\ref{comp2def}, et de l'observation (dont la vérification est laissée au lecteur) que les isomorphismes canoniques en question sont compatibles au morphisme de comparaison.

\begin{thm}\label{thm:Leonard}
Soient $X$ une catégorie et $M$ un \SL{} sur $X$. On a un isomorphisme canonique dans la catégorie dérivée des groupes abéliens
$$
\hocolim_{x \in X^{\circ}}M_x\simeq\HP(X,M)\,.
$$
\end{thm}

Soit donc $X$ une petite catégorie, $M$ un \SL{} sur $X$, et choisissons une résolution polygraphique $p:P\to X$ de $X$ dans $\ooCat$, de sorte qu'on a un isomorphisme canonique $\HP(X,M)\simeq\C(P,p^*(M))$ dans la catégorie dérivée des groupes abéliens.
On considère le foncteur
$$
\mc{?}{P}:X^\circ\to\cm{\ooCat}{X}\,,\qquad x\mapsto (\mc{x}{P},\,p_x:\mc{x}{P}\to X)\,,
$$
où $\mc{x}{P}$ désigne le produit fibré $P\times_X\mc{x}{X}$ et $p_x$ le composé de la projection $\mc{x}{P}\to P$ avec $p:P\to X$. Pour démontrer le théorème, on aura besoin des lemmes suivants.

\begin{lemme}\label{lem:obcof}
Le foncteur $\mc{?}{P}$ ci-dessus est un objet cofibrant de la catégorie $\Homint(X^\circ,\cm{\ooCat}{X})$ munie de la structure projective, $\cm{\ooCat}{X}$ étant équipée de la structure de catégorie de modèles induite de la structure folk sur $\ooCat$. En particulier, pour tout objet $x$ de $X$, la $\infty$\nbd-catégorie $\mc{x}{P}$ est un polygraphe.
\end{lemme}
 
\begin{proof}
Remarquons qu'on a un isomorphisme canonique
$$
\Homint(X^{\circ},\cm{\ooCat}{X})\simeq\cm{\Homint(X^{\circ},\ooCat)}{\mathrm{cst}_X}\,,
$$
où $\mathrm{cst}_X$ est le foncteur constant de valeur $X$. Par ailleurs, cet isomorphisme est compatible aux structures de catégories de modèles, dans le sens où la structure sur $\Homint(X^{\circ},\cm{\ooCat}{X})$ de l'énoncé est identifiée avec la structure de catégorie de modèles sur $\cm{\Homint(X^{\circ},\ooCat)}{\mathrm{cst}_X}$ induite par la structure projective folk sur $\Homint(X^\circ,\ooCat)$. En particulier, on en déduit que le foncteur $\mc{?}{P}:X^\circ\to\cm{\ooCat}{X}$ est un objet cofibrant de $\Homint(X^{\circ},\cm{\ooCat}{X})$ si et seulement son composé avec le foncteur d'oubli $\mc{?}{P}:X^\circ\to\ooCat$ est un objet cofibrant de $\Homint(X^{\circ},\ooCat)$, ce qui est exactement (le dual de) la proposition~7.9 de~\cite{Leonard}.
\end{proof}

\begin{lemme}\label{lem:colim}
L'objet $(P,p:P\to X)$ de $\cm{\ooCat}{X}$ est la \clh{} dans $\cm{\ooCat}{X}$ \emph{(}munie de la structure de catégorie de modèles induite de la structure folk sur $\ooCat$\emph{)} du foncteur $\mc{?}{P}:X^\circ\to\cm{\ooCat}{X}$.
\end{lemme}

\begin{proof}
Le lemme résulte aussitôt du lemme~\ref{lemme:abstr} et de~\cite[Lemma~7.5 et~Theorem~7.10]{Leonard}.
\end{proof}

Soit $x$ un objet de $X$. On rappelle que $p_x:\mc{x}{P}\to X$ désigne le composé du $\infty$\nbd-foncteur d'oubli $\mc{x}{P}\to P$ avec $p:P\to X$. Notons $q_x:\mc{x}{P}\to e$ l'unique $\infty$\nbd-foncteur de $\mc{x}{P}$ vers la $\infty$\nbd-catégorie ponctuelle $e$, et considérons le \SL{} $M_x$ sur $e$ défini par le groupe~$M_x$. On définit une transformation naturelle
$\varphi_x:p_x^*(M)\to q_x^*(M_x)$, en posant pour tout objet $(y,u:x\to p(y))$ de $\mc{x}{P}$
$$
\varphi_{x,(y,u)}=M_u:M_{p(y)}=(p_x^*(M))_{(y,u)}\to(q_x^*(M_x))_{(y,u)}=M_x\,,
$$
d'où un morphisme de \SLS{} dans $\ooCat$ $(q_x,\varphi_x):(\mc{x}{P},p_x^*(M))\to(e,M_x)$ (cf.~paragraphe~\ref{DoubleFonct}).
On en déduit un morphisme de complexes 
$$
\psi_x:=\C(q_x,\varphi_x):\C(\mc{x}{P},p_x^*(M))\to\C(e,M_x)\simeq M_x
$$
(où on note également $M_x$ le complexe concentré en degré $0$ correspondant à $M_x$).

\begin{lemme}\label{lem:fonctLeo}
Le morphisme de complexes $\psi_x:\C(\mc{x}{P},p_x^*(M))\to M_x$ est naturel en $x$.
\end{lemme}

\begin{proof}
Il s'agit de montrer que si $v:x\to x'$ est un morphisme de la catégorie $X$, alors le carré
$$
\xymatrixcolsep{4pc}
\xymatrix{
\C(\mc{x'}{P},p_{x'}^*(M))\ar[r]^-{\C(q_{x'},\varphi_{x'})}\ar[d]_{\C(\mc{v}{P},\,1_{p_{\kern -1ptx'}^*(M)})}
&\C(e,M_{x'})\ar[d]^{\C(1_e,M_v)}
\\
\C(\mc{x}{P},p_x^*(M))\ar[r]_-{\C(q_x,\varphi_x)}
&\C(e,M_x)
}$$
est commutatif. Or, on a des égalités de morphismes de \SLS{} dans $\ooCat$ (cf.~paragraphe~\ref{DoubleFonct})
$$\begin{aligned}
(q_x,\varphi_x)\circ(\mc{v}{P},\,1_{p_{\kern -1ptx'}^*(M)})&=(q_x\circ\mc{v}{P},\,(\mc{v}{P})^*(\varphi_x))\cr
\noalign{\vskip3pt}
(1_e,M_v)\circ(q_{x'},\varphi_{x'})&=(q_{x'},q_{x'}^*(M_v)\varphi_{x'})
\end{aligned}$$
et on remarque que $q_x\circ\mc{v}{P}=q_{x'}$, et que pour tout objet $(y,u:x'\to p(y))$ de $\mc{x'}{P}$, 
$$
((\mc{v}{P})^*(\varphi_x))_{(y,u)}=(\varphi_x*\mc{v}{P})_{(y,u)}=\varphi_{x,(y,uv)}=M_{uv}=M_vM_u=(q_{x'}^*(M_v)\varphi_{x'})_{(y,u)}\,,
$$
ce qui achève la démonstration.
\end{proof}

\begin{lemme}\label{lemme:6.5}
Le morphisme de complexes $\psi_x:\C(\mc{x}{P},p_x^*(M))\to M_x$ est un quasi-isomorphisme.
\end{lemme}

\begin{proof}
Par définition, $\psi_x=\C(q_x,\varphi_x)=\C(q_x,M_x)\circ\C(\mc{x}{P},\varphi_x)$ (cf.~paragraphe~\ref{DoubleFonct}), et comme $M$ est un \SL{} (et pas seulement un \SLF), la transformation naturelle $\varphi_x$ est un isomorphisme, et par suite, $\C(\mc{x}{P},\varphi_x)$ est un isomorphisme de complexes. Il suffit donc de prouver que $\C(q_x,M_x)$ est un quasi-isomorphisme, autrement dit (puisque en vertu du lemme~\ref{lem:obcof} la $\infty$\nbd-catégorie $\mc{x}{P}$ est un polygraphe) que $\HP(q_x,M_x)$ est un isomorphisme de la catégorie dérivée des groupes abéliens. Notons $r_x:\mc{x}{X}\to e$ le morphisme canonique et $j_x:e\to\mc{x}{X}$ le foncteur défini par l'objet initial $(x,1_x)$ de la catégorie $\mc{x}{X}$.\penalty-500{} On a $q_x=r_x\circ\mc{x}{p}$, où $\mc{x}{p}\,:\mc{x}{P}\to\mc{x}{X}$ désigne le $\infty$\nbd-foncteur induit par $p$, et par suite, $\HP(q_x,M_x)=\HP(r_x,M_x)\circ\HP(\mc{x}{p},r_x^*(M_x))$ (cf.~paragraphe~\ref{fonctHpX}).\penalty-500{} Or, la stabilité des fibrations triviales par images inverses implique que le morphisme $\mc{x}{p}$ est une fibration triviale folk, et en particulier une équivalence folk, ce qui implique, en vertu du théorème~\ref{th:fond}, que \smash{$\HP(\mc{x}{p},r_x^*(M_x))$} est un isomorphisme. Il reste à montrer que $\HP(r_x,M_x)$ est un isomorphisme, ce qui résulte aussitôt du corollaire~\ref{retris} et de la remarque~\ref{remastbis}, en observant que $r_x$ est une rétraction de $j_x$ et qu'on a une transformation naturelle canonique de $j_xr_x$ vers $1_{\mc{x}{X}}$. 
\end{proof}

\begin{proof}[Démonstration du théorème~\ref{thm:Leonard}]
Considérons le foncteur
$$
F=F_{X,M}:\cm{\ooCat}{X}\to\Comp(\Ab)\,,\qquad(X',f:X'\to X)\mapsto\C(X',f^*(M))\,,
$$
qui est un foncteur de Quillen à gauche (théorème~\ref{th:Quilleng}), de sorte qu'on a un carré commutatif à isomorphisme près
$$
\xymatrixcolsep{3pc}
\xymatrix{
\Ho(\Homint(X^\circ,\cm{\ooCat}{X}))\ar[r]^{\LL F^{X^\circ}_{X,M}}\ar[d]_{\shocolim}
&\Ho(\Homint(X^\circ,\Comp(\Ab)))\ar[d]^{\shocolim}
\\
\Ho(\cm{\ooCat}{X})\ar[r]_{\LL F_{X,M}}
&\Ho(\Comp(\Ab))\kern20pt,\kern-20pt
}$$
où $\LL F^{X^\circ}_{X,M}$ désigne le foncteur dérivé à gauche du foncteur
$$
\xymatrixcolsep{3pc}
\xymatrix{
\Homint(X^\circ,\cm{\ooCat}{X})\ar[r]^{\F^{X^\circ}_{X,M}}
&\Homint(X^\circ,\Comp(\Ab))
}$$
induit par $F_{X,M}$.
Or, le quasi-isomorphisme $\psi_x:\C(\mc{x}{P},p_x^*(M))\to M_x$ du lemme~\ref{lemme:6.5} étant naturel en $x$ (lemme~\ref{lem:fonctLeo}) définit une équivalence faible dans la catégorie $\Homint(X^\circ,\Comp(\Ab))$ de source $F^{X^\circ}_{X,M}(\mc{?}{P})$ et de but $M$. Comme en vertu du lemme~\ref{lem:obcof}, $\mc{?}{P}$ est un objet cofibrant de $\Homint(X^\circ,\cm{\ooCat}{X})$, on a $\LL F^{X^\circ}_{X,M}(\mc{?}{P})=F^{X^\circ}_{X,M}(\mc{?}{P})$, et on en déduit donc un isomorphisme
$$
\hocolim\, \LL F^{X^\circ}_{X,M}(\mc{?}{P})\simeq\hocolim M\,.
$$
D'autre part, en vertu du lemme~\ref{lem:colim}, on a un isomorphisme 
$$
\LL F_{X,M}\,\hocolim\,(\mc{?}{P})\simeq\LL F_{X,M}(P,p)=\HP(X,M)\,,
$$
ce qui achève la démonstration.
\end{proof}

\begin{cor} 
La restriction à $\Cat$ de la classe des équivalences faibles polygraphiques de $\ooCat$ \emph{(cf. paragraphe~\ref{def:eqpol})} coïncide avec la classe des équivalences de Thomason de $\Cat$.
\end{cor}

\begin{proof}
Le corollaire résulte aussitôt de la proposition~\ref{comp2def} et du théorème~\ref{thm:Leonard}, qui impliquent que le morphisme de comparaison 
$$\HT(N(X),M)\to\HP(X,M)\,,$$ 
pour une (1-)catégorie $X$ et un \SL{} $M$ sur $X$, est un isomorphisme (voir le début de la section), de l'isomorphisme canonique $\Pi_1(N(X))\simeq\Pi_1(X)$, et du théorème~\ref{thtopol}.
\end{proof}

\begin{sch}\label{scholie}
On rappelle qu'une équivalence de Thomason de $\ooCat$ est un $\infty$\nbd-foncteur dont le nerf de Street est une équivalence faible simpliciale. On dispose ainsi de trois classes d'équivalences faibles dans $\ooCat$: les classes $\Wfolk$ des équivalences folk, $\Wpol$ des équivalences polygraphiques et $\Wthom$ des équivalences de Thomason. Le corollaire précédent signifie que $\Wpol\cap\Fl(\Cat)=\Wthom\cap\Fl(\Cat)$, où $\Fl(\Cat)$ désigne la classe des flèches de $\Cat$. Par ailleurs, $\Wfolk\subset\Wpol$ (en vertu du théorème~\ref{th:fond}) et $\Wfolk\subset\Wthom$ (en vertu de la proposition~3.6.2 de~\cite{LeonardTh}). 
\smallbreak

En revanche, on n'a aucune inclusion entre les classes $\Wpol$ et $\Wthom$, que ce soit dans un sens ou dans l'autre. En effet, soient $B$ la \og bulle\fg, $2$\nbd-catégorie ayant un seul objet, l'unité de cet objet comme seule $1$\nbd-flèche et $\mathbb N$ comme monoïde d'endomorphismes de cette unité, engendré par une $2$\nbd-cellule $x$, et $f:S^2\to B$ l'unique $2$\nbd-foncteur de source la sphère $S^2$ (cf.~paragraphe~\ref{ooCat}) qui envoie les deux $2$\nbd-cellules non triviales de $S^2$ sur $x$. Une vérification immédiate montre que pour tout système local $M$ sur $B$ (qui dans ce cas est simplement un groupe abélien) le morphisme $\C(f,M):\C(S^2,f^*(M))\to\C(B,M)$ est un quasi-isomorphisme, ce qui implique, puisque $B$ et~$S^2$ sont des polygraphes, que $f$ est une équivalence polygraphique. Or, le type d'homotopie (du nerf) de $B$ est un $K(\mathbb Z,2)$~\cite[Example~4.10]{AraThB}, tandis que $S^2$ a le type d'homotopie d'une sphère de dimension $2$~\cite[6.5.4]{LeonardTh}, et par suite, $f$ \emph{n'est pas} une équivalence de Thomason.
\smallbreak

D'autre part, le morphisme d'adjonction $g:c\,\mathrm{Sd}^2\mathrm{Ex}^2N(B)\to B$, où $\mathrm{Sd}$ désigne le foncteur de subdivision et $\mathrm{Ex}$ son adjoint à droite, est une équivalence de Thomason~\cite[corollaire~6.32]{DG},~\cite[corollaire~3.9]{Ara2Thom}. Or, en vertu du théorème~\ref{thm:compc}, l'homologie polygraphique de $c\,\mathrm{Sd}^2\mathrm{Ex}^2N(B)$ est isomorphe à l'homologie de l'ensemble simplicial $\mathrm{Sd}^2\mathrm{Ex}^2N(B)$, qui est isomorphe à celle de $N(B)$ (puisque $\mathrm{Sd}$ et $\mathrm{Ex}$ respectent les types d'homotopie), qui est un $K(\mathbb Z,2)$. On en déduit que pour tout entier positif pair~$n$, $\Hp{n}(c\,\mathrm{Sd}^2\mathrm{Ex}^2N(B),\mathbb Z)=\mathbb Z\,$. En revanche, comme $B$ est un $2$\nbd-polygraphe, pour tout $n>2$, $\Hp{n}(B,\mathbb Z)=0$, et par suite, $g$ \emph{n'est pas} une équivalence polygraphique. (Voir aussi~\cite[Proposition~4.5.3]{LeonardTh}.)
\end{sch}

Les équivalences polygraphiques de $\ooCat$ satisfont à un cas particulier du théorème~A de Quillen $\infty$\nbd-catégorique~\cite{DGAI},~\cite{DGAII}, le même que celui satisfait par les équivalences folk~\cite[Proposition~9.2]{Leonard}. Pour le montrer, commençons par le lemme suivant:

\begin{lemme}\label{lem:thA}
Soit
$$
\xymatrixcolsep{1pc}
\xymatrix{
X_1\ar[rr]^f\ar[dr]_{p^{}_1}
&&X_2\ar[dl]^{p^{}_2}
\\
&X
}$$
un triangle commutatif dans $\ooCat$. On suppose que $X$ est une \emph{(1-)}catégorie, et que pour tout objet $x$ de $X$ et tout \SLF{} \emph{(resp.} \SL\emph{)} $M_x$ sur $\mc{x}{X_2}$ le morphisme $\HT(\mc{x}{f},M_x)$ de la catégorie dérivée des groupes abéliens est un isomorphisme. Alors pour tout \SLF{} \emph{(resp.} \SL\emph{)} $M$ sur $X_2$ le morphisme $\HT(f,M)$ est un isomorphisme.
\end{lemme}

\begin{proof}[Esquisse de preuve]
Soit donc $M$ un \SLF{} (resp. \SL) sur~$X_2$ et considérons le foncteur
$$
F=F_{X_2,M}:\cm{\ooCat}{X_2}\to\Comp(\Ab)\,,\qquad(X',p:X'\to X_2)\mapsto\C(X',p^*(M))\,,
$$
qui est un foncteur de Quillen à gauche (théorème~\ref{th:Quilleng}), de sorte qu'on a un carré commutatif à isomorphisme près
$$
\xymatrixcolsep{3pc}
\xymatrix{
\Ho(\Homint(X^\circ,\cm{\ooCat}{X_2}))\ar[r]^{\LL F^{X^\circ}_{X_2,M}}\ar[d]_{\shocolim}
&\Ho(\Homint(X^\circ,\Comp(\Ab)))\ar[d]^{\shocolim}
\\
\Ho(\cm{\ooCat}{X_2})\ar[r]_{\LL F_{X_2,M}}
&\Ho(\Comp(\Ab))\kern20pt,\kern-20pt
}$$
où $\LL F^{X^\circ}_{X_2,M}$ désigne le foncteur dérivé à gauche du foncteur
$$
\xymatrixcolsep{3pc}
\xymatrix{
\Homint(X^\circ,\cm{\ooCat}{X_2})\ar[r]^{\F^{X^\circ}_{X_2,M}}
&\Homint(X^\circ,\Comp(\Ab))
}$$
induit par $F_{X_2,M}$.
L'hypothèse du lemme implique que l'image par $\LL F^{X^\circ}_{X_2,M}$ du morphisme $\mc{?}{X_1}\to\mc{?}{X_2}$ de $\Homint(X^\circ,\cm{\ooCat}{X_2})$, induit par $f$, est un isomorphisme de $\Ho(\Homint(X^\circ,\Comp(\Ab)))$. Or, comme $X$ est une ($1$\nbd-)catégorie, en vertu du dual du théorème~7.10 de~\cite{Leonard} et du lemme~\ref{lemme:abstr}, l'image par la flèche verticale de gauche de ce morphisme s'identifie à $f$, vu comme morphisme au-dessus de $X_2$. On en déduit que l'image de ce dernier par $\LL F_{X_2,M}$ est un isomorphisme, ce qui prouve le lemme (cf.~remarque~\ref{retenir}). 
\end{proof}

\begin{prop}\label{prop:thA}
Soit
$$
\xymatrixcolsep{1pc}
\xymatrix{
X_1\ar[rr]^f\ar[dr]_{p^{}_1}
&&X_2\ar[dl]^{p^{}_2}
\\
&X
}$$
un triangle commutatif dans $\ooCat$. On suppose que $X$ est une \emph{(1-)}catégorie. Alors si pour tout objet $x$ de $X$ le $\infty$\nbd-foncteur $\mc{x}{f}:\mc{x}{X_1}\to\mc{x}{X_2}$ est une équivalence polygraphique, il en est de même de $f$.
\end{prop}

\begin{proof}
En vertu du dual du lemme 7.5 de \cite{Leonard}, $X$ étant une (1-)catégorie, pour tout $\infty$\nbd-foncteur $p:X'\to X$, on a un isomorphisme $\varinjlim_{x\in X}\mc{x}{X'}\simeq X'$ naturel en $X$. Comme le foncteur $\Pi_1$ commute aux limites inductives, l'hypothèse que pour tout objet $x$ de $X$ le morphisme $\Pi_1(\mc{x}{f})$ est un isomorphisme, implique donc qu'il en est de même de $\Pi_1(X)$. La proposition résulte donc du lemme~\ref{lem:thA}. 
\end{proof}

\begin{rem}
En vertu du théorème~7.7 de~\cite{{DGAII}}, la proposition précédente (en tenant compte des propositions~\ref{prop:2sur3} et~\ref{retropl}) implique qu'on a une assertion analogue pour un $2$\nbd-triangle de la forme
\[
    \shorthandoff{;}
    \xymatrix@C=1.pc{
      X_1 \ar[rr]^f \ar[dr]_{p^{}_1}_{}="f" & & X_2 \ar[dl]^(0.42){p^{}_2} \\
      & X
      \ar@{}"f";[ur]_(.15){}="ff"
      \ar@{}"f";[ur]_(.55){}="oo"
      \ar@<-0.5ex>@2"ff";"oo"^{\alpha}
      &,
    }
\]
où $\alpha$ est une transformation oplax (et $X$ toujours une (1-)catégorie).
\end{rem}

\begin{rem}\label{remlocfond}
On rappelle qu'un localisateur fondamental de $\ooCat$~\cite[paragraphe~6.2.1]{AraHDR} est une classe $\W$ de flèches de $\ooCat$ satisfaisant les conditions suivantes:
\begin{itemize}
\item[(a)] $\W$ est faiblement saturée, c'est-à-dire contient les identités, vérifie la propriété du deux sur trois et contient les $\infty$\nbd-foncteurs $j$ admettant une rétraction $r$ telle que $jr$ soit dans $\W$;
\item[(b)] pour toute $\infty$\nbd-catégorie $C$, le $\infty$\nbd-foncteur de \og projection\fg{} $D_1\otimes C\to C$ appartient à $\W$;
\item[(c)] pour tout triangle commutatif de $\infty$\nbd-foncteurs
$$
\xymatrixcolsep{1pc}
\xymatrix{
X_1\ar[rr]^f\ar[dr]_{p^{}_1}
&&X_2\ar[dl]^{p^{}_2}
\\
&X
}$$
si pour tout objet $x$ de $X$ le $\infty$\nbd-foncteur $\mc{x}{f}:\mc{x}{X_1}\to\mc{x}{X_2}$, induit par $f$, est dans $\W$, il en est de même de $f$.
\end{itemize}
\smallbreak

La classe $\Wthom$ des équivalences de Thomason dans $\ooCat$ est un localisateur fondamental de $\ooCat$~\cite[proposition~6.2.3]{AraHDR}, et on conjecture qu'il s'agit du plus petit. En vertu des propositions~\ref{prop:2sur3} et~\ref{satfaible}, et du corollaire~\ref{projpol}, la classe $\Wpol$ des équivalences polygraphiques satisfait aux conditions (a) et (b) ci-dessus. La proposition~\ref{prop:thA} montre qu'elle satisfait également la condition (c) dans le cas particulier où $X$ est une (1\nbd-)catégorie (on dira que $\Wpol$ forme un \emph{$1$\nbd-localisateur fondamental} de $\ooCat$). En revanche, il ne semble pas que $\Wpol$ soit un localisateur fondamental, puisque en vertu du scholie~\ref{scholie}, cela contredirait la conjecture de minimalité de $\Wthom$.
\end{rem}

\appendix

\section{Une adjonction}\label{adj}

\begin{paragr}
Dans ce qui suit, on fixe une $\infty$\nbd-catégorie $X$ et un \SLF{} $M$ sur $X$. On va montrer que le foncteur
$$
\F=\F_{X,M}:\cm{\ooCat}{X}\to\Comp(\Ab)
$$ 
associant à une $\infty$\nbd-catégorie $(X',p:X'\to X)$ au-dessus de $X$ le complexe de groupes abéliens $\C(X',p^*(M))$ admet un adjoint à droite. Commençons par définir un foncteur dans l'autre sens $\G=\G_{X,M}:\Comp(\Ab)\to\cm{\ooCat}{X}$. Soit $N$ un complexe de groupes abéliens 
$$
\xymatrix{
\dots\ar[r]^-{d_{n+1}}
&N_n\ar[r]^-{d_{n}}
&N_{n-1}\ar[r]^-{d_{n-1}}
&\dots\ar[r]^-{d_{2}}
&N_1\ar[r]^-{d_{1}}
&N_0
}\ .
$$
On définit une $\infty$\nbd-catégorie $\G(N)$ comme suit:
\smallbreak

--- Pour $n\geq0$, une $n$\nbd-cellule de $\G(N)$ est un couple $(x,U)$, $x\in X_n$ et
$$
U=\begin{pmatrix}
u^0_0 &u^0_1 &\cdots &u^0_{n-1}
\\
\noalign{\vskip -5pt}
&&&&u^{}_n
\\
\noalign{\vskip -5pt}
u^1_0 &u^1_1 &\cdots &u^1_{n-1}
\end{pmatrix},
$$
où, en posant $u^0_n=u^1_n=u^{}_n$,
$$
u^0_i:M_{t_0(s_ix)}\to N_i\,,\quad u^1_i:M_{t_0(t_ix)}\to N_i\,,\qquad 0\leq i\leq n\,,
$$
sont des morphismes de groupes abéliens satisfaisant aux conditions:
$$\begin{aligned}
&d_i\circ u^\e_i=u^1_{i-1}-u^0_{i-1}\,,\qquad\ 1<i\leq n\,,\cr
&d_1\circ u^\e_1=u^1_0-u^0_0\circ t_1(x)^*\,.
\end{aligned}$$
Ces formules ont un sens car on observe qu'en vertu des relations globulaires, pour $0\leq i\leq n$, $\e=0,1$, et $(i,\e)\neq(0,0)$, tous les morphismes $u^\e_i$ ont même source $M_{t_0x}\,$, la source de $u^0_0$ étant $M_{s^{}_0x}\,$.

--- Les sources et buts sont définis par
$$\begin{aligned}
s_{n-1}(x,U)=(s_{n-1}x,s_{n-1}U)\,,\qquad\kern -3pt
s_{n-1}U&=\begin{pmatrix}
u^0_0 &\cdots &u^0_{n-2}
\\
\noalign{\vskip -5pt}
&&&u^0_{n-1}
\\
\noalign{\vskip -5pt}
u^1_0 &\cdots &u^1_{n-2}
\end{pmatrix},\cr
\noalign{\vskip 3pt}
t_{n-1}(x,U)=(t_{n-1}x,t_{n-1}U)\,,\qquad
t_{n-1}U&=\begin{pmatrix}
u^0_0 &\cdots &u^0_{n-2}
\\
\noalign{\vskip -5pt}
&&&u^1_{n-1}
\\
\noalign{\vskip -5pt}
u^1_0 &\cdots &u^1_{n-2}
\end{pmatrix}.\cr
\end{aligned}$$

--- Les unités sont définies par
$$
1_{(x,\,U)}=(1_x,1_U)\,,\qquad\qquad\kern 10pt
1_U=
\begin{pmatrix}
u^0_0 &\cdots &u^0_{n-1} &u_n
\\
\noalign{\vskip -5pt}
&&&&0
\\
\noalign{\vskip -5pt}
u^1_0 &\cdots &u^1_{n-1} &u_n
\end{pmatrix}.
$$
\smallbreak

--- Les compositions entre cellules composables sont définies par 
$$
(y,V)*_i(x,U)=(y*_ix,V*_iU)\,,
$$
où, pour $0<i<n$,
$$
V*_iU=
\begin{pmatrix}
u^0_0 &\cdots &u^0_i &v^0_{i+1}+u^0_{i+1} &\cdots &v^0_{n-1}+u^0_{n-1}
\\
\noalign{\vskip -5pt}
&&&&&&v_{n}+u_{n}
\\
\noalign{\vskip -5pt}
v^1_0 &\cdots &v^1_i &v^1_{i+1}+u^1_{i+1} &\cdots &v^1_{n-1}+u^1_{n-1}
\end{pmatrix},
$$
et, pour $i=0$,
$$
V*_0U=
\begin{pmatrix}
u^0_0 &v^0_{1}+u^0_{1}\circ t_1(y)^* &\cdots &v^0_{n-1}+u^0_{n-1}\circ t_1(y)^*
\\
\noalign{\vskip -5pt}
&&&&v_{n}+u_{n}\circ t_1(y)^*
\\
\noalign{\vskip -5pt}
v^1_0 &v^1_{1}+u^1_{1}\circ t_1(y)^* &\cdots &v^1_{n-1}+u^1_{n-1}\circ t_1(y)^*
\end{pmatrix}.
$$
Il est immédiat que ces formules définissent bien des cellules de $\G(N)$, la seule vérification non tautologique étant que pour ce dernier tableau $V*_0U$ on a la compatibilité à la différentielle pour les termes de degré $1$:
$$\begin{aligned}
d_1(v^\e_{1}+u^\e_{1}\circ t_1(y)^*)&=v^1_0-v^0_0\circ t_1(y)^*+u^1_{0}\circ t_1(y)^*-u^0_{0}\circ t_1(y)^*\circ t_1(x)^*\cr
&=v^1_0-u^0_{0}\circ t_1(x*_0y)^*
\end{aligned}$$
(la dernière égalité venant du fait que comme $(x,U)$ et $(y,V)$ sont $0$\nbd-composables, on a $u^1_0=v^0_0$). La vérification des axiomes des $\infty$\nbd-catégories est fastidieuse, mais directe. La projection $(x,U)\mapsto x$ définit un $\infty$-foncteur $\G(N)\to X$ et on obtient ainsi un foncteur $\G:\Comp(\Ab)\to\cm{\ooCat}{X}$.
\smallbreak

On définit un morphisme de foncteurs $\Eta:1_{\cm{\ooCat}{X}}\to\G\F$ comme suit. Soient $(X',\,p:X'\to X)$ une $\infty$\nbd-catégorie au-dessus de $X$ et $N$ le complexe $N=\F(X',p)=\C(X',p^*(M))$. Il s'agit de définir un $\infty$\nbd-foncteur $\Eta_{(X'\kern -1.5pt,\kern 1ptp)}:X'\to\G(N)$ au-dessus de $X$, naturel en $(X'\kern-1pt,p)$. Soient $n\geq0$ et $x'$ une $n$\nbd-cellule de $X'$. On pose 
$$
\Eta_{(X'\kern -1.5pt,\kern1ptp)}(x')=(p(x'),U_{x'})\,,
$$ 
où
$$
U_{x'}=\begin{pmatrix}
u^0_0{}_{,\kern.8ptx'} &u^0_1{}_{,\kern.8ptx'} &\cdots &u^0_{n-1}{}_{,\kern.8ptx'}
\\
\noalign{\vskip -5pt}
&&&&u^{}_{n,\kern.8ptx'}
\\
\noalign{\vskip -5pt}
u^1_0{}_{,\kern.8ptx'} &u^1_1{}_{,\kern.8ptx'} &\cdots &u^1_{n-1}{}_{,\kern.8ptx'}
\end{pmatrix},
$$
le morphisme $u^0_{i,x'}:M_{t_0(s_ip(x'))}\to N_i$ (resp.~$u^1_{i,x'}:M_{t_0(t_ip(x'))}\to N_i$), pour $0\leq i\leq n$, étant le composé
$$
M_{p(t_0y')}\hookrightarrow\textstyle\bigoplus\limits_{y'\in X'_i}\kern -5pt M_{p(t_0y')}\to\C_i(X',p^*(M))=N_i
$$
de l'inclusion canonique correspondant à $y'=s_i(x')$ (resp.~à~$y'=t_i(x')$), suivie de la surjection canonique (en remarquant que, pour $i=n$, on a $u^0_{n,x'}=u^1_{n,x'}$ et en posant $u^{}_{n,x'}=u^0_{n,x'}=u^1_{n,x'}$). On laisse au lecteur le soin de vérifier qu'on obtient bien ainsi un $\infty$\nbd-foncteur, $\Eta_{(X'\kern -1.5pt,\kern 1ptp)}:X'\to\G(N)$ naturel en $(X'\kern-1pt,p)$.
\smallbreak

De même, on définit un morphisme de foncteurs $\Eps:\F\G\to1_{\Comp(\Ab)}$ comme suit. Soient $N$ un complexe de groupes abéliens et $(X',p:X'\to X)=\G(N)$. Il s'agit de définir un morphisme de complexes $\Eps^{}_N:\C(X',p^*(M))\to N$, naturel en $N$. Soit $n\geq0$. On définit un morphisme de groupes abéliens $\Eps^{}_{N,n}:\C_n(X',p^*(M))\to N_n$ en posant, pour $(x,U,m)$ un générateur de $\C_n(X',p^*(M))$, autrement dit, $x\in X_n$,
$$
U=\begin{pmatrix}
u^0_0 &u^0_1 &\cdots &u^0_{n-1}
\\
\noalign{\vskip -5pt}
&&&&u^{}_n
\\
\noalign{\vskip -5pt}
u^1_0 &u^1_1 &\cdots &u^1_{n-1}
\end{pmatrix},
$$
$$
u^0_i:M_{t_0(s_ix)}\to N_i\,,\quad u^1_i:M_{t_0(t_ix)}\to N_i\,,\qquad 0\leq i\leq n\,,
$$
(toujours avec la convention $u^0_n=u^1_n=u_n$), et $m\in M_{t_0x}$,
$$
\Eps^{}_{N,n}(x,U,m)=u^{}_n(m).
$$
La compatibilité aux relations $(*_i)$, $0\leq i<n$, et aux différentielles, ainsi que la naturalité en $N$ sont tautologiques.
\end{paragr}

\begin{thm}
Le couple des foncteurs
$$
\F:\cm{\ooCat}{X}\to\Comp(\Ab)\,,\quad\G:\Comp(\Ab)\to\cm{\ooCat}{X}
$$
est un couple de foncteurs adjoints.
\end{thm}

\begin{proof}
Une vérification facile laissée au lecteur montre que les morphismes de foncteurs $\Eps$ et $\Eta$ satisfont aux égalités triangulaires.
\end{proof}

\section{Homologie et version abélienne de la construction de Grothendieck \hbox{\pdfoo-catégorique}}

On rappelle que le foncteur d'oubli $\Ab\to\Ens$ de la catégorie des groupes abéliens vers celle des ensembles admet un adjoint à gauche $\AB$, le \emph{foncteur d'abélianisation} associant à un ensemble le groupe abélien libre engendré par cet ensemble. De même, le foncteur d'oubli $\ooCat(\Ab)\to\ooCat$ des objets $\infty$\nbd-catégories dans $\Ab$ vers celle des $\infty$-catégories admet un foncteur adjoint à gauche noté aussi $\AB$ et appelé \emph{foncteur d'abélianisation}. On rappelle également que la catégorie $\ooCat(\Ab)$ est canoniquement équivalente à celle des complexes de chaînes de groupes abéliens $\Comp(\Ab)$~\cite{Bourn}. Modulo cette identification le foncteur d'abélianisation n'est autre que le foncteur
$$\ooCat\to\Comp(\Ab)\,,\qquad X\mapsto\C(X,\mathbb Z)\,$$
(cf. paragraphe~\ref{relfond}), où $\mathbb Z$ désigne le \SL{} constant de valeur $\mathbb Z$~\cite{Steiner},~\cite{Leonard}.
\smallbreak

Il n'est pas vrai que, pour un \SLF{} ou même un \SL{} $M$ arbitraire sur une $\infty$\nbd-catégorie $X$, le complexe $\C(X,M)$ soit toujours obtenu comme l'abélianisation d'une $\infty$\nbd-catégorie. En effet, le groupe abélien $\C_0(X,M)$ n'est pas en général un $\mathbb Z$\nbd-module libre. Néanmoins, cela est vrai dans le cas particulier important où le $\infty$\nbd-foncteur $M:X^\circ\to\Ab$ se factorise par le foncteur d'abélianisation $\AB:\Ens\to\Ab$.
\smallbreak

En effet, soit $\E:X^\circ\to\Ens$ un $\infty$\nbd-foncteur. 
Cela revient à se donner pour tout objet $x$ de $X$, un ensemble $\E_x$ et pour toute $1$\nbd-flèche $x:x_0\to x_1$ de $X$, une application $x^*=E_x:E_{x_1}\to E_{x_0}$ tels que pour toute $2$\nbd-flèche $x:x_0\Rightarrow x_1$
de $X$, on ait $x^*_0=x^*_1$, pour tout couple de $1$\nbd-flèches composables $x_0,x_1$ de $X$, on ait $(x_1*_0x_0)^*=x_0^*x_1^*$, et pour tout objet $x$ de $X$, on ait $1^*_x=1_{E_x}$. On définit une $\infty$\nbd-catégorie $\int\E$ comme suit: Pour $n\geq0$, une $n$\nbd-cellule de $\int\E$ est un couple $(x,a)$, avec $x\in X_n$ et $a\in \E_{t_0x}$. Pour $n>0$, les buts et sources d'une telle cellule sont définis par les formules:
$$
t_{n-1}(x,a)=(t_{n-1}(x),a)\,,\qquad s_{n-1}(x,a)=\left\{\begin{aligned}
&(s_{n-1}(x),a)\,,\quad\kern 5.9pt n>1\,,\cr
&(s_{0}(x),x^*(a))\,,\quad n=1\,,
\end{aligned}
\right.
$$
et pour $n\geq0$, l'unité est définie par la formule $1_{(x,a)}=(1_x,a)$. 
Si $(x_0,a_0)$ et $(x_1,a_1)$ sont deux $n$\nbd-cellules $i$\nbd-composables, pour $0\leq i<n$, leur composé est donné par
$$
(x_1,a_1)*_i(x_0,a_0)=(x_1*_ix_0,a_1)\,,
$$
où on remarque que si $i>0$, alors $a_0=a_1$, et si $i=0$, on a $a_0=t_1(x_1)^*(a_1)$.
Si $M$ désigne le \SLF{} composé
$$\xymatrix{
X^\circ\ar[r]^-{\E}
&\Ens\ar[r]^-{\AB}
&\Ab
}$$
du foncteur $\E$ suivi du foncteur d'abélianisation, on vérifie aussitôt que le complexe $\C(X,M)$ est l'abélianisation de la $\infty$\nbd-catégorie $\Int \E$.
\smallbreak

La formation de $\int\E$ est un cas très dégénéré de la construction de Grothendieck $\infty$\nbd-catégorique~\cite{Warren},~\cite{AGG}, où le $\infty$\nbd-foncteur $X\to\ooCat$ est à valeurs des $\infty$\nbd-catégories discrètes, c'est-à-dire des ensembles. 
Ainsi, la formation du complexe $\C(X,M)$, pour un \SLF{} $M$ général, peut être vue comme une variante abélienne de cette construction, où les groupes abéliens $M_x$ sont considérés comme objets $\infty$\nbd-catégories \emph{discrètes} de $\Ab$, autrement dit, comme complexes réduits en degré $0$.   

\section{Homologie polygraphique d'une \pdfoo-catégorie faible}

On ne rappellera pas ici la définition des cohérateurs. On renvoie le lecteur à~\cite{MalCat},~\cite{MalGrCat},~\cite{Ara}.
\smallbreak

Fixons un cohérateur de $\infty$\nbd-catégories $C$. On rappelle que les objets de $C$ sont les sommes globulaires de la forme
$$
D_{i_1}\amalg_{D_{i'_1}}D_{i_2}\amalg_{D_{i'_2}}\cdots\amalg_{D_{i'_{n-1}}}D_{i_n}
$$
et qu'une $\infty$\nbd-$C$\nbd-catégorie est un préfaisceau sur $C$ transformant les sommes globulaires en produits fibrés. Les préfaisceaux représentables, et en particulier les disques $D_n$, $n\geq0$, sont des $\infty$\nbd-$C$\nbd-catégories. La catégorie $\ooCCat$ des $\infty$\nbd-$C$\nbd-catégories est la sous-catégorie pleine de $\pref C$ formée des $\infty$\nbd-$C$\nbd-catégories. On a donc des inclusions
$$
C\hookrightarrow\ooCCat\hookrightarrow\pref C\,.
$$
Pour $n>0$, la sphère $S^{n-1}$ est le sous-préfaisceau de $D_n$ dont les sections se factorisent par la cosource ou le cobut de $D_n$, et $S^{-1}$ est le préfaisceau vide sur $C$. Les polygraphes faibles se définissent de façon parfaitement analogue aux polygraphes stricts. Une résolution polygraphique d'une $\infty$\nbd-$C$\nbd-catégorie $X$ est un morphisme $p:P\to X$ de $\ooCCat$ de source un polygraphe faible ayant la propriété de relèvement à droite relativement aux inclusions $S^{n-1}\hookrightarrow D_n$. L'argument du petit objet montre que toute $\infty$\nbd-$C$\nbd-catégorie admet une résolution polygraphique.
\smallbreak

Dimitri Ara a construit dans~\cite{Ara}, pour tout cohérateur de $\infty$\nbd-catégories $C$, un foncteur canonique $u:C\to\Theta$ de but la catégorie $\Theta$ de Joyal~\cite{Joy},~\cite{MZ},~\cite{CB2},~\cite{CiMal} induisant l'identité sur les objets. Ce foncteur permet d'associer à toute $\infty$\nbd-catégorie stricte $X$ (vue comme un préfaisceau sur $\Theta$ transformant les sommes globulaires en produits fibrés) une $\infty$\nbd-$C$\nbd-catégorie $u^*(X)=X\circ u^\circ$, définissant ainsi un foncteur pleinement fidèle $u^*:\ooCat\to\ooCCat$ pouvant être considéré comme une inclusion. Ce foncteur admet un adjoint à gauche $u_!:\ooCCat\to\ooCat$, le \emph{foncteur de strictification}. 
\smallbreak

Le groupoïde fondamental d'une $\infty$\nbd-$C$\nbd-catégorie $X$ peut être défini comme pour les $\infty$\nbd-catégorie strictes, mais une définition équivalente plus rapide consiste à poser $\Pi_1(X):=\Pi_1(u_!(X))$. Un \emph{\SL} sur $X$ est un foncteur contravariant $M:\Pi_1(X)^\circ\to\Ab$, ou de façon équivalente un \SL{} sur la $\infty$\nbd-catégorie stricte $u_!(X)$. De même, la catégorie fondamentale d'une $\infty$\nbd-$C$\nbd-catégorie $X$ peut être définie directement, ou plus rapidement comme étant celle de $u_!(X)$, et un \emph{\SLF} sur $X$ comme étant un \SLF{} sur $u_!(X)$.
\smallbreak

Soient $X$ une $\infty$\nbd-$C$\nbd-catégorie et $M$ un \SLF. On définit l'homologie polygraphique de $X$ à coefficients dans $M$, en choisissant une résolution polygraphique $p:P\to X$ de $X$ dans $\ooCCat$ et en posant dans la catégorie dérivée des groupes abéliens 
$$
\HP(X,M):=\HP(u_!(P),p^*(M))
\,.$$
Autrement dit, l'homologie polygraphique de la $\infty$\nbd-$C$\nbd-catégorie $X$ à coefficients dans le \SLF{} $M$ est l'homologie polygraphique de la strictification d'une résolution polygraphique faible de $M$. Puisque le foncteur $u_!$ commute aux limites inductives, cette strictification est un polygraphe (strict), et son homologie polygraphique est donc définie par le complexe $\C(u_!(P),p^*(M))$.

\begin{conj}
L'homologie polygraphique $\HP(X,M)$ est indépendante, à isomorphisme canonique près, de la résolution choisie.
\end{conj}

Soit maintenant $X$ une $\infty$\nbd-catégorie stricte et $M$ un \SLF{} sur $X$. L'homologie polygraphique de $X$, vue comme $\infty$\nbd-catégorie faible, autrement dit celle de $u^*(X)$, est donc définie en choisissant une résolution polygraphique $p:P\to u^*(X)$ de $u^*(X)$ dans $\ooCCat$ et en posant $\HP(u^*(X),M)=\HP(u_!(P),p^*(M))$. Le composé
$$
\xymatrix{
u_!(P)\ar[r]^-{u_!(p)}
&u_!u^*(X)\ar[r]
&X
}$$
de $u_!(p)$ suivi du foncteur d'adjonction définit un morphisme canonique dans la catégorie dérivée des groupes abéliens $\HP(u^*(X),M)\to\HP(X,M)$.

\begin{conj}
On a un isomorphisme canonique dans la catégorie dérivée des groupes abéliens $\HP(u^*(X),M)\simeq\HT(X,M)$ identifiant le morphisme canonique ci-dessus au morphisme de comparaison $\HT(X,M)\to\HP(X,M)$ de la section~\emph{\ref{Leonard2}}.
\end{conj}

En d'autres termes, l'homologie polygraphique d'une $\infty$\nbd-catégorie stricte vue comme $\infty$\nbd-catégorie faible est isomorphe à sa \og vraie\fg{}
homologie.
\smallbreak

On définit la classe des équivalences polygraphiques de $\ooCCat$ exactement comme celles de $\ooCat$: un morphisme $f:X'\to X$ de $\ooCCat$ est une \emph{équivalence polygraphique} si $f$ induit une équivalence des groupoïdes fondamentaux et si pour tout système local $M$ sur $X$ le morphisme
$$\HP(f,M):\HP(X', f^*(M))\to\HP(X,M)$$
est un isomorphisme dans la catégorie dérivée des groupes abéliens.
\smallbreak

On dit qu'un morphisme  $f:X'\to X$ de $\ooCCat$ est une \emph{équivalence de Thomason} si le foncteur $\cm{C}{X'}\to\cm{C}{X}$, induit par $f$, entre les catégories des éléments de $X'$ et~$X$, vues comme préfaisceaux sur $C$, est une équivalence de Thomason de $\Cat$. Cette définition est raisonnable car conjecturalement tout cohérateur est une catégorie test, de sorte que $\pref C$ modélise les types d'homotopie et l'inclusion $\ooCCat\hookrightarrow\pref C$ peut être vue comme un foncteur nerf, analogue au nerf cellulaire pour les $\infty$\nbd-catégories strictes.

\begin{conj}
La classe des équivalences polygraphiques de $\ooCCat$ coïncide avec la classe des équivalences de Thomason.
\end{conj}

\backmatter

\providecommand{\bysame}{\leavevmode ---\ }
\providecommand{\og}{``}
\providecommand{\fg}{''}
\providecommand{\smfandname}{\&}
\providecommand{\smfedsname}{\'eds.}
\providecommand{\smfedname}{\'ed.}
\providecommand{\smfmastersthesisname}{M\'emoire}
\providecommand{\smfphdthesisname}{Th\`ese}

\end{document}